\documentclass[12pt,reqno,a4paper]{amsart}
\usepackage{fullpage}
\usepackage[utf8]{inputenc}
\usepackage[T1]{fontenc}
\usepackage{hyperref}
\usepackage{enumitem}
\usepackage{amsmath,amssymb,amsthm,color}
\usepackage{todonotes,framed}
\usepackage{nameref}
\usepackage{enumitem}
\usepackage{moreverb}
\usepackage{booktabs}
\usepackage{bm}
\usepackage[normalem]{ulem} 
\usepackage{cancel} 

\title{Convergence and quasi-optimal cost of adaptive algorithms for nonlinear operators including iterative linearization and algebraic solver}
\author{Alexander Haberl}
\author{Dirk Praetorius}
\author{Stefan Schimanko}
\author{Martin Vohral\'{i}k}

\address{TU Wien, Institute for Analysis and Scientific Computing, Wiedner Hauptstr. 8-10/E101/4, 1040 Vienna, Austria}
\email{alexander.haberl@asc.tuwien.ac.at}
\email{dirk.praetorius@asc.tuwien.ac.at \quad \rm(corresponding author)}
\email{stefan.schimanko@asc.tuwien.ac.at}

\address{Inria, 2 rue Simone Iff, 75589 Paris, France \& Universit\'e Paris-Est, CERMICS (ENPC), 77455 Marne-la-Vall\'ee, France}
\email{martin.vohralik@inria.fr}

\thanks{{\bf Acknowledgement.} The authors thankfully acknowledge support by the Austrian Science Fund (FWF) through the research projects \emph{Computational nonlinear PDEs} (grant P33216), \emph{Optimal isogeometric boundary element method} (grant P29096), as well as \emph{Taming complexity in partial differential systems} (grant SFB F65). This project has also received funding from the European Research Council (ERC) under the European Union’s Horizon 2020 research and innovation program (grant agreement No 647134 GATIPOR)}

\makeatother

\def\R{\mathbb{R}}
\def\UU{\mathcal{U}}
\def\lpic{\lambda_{\rm Pic}}
\def\qpic{q_{\rm Pic}}
\def\lpcg{\lambda_{\rm alg}}

\def\qpcg{q_{\rm alg}}

\def\qsolve{q_{\rm alg}}

\def\Cpic{C_{\rm Pic}}
\def\Cpcg{C_{\rm alg}}

\def\Csum{C_{\rm sum}}

\def\Revision#1{{\color{blue}#1}}
\def\Revision#1{}

\def\set#1#2{\big\{#1 \,:\, #2 \big\}}

\def\Alpha{{\rm A}}

\usepackage{fancyhdr}
\advance\footskip0.4cm
\advance\textheight-0.4cm
\calclayout
\newcommand*\patchAmsMathEnvironmentForLineno[1]{%
  \expandafter\let\csname old#1\expandafter\endcsname\csname #1\endcsname
  \expandafter\let\csname oldend#1\expandafter\endcsname\csname end#1\endcsname
  \renewenvironment{#1}%
     {\linenomath\csname old#1\endcsname}%
     {\csname oldend#1\endcsname\endlinenomath}}%
\newcommand*\patchBothAmsMathEnvironmentsForLineno[1]{%
  \patchAmsMathEnvironmentForLineno{#1}%
  \patchAmsMathEnvironmentForLineno{#1*}}%
\AtBeginDocument{%
\patchBothAmsMathEnvironmentsForLineno{equation}%
\patchBothAmsMathEnvironmentsForLineno{align}%
\patchBothAmsMathEnvironmentsForLineno{flalign}%
\patchBothAmsMathEnvironmentsForLineno{alignat}%
\patchBothAmsMathEnvironmentsForLineno{gather}%
\patchBothAmsMathEnvironmentsForLineno{multline}%
}
\usepackage[mathlines]{lineno}

\makeatletter
\def\@seccntformat#1{\hspace*{4mm}%
  \protect\textup{\protect\@secnumfont
    \ifnum\pdfstrcmp{subsection}{#1}=0 \bfseries\fi
    \csname the#1\endcsname
    \protect\@secnumpunct
  }%
}
\makeatother

\def\coarse{H}
\def\fine{h}

\def\N{\mathbb{N}}
\def\T{\mathbb{T}}
\def\MM{\mathcal{M}}
\def\TT{\mathcal{T}}
\def\XX{\mathcal{X}}
\def\QQ{\mathcal{Q}}
\def\UU{\mathcal{V}}
\def\AA{\mathcal{A}}
\def\HH{\mathcal{X}}

\def\EE{\mathcal{E}}
\def\K{\mathbb{K}}
\def\C{\mathbb{C}}
\def\RR{\mathcal{R}}

\def\j{{\underline{j}}}
\def\k{{\underline{k}}}

\def\R{\mathbb{R}}
\def\d#1{\,{\rm d}#1}
\def\div{{\rm div}\,}

\def\lpic{\lambda_{\rm Pic}}
\def\lpcg{\lambda_{\rm alg}}

\def\Crel{C_{\rm rel}}
\def\Cstab{C_{\rm stab}}
\def\qred{q_{\rm red}}
\def\Clin{C_{\rm lin}}
\def\qlin{q_{\rm lin}}
\def\qghps{q_{\rm GHPS}}

\def\Ccea{C_{\text{\rm C\'ea}}}
\def\Cghps{C_{\rm GHPS}}
\def\Cson{C_{\rm son}}
\def\Cmesh{C_{\rm mesh}}

\def\Cmark{C_{\rm mark}}
\def\Copt{C_{\rm opt}}
\def\copt{c_{\rm opt}}
\def\A{\mathbb{A}}
\def\OO{\mathcal{O}}

\def\refine{{\tt refine}}

\def\norm#1#2{\|#1\|_{#2}}
\def\enorm#1{|\!|\!| #1 |\!|\!|}

\def\reff#1#2{\stackrel{\eqref{#1}}{#2}}

\newlength{\leftstackrelawd}
\newlength{\leftstackrelbwd}
\def\reffleft#1#2{\settowidth{\leftstackrelawd}%
{${{}^{\eqref{#1}}}$}\settowidth{\leftstackrelbwd}{$#2$}%
\addtolength{\leftstackrelawd}{-\leftstackrelbwd}%
\leavevmode\ifthenelse{\lengthtest{\leftstackrelawd>0pt}}%
{\kern-.48\leftstackrelawd}{}\mathrel{\mathop{#2}\limits^{\eqref{#1}}}}

\def\dual#1#2{\langle#1\,,\,#2\rangle}
\def\product#1#2{\bm{(}#1\,,\,#2\bm{)}}
\def\Ltproduct#1#2{(#1\,,\,#2)}

\newcounter{statement}
\newenvironment{statement}[2][!]{%
\vskip3mm
\hrule
\hrule
\hrule
\vskip1mm
\noindent%
\refstepcounter{statement}%
\bf#2~\thestatement%
\ifthenelse{\equal{#1}{!}}{.\ }{~(#1).\ }%
\it%
}{%
\vskip1mm
\hrule
\hrule
\hrule
\vskip2mm
}

\newenvironment{theorem}[1][!]{\begin{statement}[#1]{Theorem}}{\end{statement}}
\newenvironment{lemma}[1][!]{\begin{statement}[#1]{Lemma}}{\end{statement}}

\newenvironment{proposition}[1][!]{\begin{statement}[#1]{Proposition}}{\end{statement}}

\newenvironment{remark}[1][!]{\begin{statement}[#1]{Remark}}{\end{statement}}
\newenvironment{algorithm}[1][!]{\begin{statement}[#1]{Algorithm}}{\end{statement}}

\newcommand\ie{i.e.}

\newcommand\eg{e.g.}

\begin{document}

\maketitle
\thispagestyle{fancy}

\begin{abstract}
We consider a second-order elliptic boundary value problem with strongly monotone and Lipschitz-continuous nonlinearity. We design and study its adaptive numerical approximation interconnecting a finite element discretization, the Banach--Picard linearization, and a contractive linear algebraic solver. We in particular identify stopping criteria for the algebraic solver that on the one hand do not request an overly tight tolerance but on the other hand are sufficient for the inexact (perturbed) Banach--Picard linearization to remain contractive. Similarly, we identify suitable stopping criteria for the Banach--Picard iteration that leave an amount of linearization error that is not harmful for the residual {\sl a~posteriori} error estimate to steer reliably the adaptive mesh-refinement. For the resulting algorithm, we prove a contraction of the (doubly) inexact iterates after some amount of steps of mesh-refinement/linerization/algebraic solver, leading to its linear convergence. Moreover, for usual mesh-refinement rules, we also prove that the overall error decays at the optimal rate with respect to the number of elements (degrees of freedom) added with respect to the initial mesh. Finally, we prove that our fully adaptive algorithm drives the overall error down with the same optimal rate also with respect to the overall algorithmic cost expressed as the cumulated sum of the number of mesh elements over all mesh-refinement, linearization, and algebraic solver steps. Numerical experiments support these theoretical findings and illustrate the optimal overall algorithmic cost of the fully adaptive algorithm on several test cases.
\end{abstract}


\def\SS{\mathcal{S}}
\section{Introduction}
\label{section:introduction}

Let $\Omega \subset \R^d$ with $d \ge 1$ be a bounded Lipschitz domain with polytopal boundary.
Given $f \in L^2(\Omega)$, we aim to numerically approximate the weak solution $u^\star \in H^1_0(\Omega)$ of the nonlinear boundary value problem
\begin{align}\label{eq:strongform}
 \begin{split}
  - \div A(\nabla u^\star) &= f \quad \text{in } \Omega,\\
 u^\star &= 0 \quad \text{ on } \partial\Omega.
 \end{split}
\end{align}
To this end, we propose an adaptive algorithm of the type
\begin{align}\begin{split}
 & \hspace*{2.7cm} \boxed{\text{ estimate total error and its components}}\\
 & \hspace*{6.3cm} \downarrow\\
 & \boxed{\text{advance algebra/advance linearization/mark and refine mesh elements}}
\end{split}\end{align}
which monitors and adequately stops the iterative linearization and the linear algebraic solver as well as steers the local mesh-refinement. The goal of this contribution is to perform a first rigorous mathematical analysis of this algorithm in terms of convergence and quasi-optimal computational costs.

\subsection{Finite element approximation and Banach--Picard iteration}

Suppose that the nonlinearity $A$ in~\eqref{eq:strongform} is Lipschitz-continuous (with constant $L > 0$) and strongly monotone (with constant $\alpha > 0$); see Section~\ref{section:main_results} for details. Then, the main theorem on monotone operators yields the existence and uniqueness of the weak solution $u^\star \in H^1_0(\Omega)$; see, e.g., \cite[Theorem~25.B]{zeidler}.
Given a triangulation $\TT_\coarse$ of $\Omega$, the lowest-order finite element approximation to problem~\eqref{eq:strongform} reads as follows: Find $u_\coarse^\star \in \XX_\coarse := \set{v_\coarse \in C(\overline \Omega)}{v_\coarse|_T \text{ is affine for all } T \in \TT_\coarse\textrm{ and } v_\coarse|_{\partial\Omega} = 0}$ such that
\begin{align}\label{eq:intro:discrete}
 \Ltproduct{A(\nabla u_\coarse^\star)}{\nabla v_\coarse}_\Omega
 = \Ltproduct{f}{v_\coarse}_{\Omega}
 \quad \text{for all } v_\coarse \in \XX_\coarse.
\end{align}
The discrete solution $u_\coarse^\star \in \XX_\coarse$ again exists and is unique, but~\eqref{eq:intro:discrete} corresponds to a {\em nonlinear discrete system} which can typically only be solved {\em inexactly}.

The most straightforward algorithm for {\em iterative linearization} of~\eqref{eq:intro:discrete} stems from the proof of the main theorem on monotone operators which is constructive and relies on the Banach fixed point theorem: Define the (nonlinear) operator $\Phi_\coarse : \XX_\coarse \to \XX_\coarse$ by
\begin{align}\label{eq2:intro:phi}
 \Ltproduct{\nabla \Phi_\coarse(w_\coarse)}{\nabla v_\coarse}_\Omega
 = \Ltproduct{\nabla w_\coarse}{\nabla v_\coarse}_\Omega
 - \frac{\alpha}{L^2} \, \big[ \, \Ltproduct{A(\nabla w_\coarse)}{\nabla v_\coarse}_\Omega
 - \Ltproduct{f}{v_\coarse}_\Omega \, \big]
\end{align}
for all $w_\coarse,v_{\coarse} \in \XX_\coarse$. Note that~\eqref{eq2:intro:phi} corresponds to a discrete Poisson problem and hence $\Phi_\coarse(w_\coarse)\in\XX_H$ is well-defined. Then, it holds that
\begin{align}
 \hspace*{-1mm}\norm{\nabla(u_\coarse^\star - \Phi_\coarse(w_\coarse))}{L^2(\Omega)}
 \le \qpic \, \norm{\nabla(u_\coarse^\star - w_\coarse)}{L^2(\Omega)}
 \,\, \text{with} \,\,
 \qpic := (1-\alpha^2/L^2)^{1/2} < 1;
\end{align}
see, \eg,~\cite[Section~25.4]{zeidler}. Based on the {\em contraction} $\Phi_\coarse$, the Banach--Picard iteration starts from an arbitrary discrete initial guess and applies $\Phi_\coarse$ inductively to generate a sequence of discrete functions which hence converge towards $u_\coarse^\star$. Note that the computation of $\Phi_\coarse(w_h)$ by means of the discrete variational formulation~\eqref{eq2:intro:phi} corresponds to the solution of a (generically large) {\em linear discrete system} with symmetric and positive definite matrix that does not change during the iterations. In this work, we suppose that also~\eqref{eq2:intro:phi} is solved {\em inexactly} by means of a {\em contractive iterative algebraic solver} (with contraction factor $\qpcg < 1$), e.g., PCG with optimal preconditioner; see, e.g.,~\cite{Olsh_Tyrt_it_methods_14}.

\subsection{Fully adaptive algorithm} \label{sec_alg_intr}

In our approach, we compute a {\em sequence} of discrete approximations $u_\ell^{k,j}$ of $u^\star$ that have an index $\ell$ for the {\em mesh-refinement}, an index $k$ for the Banach--Picard {\em linearization} iteration, and an index $j$ for the {\em algebraic solver} iteration.

First, we design a stopping criterion for the algebraic solver such that, at linearization step $k-1 \in \N_0$ on the mesh $\TT_\ell$, we stop for some index $\j \in \N$. At the next linearization step $k \in \N$, the arising linear system reads as follows:
\begin{align}\label{eq3:intro:phi}
 \begin{split}
 &\text{Find } u_\ell^{k,\star} \in \XX_\ell \text{ such that, for all $v_\ell \in \XX_\ell$,}
 \\& \quad
 \Ltproduct{\nabla u_\ell^{k,\star}}{\nabla v_\ell}_\Omega
 = \Ltproduct{\nabla u_\ell^{k-1,\j}}{\nabla v_\ell}_\Omega
 - \frac{\alpha}{L^2} \, \big[ \, \Ltproduct{A(\nabla u_\ell^{\!k-1,\j})}{\nabla v_\ell}_\Omega
 - \Ltproduct{f}{v_\ell}_\Omega \, \big],
 \end{split}
\end{align}
with uniquely defined but not computed exact solution $u_\ell^{k,\star} = \Phi_\ell(u_\ell^{k-1,\j})$ and computed iterates $u_\ell^{k,j}$ that approximate $u_\ell^{k,\star}$. Note that~\eqref{eq3:intro:phi} is a {\em perturbed} Banach--Picard iteration since it starts from the available $u_\ell^{k-1,\j}$, typically not equal to the unavailable $u_\ell^{k-1,\star}$.

Second, we design a stopping criterion for the perturbed Banach--Picard iteration at some index $\k$, producing a discrete approximation $u_{\ell}^{\k,\j}$. 

Finally, we locally refine the triangulation $\TT_{\ell}$ on the basis of the D\"orfer marking criterion for the local contributions of the residual error estimator $\eta_{\ell}(u_{\ell}^{\k,\j})$, and, to lower the computational effort, employ nested iteration in that the continuation on the new triangulation $\TT_{\ell +1}$ is started with the initial guess $u_{\ell+1}^{0,0} := u_{\ell}^{\k,\j}$.

\subsection{Previous contributions}

Solving the linear and nonlinear discrete systems ``exactly'' is often not possible in practical situations due to the size of the considered systems, and, actually, performing inexact solves on purpose is a traditional and popular approach to speed-up the simulations. Focusing on the inexact solve of the linear systems gives in particular rise to the ``inexact Newton method''; see, \eg,~\cite{Deuf_glob_inex_Newt_91,Eis_Walk_inex_Newt_94}, and the references therein. Under appropriate conditions, these can asymptotically preserve the convergence speed of the ``exact'' method. Note that these approaches only focus on the finite-dimensional system of nonlinear algebraic equations of the form~\eqref{eq:intro:discrete} but do not see/take into account the continuous problem~\eqref{eq:strongform}.

Taking into account the error from numerical discretization and distinguishing the {\em linearization} and {\em discretization} errors sets a new level of difficulty as, at this moment, one leaves the finite-dimensional world of~\eqref{eq:intro:discrete} and the overall error is evaluated with respect to~\eqref{eq:strongform}. For strongly monotone nonlinear model problems, this has been done in, \eg, ~\cite{Chai_Sur_comp_a_post_NL_06, Chai_Sur_a_post_lin_err_mon_NL_07}; see also the references therein. Later, reliable (actually guaranteed) and efficient (actually robust with respect to the size of the nonlinearity) {\sl a~posteriori} error estimates in such a framework were obtained in~\cite{El_Al_Ern_Voh_a_post_NL_11}.
Therein, adaptive algorithms balancing the estimates of the linearization and discretization error components are proposed and their optimal performance is observed numerically, but no theoretical proofs of convergence and optimality of the arising approximate solutions are given.
Similar ideas and achievements are presented in~\cite{Ber_Dak_Mans_Say_a_post_it_NL_15,Beck_Cap_Luce_st_crit_flux_rec_15,cw2017}, and in~\cite{Hous_Wihl_Newt_hp_18}, where an adaptive choice of the damping parameter in the Newton method is studied in the context of semilinear singularly-perturbed reaction--diffusion problems.

Recently, theoretical analyses of algorithms balancing linearization and discretization components have been undertaken. The works~\cite{gmz2011,Heid_Wih_lin_un_cvg_19} prove convergence of the combined iterative linearization and finite element (Galerkin) discretization, where~\cite{Heid_Wih_lin_un_cvg_19} builds on the unified framework of~\cite{Heid_Wih_lin_un_18} encompassing also Ka\v canov and (damped) Newton linearizations. Moreover,~\cite{banach,banach2} prove linear {\em convergence}, optimal {\em decay rate} in terms of the {\em number of degrees of freedom}, and (almost) optimal decay rate in terms of the {\em overall computational cost} for a fixed-point (Banach--Picard) iterative scheme.
These last references extend concepts from~\cite{Vees_cvg_p_Lap_02,gmz2012,Bel_Die_Kreu_opt_p_Lap_12} in order to take into account inexact linearization solvers, whereas the linear algebraic solver is supposed exact.


Taking into account all {\em algebraic}, {\em linearization}, and {\em discretization} error components is in the heart of the ``adaptive inexact Newton method''; see~\cite{ev2013} and the references therein. Here dedicated stopping criteria are used both for the outer linearization loop and the inner algebraic solver loop, in conjunction with adaptive mesh-refinement. Extensions to more complicated problems are presented in~\cite{Canc_Pop_Voh_a_post_2P_14,Di_Pi_Voh_Yous_a_post_Stef_15,Di_Pi_Voh_Yous_a_post_comp_14}; see also~\cite{Pol_inex_adpt_NL_16} for regularizations on coarse meshes ensuring well-posedness of the discrete systems in Newton-like linearizations. Again reliability (and efficiency) of the estimates are theoretically established and optimal performances of the fully adaptive algorithms are numerically observed, but no theoretical proofs of the latter are presented. Instead, this is our goal in the present work.
We stress that such results have already been derived for adaptive wavelet discretizations~\cite{cdd2003,stevenson2014} which provide inherent control of the residual error in terms of the wavelet coefficients, while the present analysis for standard finite element discretizations has to rely on the local information of appropriate {\sl a~posteriori} error estimators.

\subsection{Main results: linear convergence, optimal decay rate, and optimal cost}

The present contribution appears to be the first work that provides a thorough convergence analysis of fully adaptive strategies for nonlinear equations. To describe more precisely our results, note that the sequential nature of the fully adaptive algorithm of Section~\ref{sec_alg_intr} gives rise to an index set
\begin{align*}
 \QQ := \set{(\ell,k,j) \in \N_0^3}{\text{discrete approximation $u_\ell^{k,j}$ is computed by the algorithm}}
\end{align*}
together with an ordering
\begin{align*}
 |(\ell,k,j)| < |(\ell',k',j')|
 \quad \stackrel{\rm def}{\Longleftrightarrow}
\quad u_\ell^{k,j} \text{ is computed earlier than } u_{\ell'}^{k',j'}.
\end{align*}
Then, our first main result, formulated in Theorem~\ref{theorem:linconv} below, proves that the proposed adaptive strategy is {\em contractive} after some amount of steps and {\em linearly convergent} in the sense of
\begin{align}\label{eq:intro_lin_conv}
 \Delta_{\ell'}^{k',j'}
 \le \Clin \qlin^{|(\ell',k',j')| - |(\ell,k,j)|} \, \Delta_\ell^{k,j}
 \quad \text{for all } 
 |(\ell,k,j)| \leq |(\ell',k',j')|,
\end{align}
where $\Clin\ge1$ and $0 < \qlin < 1$ are generic constants and $\Delta_\ell^{k,j}$ is an appropriate quasi-error quantity involving the error $\norm{\nabla(u^\star - u_\ell^{k,j})}{L^2(\Omega)}$ as well as the error estimator $\eta_\ell(u_\ell^{k,j})$. The estimate~\eqref{eq:intro_lin_conv} appears to be the key argument to prove the {\em optimal error decay rate} with respect to the number of {\em degrees of freedom} added with respect to the initial mesh in the sense that, in particular,
\[
    \sup_{(\ell,k,j) \in \QQ} ( \#\TT_{\ell} - \#\TT_0 + 1 )^s \Delta_{\ell}^{k,j}
		<\infty
\]
whenever $u^\star$ is approximable at rate $s$; see Theorem~\ref{theorem:rate} below for the details. Finally, our most eminent result is the {\em optimal error decay rate} with respect to the {\em overall cost} of the fully adaptive algorithm which steers the mesh-refinement, the perturbed Banach--Picard linearization, and the algebraic solver. In short, this reads
\[
\sup_{(\ell',k',j') \in \QQ} \bigg( \sum_{\substack{(\ell,k,j)\in\QQ \\ (\ell,k,j) \le (\ell',k',j')}} \# \TT_\ell \bigg)^s \Delta_{\ell'}^{k',j'}
				<\infty
\]
whenever $u^\star$ is approximable at rate $s$; see Theorem~\ref{theorem:cost} below for the details.

\subsection{Outline}

The remainder of the paper is organised as follows. In Section~\ref{section:main_results}, we introduce an abstract setting in which all our results will be formulated, define the exact weak and finite elements solutions (none of which is available in our setting), and introduce our requirements on mesh-refinement and error estimator. We also give here precise requirements on the algebraic solver, state our adaptive algorithm and stopping criteria in all details, and present our main results, including some discussions. The proofs of some auxiliary results and of Proposition~\ref{proposition:reliability} (reliability in Algorithm~\ref{algorithm}), Theorem~\ref{theorem:linconv} (linear convergence), Theorem~\ref{theorem:rate} (decay rate wrt.\ degrees of freedom), and Theorem~\ref{theorem:cost} (decay rate wrt.\ computational cost) are respectively given in Sections~\ref{section:aux}, \ref{section:lin_cvg}, \ref{section:optimal_rate_DoFs}, and~\ref{section:optimal_cost}. Finally numerical experiments in Section~\ref{section:num_exp} underline the theoretical findings.

Throughout our work, we apply the following convention: In statements of theorems, lemmas, etc., we explicitly state all constants  together with their dependencies. In proofs, however, we abbreviate $A \le c B$ with a generic constant $c > 0$ by writing $A \lesssim B$. Moreover, $A \simeq B$ abbreviates $A \lesssim B \lesssim A$.

\section{Adaptive algorithm and main results}
\label{section:main_results}

In this section, we introduce an abstract setting, in which all our results will be formulated, define the exact weak and finite elements solutions, introduce our requirements on mesh-refinement, error estimator, and algebraic solver, state our adaptive algorithm, and present our main results, including some discussions.

\subsection{Abstract setting}
\label{section:abstract}
Let $\HH$ be a Hilbert space over $\K \in \{ \R, \C \}$ with scalar product $\product\cdot\cdot$, corresponding norm $\enorm\cdot$, and dual space $\HH'$ (with canonical operator norm $\enorm{\cdot}'$).
Let $P: \HH \to \K$ be G\^ateaux-differentiable with derivative $\AA := {\rm d}P: \HH \to \HH'$, i.e.,
\begin{align*}
 \dual{\AA w}{v}_{\HH'\times\HH} = \lim_{\substack{t \to 0 \\ t \in \R}} \frac{P(w+tv)-P(w)}{t}
 \quad \text{for all } v, w \in \HH.
\end{align*}
We suppose that the operator $\AA$ is {\em strongly monotone} and {\em Lipschitz-continuous}, i.e.,
\begin{align}\label{def:assumptions_operator}
 \alpha \, \enorm{w - v}^2 \le {\rm Re} \, \dual{\AA w - \AA v}{w - v}_{\HH'\times\HH}
 \quad \text{and} \quad
 \enorm{\AA w - \AA v}' \le L \, \enorm{w - v}
\end{align}
for all $v, w \in \HH$, where $0 < \alpha \le L$ are generic real constants. 

Given a linear and continuous functional $F \in \HH'$, the main theorem on monotone operators~\cite[Section~25.4]{zeidler} yields existence and uniqueness of the solution $u^\star \in \HH$ of
\begin{align}\label{eq:exact_solution}
 \dual{\AA u^\star}{v}_{\HH'\times\HH} = F(v)
 \quad \text{for all } v \in \HH.
\end{align}
The result actually holds true for any closed subspace $\XX_\coarse \subseteq \HH$, which also gives rise to a unique $u_\coarse^\star \in \XX_\coarse$ such that
\begin{align}\label{eq:exact_solution_disc}
 \dual{\AA u_\coarse^\star}{v_\coarse}_{\HH'\times\HH} = F(v_\coarse)
 \quad \text{for all } v_\coarse \in \XX_\coarse.
\end{align}

Finally, with the {\em energy functional} $\EE := {\rm Re}\,(P - F)$, it holds that
\begin{align}\label{eq:energy}
 \frac{\alpha}{2} \, \enorm{v_\coarse - u_\coarse^\star}^2
 \le \EE(v_\coarse) - \EE(u_\coarse^\star)
 \le \frac{L}{2} \, \enorm{v_\coarse - u_\coarse^\star}^2
 \quad \text{for all } v_\coarse \in \XX_\coarse;
\end{align}
see, e.g.,~\cite[Lemma~5.1]{banach}.
In particular, $u^\star \in \HH$ (resp.\ $u_\coarse^\star \in \XX_\coarse^\star$) is the unique minimizer of the minimization problem
\begin{align}\label{eq:minimization}
 \EE(u^\star) = \min_{v \in \HH} \EE(v)
 \quad \big( \text{resp.} \quad
 \EE(u_\coarse^\star) = \min_{v_\coarse \in \XX_\coarse} \EE(v_\coarse) \big).
\end{align}
As for linear elliptic problems, it follows from \eqref{def:assumptions_operator}--\eqref{eq:exact_solution_disc} that the present setting guarantees the C\'ea lemma (see, e.g.,~\cite[Section~25.4]{zeidler})
\begin{align}\label{eq:cea}
 \enorm{u^\star - u_\coarse^\star}
 \le \Ccea \, \enorm{u^\star - v_\coarse}
 \quad \text{for all } v_\coarse \in \XX_\coarse
 \quad \text{with} \quad
 \Ccea := L/\alpha.
\end{align}

\subsection{Mesh-refinement}
\label{section:mesh-refinement}

Let $\TT_\coarse$ be a conforming simplicial mesh of $\Omega$, \ie, a partition of $\overline \Omega$ into closed simplices $T$ such that $\bigcup_{T \in \TT_\coarse} T = \overline \Omega$ and such that the intersection of two different simplices is either empty or their common vertex, edge, or face.
We assume that $\refine(\cdot)$ is a fixed mesh-refinement strategy, e.g., newest vertex bisection~\cite{stevenson2008}.
We write $\TT_\fine = \refine(\TT_\coarse,\MM_\coarse)$ for the coarsest one-level refinement of $\TT_\coarse$, where all marked elements $\MM_\coarse \subseteq \TT_\coarse$ have been refined, i.e., $\MM_\coarse \subseteq \TT_\coarse \backslash \TT_\fine$. We write $\TT_\fine \in \refine(\TT_\coarse)$, if $\TT_\fine$ can be obtained by finitely many steps of one-level refinement (with appropriate, yet arbitrary marked elements in each step). We define $\T := \refine(\TT_0)$ as the set of all meshes which can be generated from the initial simplicial mesh $\TT_0$ of $\Omega$ by use of $\refine(\cdot)$.
Finally, we associate to each $\TT_\coarse\in\T$ a corresponding finite-dimensional subspace $\XX_\coarse \subsetneqq \HH$, where we suppose that $\XX_H\subseteq\XX_h$ whenever $\TT_H,\TT_h\in\T$ with $\TT_h\in\refine(\TT_H)$.

For our analysis, we only employ that the shape-regularity of all meshes $\TT_\coarse\in\T$ is uniformly bounded by that of $\TT_0$ together with the following structural properties~\eqref{axiom:sons}--\eqref{axiom:closure}, where $\Cson \ge 2$ and $\Cmesh > 0$ are generic constants:
\renewcommand{\theenumi}{R\arabic{enumi}}
\begin{enumerate}
\bf
\item\label{axiom:sons}
\rm
\textbf{splitting property:} Each refined element is split into finitely many sons, i.e., for all $\TT_\coarse \in \mathbb{T}$ and
all $\mathcal{M}_\coarse \subseteq \TT_\coarse$, the mesh
$\TT_\fine = \refine(\TT_\coarse, \mathcal{M}_\coarse)$ satisfies that
\begin{align*}
	\# (\TT_\coarse \setminus \TT_\fine) + \# \TT_\coarse \leq \# \TT_\fine
	\leq \Cson \, \# (\TT_\coarse \setminus \TT_\fine) + \# (\TT_\coarse \cap \TT_\fine);
\end{align*}
\bf
\item\label{axiom:overlay}
\rm
\textbf{overlay estimate:} For all meshes $\TT \in \mathbb{T}$ and $\TT_\coarse,\TT_\fine \in \refine(\TT)$,
there exists a common refinement $\TT_\coarse \oplus \TT_\fine \in \refine(\TT_\coarse) \cap \refine(\TT_\fine) \subseteq \refine(\TT)$ such that
\begin{align*}
	\# (\TT_\coarse \oplus \TT_\fine) \leq \# \TT_\coarse + \# \TT_\fine - \# \TT;
\end{align*}
\bf
\item\label{axiom:closure}
\rm
\textbf{mesh-closure estimate:}
 For each sequence $(\TT_\ell)_{\ell \in \N_0}$ of successively refined meshes, i.e., $\TT_{\ell+1} := \refine(\TT_\ell,\MM_\ell)$ with $\MM_\ell \subseteq \TT_\ell$ for all $\ell \in \N_0$, it holds  that	
\begin{align*}
	\# \TT_\ell - \# \TT_0 \leq \Cmesh \sum_{j=0}^{\ell -1} \# \mathcal{M}_j.
\end{align*}%
\end{enumerate}

For newest vertex bisection, we refer to~\cite{bdd2004,stevenson2007,stevenson2008,ckns2008,kpp2013,gss2014} for the validity of~\eqref{axiom:sons}--\eqref{axiom:closure}. For red-refinement with first-order hanging nodes, details are found in~\cite{bn2010}.

\subsection{Error estimator}
\label{section:estimator}
For each mesh $\TT_\coarse \in \T$, suppose that we can compute refinement indicators
\begin{align}
 \eta_\coarse(T,v_\coarse) \ge 0
 \quad \text{for all } T \in \TT_\coarse
 \text{ and all } v_\coarse \in \XX_\coarse.
\end{align}
We denote
\begin{align}
 \eta_\coarse(\UU_\coarse, v_\coarse)
 := \bigg(\sum_{T \in \UU_\coarse} \eta_\coarse(T,v_\coarse)^2 \bigg)^{1/2}
 \quad \text{for all } \UU_\coarse \subseteq \TT_\coarse
\end{align}
and abbreviate $\eta_\coarse(v_\coarse) := \eta_\coarse(\TT_\coarse, v_\coarse)$.
As far as the estimator is concerned, we assume the following \emph{axioms of adaptivity} from~\cite{axioms} for all $\TT_\coarse \in \T$ and all $\TT_\fine \in \refine(\TT_\coarse)$, where $\Cstab, \Crel > 0$, and $0 < \qred < 1$ are generic constants:
\renewcommand{\theenumi}{A\arabic{enumi}}
\begin{enumerate}
\bf
\item stability:\label{axiom:stability}\rm\,
$| \eta_\fine(\UU_\coarse, v_\fine) - \eta_\coarse(\UU_\coarse,v_\coarse) |
\le \Cstab \enorm{v_\fine - v_\coarse}$
 for all $v_\fine\in\XX_\fine$, $v_\coarse \in \XX_\coarse$ and all $\UU_\coarse \subseteq \TT_\coarse \cap \TT_\fine$;
\bf
\item reduction:\label{axiom:reduction}\rm\,
$\XX_\coarse \subseteq \XX_\fine$ and $\eta_\fine(\TT_\fine \backslash \TT_\coarse, v_\coarse) \le \qred \, \eta_\coarse(\TT_\coarse \backslash \TT_\fine, v_\coarse)$ for all $v_\coarse \in \XX_\coarse$;
\bf
\item reliability:\label{axiom:reliability}\rm\,
$\enorm{u^\star - u_\coarse^\star} \le \Crel \, \eta_\coarse(u_\coarse^\star)$;
\bf
\item discrete reliability:\label{axiom:discrete_reliability}\rm\,
$\enorm{u_\fine^\star - u_\coarse^\star} \le \Crel \, \eta_\coarse(\TT_\coarse \backslash \TT_\fine, u_\coarse^\star)$.
\end{enumerate}
We stress that the exact discrete solutions $u_\coarse^\star$ (resp.\ $u_\fine^\star$) in~\eqref{axiom:reliability}--\eqref{axiom:discrete_reliability} will never be computed but are only auxiliary quantities for the analysis.

We refer to Section~\ref{sec:numerics_model_problem} below for precise assumptions on the nonlinearity $A(\cdot)$ of problem~\eqref{eq:strongform} such that the standard residual error estimator satisfies~\eqref{axiom:stability}--\eqref{axiom:discrete_reliability} for lowest-order Courant finite elements; see also Section~\ref{sec:numerics_weak_formulation}--\ref{example:estimator}.

%
%
%
%

\subsection{Algebraic solver}
\label{section:alg_solver}

For given linear and continuous functionals $G \in \HH'$, we consider linear systems of algebraic equations of the type
\begin{align}\label{eq:exact_solution_alg}
 \product{u_\coarse^\flat}{v_\coarse} = G(v_\coarse)
 \quad \text{for all } v_\coarse \in \XX_\coarse
\end{align}
with unique (but not computed) exact solution $u_\coarse^\flat \in \XX_\coarse$.
We suppose here that we have at hand a contractive iterative algebraic solver for problems of the form~\eqref{eq:exact_solution_alg}. More precisely, let $u_\coarse^{0} \in \XX_\coarse$ be an initial guess and let the solver produce a sequence $u_\coarse^{j} \in \XX_\coarse$, $j \ge 1$. Then, we suppose that there exists a generic constant $0 < \qpcg < 1$ such that
\begin{align}\label{eq:alg:contraction}
 \enorm{u_\coarse^\flat - u_\coarse^j}
 \le \qpcg \, \enorm{u_\coarse^\flat - u_\coarse^{j-1}}
 \quad \text{for all } j \ge 1.
\end{align}
Examples for such solvers are suitably preconditioned conjugate gradients or multigrid; see, e.g., Olshanskii and Tyrtyshnikov~\cite{Olsh_Tyrt_it_methods_14} and the references therein.

\subsection{Adaptive algorithm}
\label{section:algorithm}
The present work considers an adaptive algorithm for numerical approximation of problem~\eqref{eq:exact_solution} which steers mesh-refinement with index $\ell$, a (perturbed) contractive Banach--Picard iteration with index $k$, and a contractive algebraic solver with index $j$. On each step $(\ell,k,j)$, it yields an approximation $u_\ell^{k,j}\in\XX_\ell$ to the unique but unavailable $u_\ell^\star \in \XX_\ell$ on the mesh $\TT_\ell$ defined by
\begin{align}\label{eq:exact_solution_ell}
 \dual{\AA u_\ell^\star}{v_\ell}_{\HH'\times\HH} = F(v_\ell)
 \quad \text{for all } v_\ell \in \XX_\ell.
\end{align}
Reporting for the summary of notation to Table~\ref{tab_not}, the algorithm reads as follows:

\begin{table}
\begin{center}
\begin{tabular}{cccccc}
\toprule %
\multicolumn{1}{c}{}&\multicolumn{2}{c}{counter}&\multicolumn{3}{c}{discrete solution}\\
\multicolumn{1}{c}{}&\multicolumn{2}{c}{}&\multicolumn{2}{c}{available}&\multicolumn{1}{c}{unavailable}\\
\cmidrule(r){2-3}\cmidrule(l){4-5}\cmidrule(l){6-6} & running & stopping &
running & stopping & exact \\\hline
mesh & $\ell$  & $\underline\ell$ & $u_\ell^{\k,\j}$ & $u_{\underline\ell}^{\k,\j}$ & $u_\ell^\star$ from~\eqref{eq:exact_solution_ell}\\
linearization & $k$ & $\k$  & $u_\ell^{k,\j}$ & $u_{\ell}^{\k,\j}$ & $u_\ell^{k,\star}$ from~\eqref{eq:banach_non_pert}\\
algebraic solver & $j$ & $\j$ & $u_\ell^{k,j}$ & $u_\ell^{k,\j}$ & \\
\bottomrule
\end{tabular}
\vspace{0.1cm} \caption{Counters and discrete solutions in Algorithm~\ref{algorithm}.} \label{tab_not}
\end{center}
\end{table}

\begin{algorithm}\label{algorithm}
{\bfseries Input:} Initial mesh $\TT_0$ and initial guess $u_0^{0,0} = u_0^{0,\j} \in \XX_0$, parameters $0 < \theta \le 1$, $0<\lpcg<1$, $0 < \lpic$, and $\Cmark \ge 1$, counters $\ell=k=j= 0$.
\\
{\bfseries Adaptive loop:} Iterate the following steps~{\rm(i)--(vi)}:
\begin{enumerate}

	\item[\rm{(i)}] {\bf Repeat} the following steps {\rm(a)--(c)}:
	
	\begin{enumerate}
		\medskip
		\item[\rm{(a)}] Define $u_\ell^{k+1,0} := u_\ell^{k,j}$ and update counters $k:=k+1$ as well as $j:=0$.
	
		\medskip
		\item[\rm{(b)}] {\bf Repeat} the following steps {\rm(I)--(III)}:
	
		\begin{enumerate}
			
			\medskip
			\item[\rm{(I)}]    Update counter $j := j+1$.
		
			\smallskip
			\item[\rm{(II)}] \label{ref_alg_step} Consider the problem of finding
                \begin{align} \label{eq:banach_non_pert}
                \begin{split}
                 &u_\ell^{k,\star} \in \XX_\ell \text{ such that, for all } v_\ell \in \XX_\ell,
                 \\& \quad
                 \product{u_\ell^{k,\star}}{v_\ell} = \product{u_\ell^{k-1,\j}}{v_\ell} \!-\! \frac{\alpha}{L^2}  \dual{\AA u_\ell^{\!k-1,\j} - F}{v_\ell}_{\HH' \times \HH}
                \end{split}
                \end{align}
                and do one step of the algebraic solver applied to~\eqref{eq:banach_non_pert} starting from $u_\ell^{k,j-1}$, which yields $u_\ell^{k,j}$ (an approximation to $u_\ell^{k,\star}$).

			\smallskip
			\item[\rm{(III)}]  	Compute the local indicators $\eta_\ell(T,u_\ell^{k,j})$ for all $T\in\TT_\ell$.

		\end{enumerate}

		\smallskip
        \noindent {\bf Until}\quad
        $\enorm{u_\ell^{k,j}-u_\ell^{k,j-1}} \le \lpcg\big[\,\eta_\ell(u_\ell^{k,j})+\enorm{u_\ell^{k,j}-u_\ell^{k-1,\j}}\,\big]$.
        \refstepcounter{equation}\hfill{\rm(\theequation)}\label{eq:st_crit_pcg}
		
		\medskip
		\item[\rm{(c)}]    Define $\j := \j(\ell,k):= j$.
		
	\end{enumerate}

	\smallskip
    \noindent {\bf Until}\quad
    $\enorm{u_\ell^{k,\j}-u_\ell^{k-1,\j}} \le \lpic\eta_\ell(u_\ell^{k,\j})$.
    \refstepcounter{equation}\hfill{\rm(\theequation)}\label{eq:st_crit_pic}

	\medskip
	\item[\rm{(ii)}] Define $\k := \k(\ell):= k$.

    \medskip
    \item[\rm{(iii)}] If $\eta_{\ell}(u_\ell^{\k,\j}) = 0$, set $\underline \ell := \ell$ and exit.

	\medskip
	\item[\rm{(iv)}]  Determine a set $\MM_\ell \subseteq \TT_\ell$ with up to the multiplicative constant $\Cmark$  minimal cardinality such that
		\begin{equation}\label{eq:doerfler}
			\theta\, \eta_{\ell}(u_{\ell}^{\k,\j}) \le  \eta_{\ell}(\MM_\ell, u_{\ell}^{\k,\j}).
		\end{equation}
	
	\medskip
	\item[\rm{(v)}]  Generate $\TT_{\ell+1} := \refine(\TT_\ell,\MM_\ell)$ and define $u_{\ell+1}^{0,0} := u_{\ell+1}^{0,\j} := u_\ell^{\k,\j}$.
	
	\medskip
	\item[\rm{(vi)}] Update counters $\ell := \ell + 1$, $k := 0$, and $j := 0$ and continue with \rm{(i)}.
	
\end{enumerate}
{\bfseries Output:} Sequence of discrete solutions $u_\ell^{k,j}$ and
corresponding error estimators $\eta_\ell(u_\ell^{k,j})$.
\end{algorithm}

Some remarks are in order to explain the nature of Algorithm~\ref{algorithm}. The innermost loop (Algorithm~\ref{algorithm}(ib)) steers the algebraic solver. Note here that the exact solution $u_\ell^{k,\star}$ of~\eqref{eq:banach_non_pert} is not computed but only approximated by the computed iterates $u_\ell^{k,j}$. For the linear system~\eqref{eq:banach_non_pert}, the contraction assumption~\eqref{eq:alg:contraction} reads as
\begin{align}\label{eq:pcg:contraction}
 \enorm{u_\ell^{k,\star} - u_\ell^{k,j}}
 \le \qpcg \, \enorm{u_\ell^{k,\star} - u_\ell^{k,j-1}}
 \quad \text{for all } j \ge 1.
\end{align}
Then, the triangle inequality implies that
\begin{align}\label{eq2:pcg:contraction}
 \frac{1 - \qpcg}{\qpcg} \, \enorm{u_\ell^{k,\star} - u_\ell^{k,j}}
 \le \enorm{u_\ell^{k,j} - u_\ell^{k,j-1}}
 \le (1 + \qpcg) \, \enorm{u_\ell^{k,\star} - u_\ell^{k,j-1}}.
\end{align}
Hence, the term $\enorm{u_\ell^{k,j} - u_\ell^{k,j-1}}$ provides a means to estimate the algebraic error $\enorm{u_\ell^{k,\star} - u_\ell^{k,j}}$.
Thus, the approximation $u_\ell^{k,j}$ is accepted and the algebraic solver is stopped if the algebraic error estimate $\enorm{u_\ell^{k,j} - u_\ell^{k,j-1}}$ is, up to the threshold $\lpcg$, below the estimate on the sum~$\eta_\ell(u_\ell^{k,j})+\enorm{u_\ell^{k,j}-u_\ell^{k-1,\j}}$ of the discretization and linearization errors; see~\eqref{eq:st_crit_pcg}. Since $\enorm{u_\ell^{k,1}-u_\ell^{k,0}} = \enorm{u_\ell^{k,1}-u_\ell^{k-1,\j}}$, the stopping criterion~\eqref{eq:st_crit_pcg} terminates the solver for $\lpcg \ge 1$ for $j=1$, i.e., the algebraic solver would always be stopped after one step. This motivates the restriction $\lpcg<1$.

The middle loop (Algorithm~\ref{algorithm}(i)) steers the linearization by means of the (perturbed) Banach--Picard iteration. Lemma~\ref{lemma:picard:contraction} below shows that the term $\enorm{u_\ell^{k,\j} - u_\ell^{k-1,\j}}$ estimates the linearization error $\enorm{u_\ell^\star - u_\ell^{k,\j}}$. Note here that, {\sl a~priori}, only the non-perturbed Banach--Picard iteration corresponding to the (unavailable) exact solve of~\eqref{eq:banach_non_pert} yielding $u_\ell^{k,\star}$ would lead to the contraction
\begin{align}\label{eq:picard:contraction0}
 \enorm{u_\ell^\star - u_\ell^{k,\star}} \le \qpic \, \enorm{u_\ell^\star - u_\ell^{k-1,\j}}
 \quad \text{for all }
 (\ell, k, 0) \in \QQ \text{ with } k \ge 1,
\end{align}
where $0 < \qpic := (1-\alpha^2/L^2)^{1/2} < 1$.
The approximation $u_\ell^{k,\j}$ is accepted and the linearization is stopped if the linearization error estimate $\enorm{u_\ell^{k,\j} - u_\ell^{k-1,\j}}$ is, up to the threshold $\lpic$, below the discretization error estimate $\eta_\ell(u_\ell^{k,\j})$; see~\eqref{eq:st_crit_pic}.

Finally, the outermost adaptive loop steers the local mesh-refinement. To this end, the D\"orfler marking criterion~\eqref{eq:doerfler} from~\cite{doerfler1996} is employed to mark elements $T \in \MM_\ell$ for refinement, unless $\eta_{\ell}(u_\ell^{\k,\j}) = 0$, in which case Proposition~\ref{proposition:reliability} below ensures that the approximation $u_\ell^{\k,\j}$ coincides with the exact solution $u^\star$ of~\eqref{eq:exact_solution}.

\begin{remark}
In a practical implementation, Algorithm~\ref{algorithm} has to be complemented by appropriate stopping criteria in all of the loops so that the computation is terminated if $u_\ell^{k,j}\in\XX_\ell$ is a sufficiently accurate approximation of $u^\star$. This can be done with the help of the reliable {\sl a~posteriori} error estimates summarized in Proposition~\ref{proposition:reliability} below.
\end{remark}

\subsection{Index set $\QQ$ for the triple loop}

To analyze Algorithm~\ref{algorithm}, define the index set
\begin{align} \label{eq_QQ}
 \QQ := \set{(\ell,k,j) \in \N_0^3}{\text{index triple $(\ell,k,j)$ is used in Algorithm~\ref{algorithm}}}.
\end{align}
Since Algorithm~\ref{algorithm} is sequential, the index set $\QQ$ is naturally ordered. For indices $(\ell,k,j), (\ell',k',j') \in \QQ$, we write
\begin{align}
 (\ell,k,j) < (\ell',k',j')
 \,\,\, \stackrel{\text{def}}{\Longleftrightarrow} \,\,\,
 (\ell,k,j) \text{ appears earlier in Algorithm~\ref{algorithm} than } (\ell',k',j').
\end{align}
With this order, we can define
\begin{align*}
|(\ell,k,j)|:=\#\set{(\ell',k',j')\in\QQ}{(\ell',k',j')<(\ell,k,j)},
\end{align*}
which is the \emph{total step number} of Algorithm~\ref{algorithm}.
We make the following definitions, which are consistent with that of Algorithm~\ref{algorithm}, and additionally define $\j(\ell,0):=0$:
\begin{align*}
 \underline\ell &:= \sup\set{\ell \in \N_0}{(\ell,0,0) \in \QQ} \in\N_0\cup\{\infty\},
 \\
 \k(\ell) &:= \sup\set{k \in \N_0}{(\ell,k,0) \in \QQ} \in\N_0\cup\{\infty\}
 \quad \text{if } (\ell,0,0) \in \QQ,
 \\
 \j(\ell,k) &:= \sup\set{j \in \N_0}{(\ell,k,j) \in \QQ} \in\N_0\cup\{\infty\}
 \quad \text{if } (\ell,k,0) \in \QQ.
\end{align*}
Generically, it holds that $\underline\ell = \infty$, i.e., infinitely many steps of mesh-refinement take place. However, our analysis also covers the cases that either the $k$-loop (linearization) or the $j$-loop (algebraic solver) do not terminate, i.e.,
\begin{align*}
 \k(\underline\ell) = \infty \,\, \text{ if } \,\, \underline\ell < \infty
 \quad \text{resp.} \quad
 \j(\underline\ell,\k)=\infty \,\, \text{ if } \,\, \underline\ell<\infty{\rm~and~} \k(\underline\ell) < \infty,
\end{align*}%
or that the exact solution $u^\star$ is hit at step (iii) of Algorithm~\ref{algorithm} (recall that $\eta_{\underline\ell}(u_{\underline\ell}^{\k,\j}) = 0$ implies $u^\star = u_{\underline\ell}^{\k,\j}$ by virtue of Proposition~\ref{proposition:reliability} below).
To abbreviate notation, we make the following convention: If the mesh index $\ell \in \N_0$ is clear from the context, we simply write $\k := \k(\ell)$, e.g., $u_\ell^{\k,j} := u_\ell^{\k(\ell),j}$. Similarly, we simply write $\j := \j(\ell,k)$, e.g., $u_\ell^{k,\j} := u_\ell^{k,\j(\ell,k)}$.

Note that there in particular holds $u_{\ell-1}^{\k,\j} = u_{\ell}^{0,0} = u_{\ell}^{1,0}$ for all $(\ell,0,0)\in\QQ$ with $\ell \geq 1$. Hence, these approximate solutions are indexed three times. This is our notational choice that will not be harmful for what follows; alternatively, one could only index the approximate solutions that appear on step~(i.b.II) of Algorithm~\ref{algorithm}.
\subsection{Main results}
\label{section:theorems}
Our first proposition provides computable upper bounds for the energy error $\enorm{u^\star-u_\ell^{k,j}}$ of the iterates $u_\ell^{k,j}$ of Algorithm~\ref{algorithm} at any step $(\ell,k,j) \in \QQ$. In particular, we note that the stopping criteria~\eqref{eq:st_crit_pcg}--\eqref{eq:st_crit_pic} ensure reliability of $\eta_\ell(u_\ell^{\k,\j})$ for the final perturbed Banach--Picard iterates $u_\ell^{\k,\j}$. The proof ist postponed to Section~\ref{sec:proof_error_control}. 

\begin{proposition}[Reliability at various stages of Algorithm~\ref{algorithm}]\label{proposition:reliability}
Suppose~\eqref{axiom:stability} and~\eqref{axiom:reliability}. Then, for all $(\ell,k,j)\in\QQ$, it holds that
\begin{align}\label{eq:reliability}
 \enorm{u^\star - u_\ell^{k,j}}
  \le \Crel'  \left\{
  \begin{aligned}
  &\eta_\ell(u_\ell^{k,j}) + \enorm{u_\ell^{k,j}-u_\ell^{k-1,\j}} + \enorm{u_\ell^{k,j}-u_\ell^{k,j-1}}&&\\
  &											&&\hspace{-3.27cm}\text{if \ }0 < k \le \k(\ell) \text{ and } 0 < j \le \j(\ell,k), \\
  &\eta_\ell(u_\ell^{k,\j}) + \enorm{u_\ell^{k,\j}-u_\ell^{k-1,\j}}
  												&&\hspace{-3.27cm} \text{if \ }0 < k \le \k(\ell)\text{ and }j=\j(\ell,k), \\
  &\eta_\ell(u_\ell^{\k,\j})
  												&&\hspace{-3.27cm} \text{if \ } k = \k(\ell) \text{ and }\j=\j(\ell,\k), \\
  &\eta_{\ell-1}(u_{\ell-1}^{\k,\j})
  												&&\hspace{-3.27cm} \text{if \ } k = 0 \text{ and }\ell > 0.
\end{aligned}\right.
\end{align}
The constant $\Crel' > 0$ depends only on $\Crel$, $\Cstab$, $\qpcg$, $\lpcg$, $\qpic$, and $\lpic$.
\end{proposition}

The first main theorem states linear convergence in {\em each step} of the adaptive algorithm, i.e., algebraic solver \emph{or} linearization \emph{or} mesh-refinement. The proof is given in Section~\ref{section:lin_cvg}.

\begin{theorem}[linear convergence]\label{theorem:linconv}
Suppose~\eqref{axiom:stability}--\eqref{axiom:reliability}. Then, there
exist $\lpcg^\star, \lpic^\star > 0$ such that for arbitrary $0<\theta\leq 1$ as well as for all $0<\lpcg<1$ and $0 < \lpic$ with $0 < \lpcg + \lpcg/\lpic < \lpcg^\star$ and $0 < \lpic/\theta < \lpic^\star$, there exist constants $\Clin\geq 1$ and $0<\qlin<1$ such that the quasi-error
\begin{align}\label{eq:def:Delta}
 \Delta_\ell^{k,j}
 := \enorm{u^\star - u_\ell^{k,j}} + \enorm{u_\ell^{k,\star} - u_\ell^{k,j}} + \eta_\ell(u_\ell^{k,j}),
\end{align}
composed of the overall error, the algebraic error, and the error estimator,
is linearly convergent in the sense of
\begin{align}\label{eq:linconv}
\Delta_{\ell'}^{k',j'} \leq \Clin\,\qlin^{|(\ell',k',j')|-|(\ell,k,j)|}\,\Delta_\ell^{k,j}
\end{align}
for all $(\ell,k,j),(\ell',k',j')\in\QQ$ with $(\ell',k',j')\geq(\ell,k,j)$. The constants $\Clin$ and $\qlin$  depend only on $\Crel$, $\Cstab$, $\qred$, $\theta$, $\qpcg$, $\lpcg$, $\qpic$, $\lpic$, $\alpha$, and $L$.
\end{theorem}

Note that $\Delta_{\ell'}^{k',j'} = \Delta_\ell^{k,j}$ when $(\ell',k',j')=(\ell,k,j)$, and then~\eqref{eq:linconv} holds with equality for $\Clin=1$. There are other cases where $u_{\ell'}^{k',j'} = u_\ell^{k,j}$ and where $u_{\ell'}^{k',j'} = u_\ell^{k,j}$ together with $\TT_{\ell'} = \TT_\ell$, and consequently $\eta_{\ell'}(u_{\ell'}^{k',j'}) = \eta_\ell(u_\ell^{k,j})$, related to our notational choice for $\QQ$ in~\eqref{eq_QQ} that also indexes nested iterates. The case with $\ell'=\ell$ arises for instance when $j = \j$, $j'=0$, and $k' = k+1$; see step (ia) of Algorithm~\ref{algorithm}. Note, however, that in such a situation, typically $u_{\ell'}^{k',\star} \neq u_\ell^{k,\star}$, and consequently $\Delta_{\ell'}^{k',j'} \neq \Delta_\ell^{k,j}$. A situation where $\Delta_{\ell'}^{k',j'} = \Delta_\ell^{k,j}$ for $(\ell',k',j') \neq (\ell,k,j)$ can nevertheless also appear, and is covered in~\eqref{eq:linconv}. For instance, in the above example, when $j = \j$, $j'=0$, $k' = k+1$, and $\ell' = \ell$, and where moreover $u_\ell^{k,j} = u_\ell^{k,\star} = u_\ell^\star$ (so that $u_\ell^{k,j} = u_\ell^{k,\star} = u_{\ell'}^{k',\star} = u_{\ell'}^{k',j'} = u_\ell^\star$), Algorithm~\ref{algorithm} only effectuates one step of the algebraic solver on the linearization step $k'$, so that $\Clin = 1/\qlin$ leads to equality in~\eqref{eq:linconv} where now $|(\ell',k',j')|-|(\ell,k,j)| = 1$.

The second main result states optimal decay rate of the quasi-error $\Delta_\ell^{k,j}$ of~\eqref{eq:def:Delta} (and consequently of the total error $\enorm{u^\star - u_\ell^{k,j}}$) in terms of the number of degrees of freedom added in the space $\XX_\ell$ with respect to $\XX_0$. More precisely, the result states that if the unknown weak solution $u$ of~\eqref{eq:exact_solution} can be approximated at algebraic decay rate $s$ with respect to the number of mesh elements added in the refinement of $\TT_0$ (plus one) for a best-possible mesh, then Algorithm~\ref{algorithm} achieves the same decay rate $s$ with respect to the number of elements actually added in Algorithm~\ref{algorithm}, $(\#\TT_{\ell} - \#\TT_0 + 1 )$, up to a generic multiplicative constant. The proof of the following Theorem~\ref{theorem:rate} is given in Section~\ref{section:optimal_rate_DoFs}.

\begin{theorem}[optimal decay rate wrt.\ degrees of freedom]\label{theorem:rate}
Suppose~\eqref{axiom:stability}--\eqref{axiom:discrete_reliability} and~\eqref{axiom:sons}--\eqref{axiom:closure}.
Recall $\lpcg^\star, \lpic^\star > 0$ from Theorem~\ref{theorem:linconv}.
Let $\Cpic:=\qpic/(1-\qpic)>0$, $\Cpcg:=\qpcg/(1-\qpcg)>0$, and $\theta_{\rm opt}:=(1+\Cstab^2\Crel^2)^{-1}$. Then, there exists $\theta^\star$ such that for all $0<\lpcg,\lpic,\theta$  with $0<\theta<\min\{1,\theta^\star\}$ as well as $\lpcg<1$, $0<\lpcg + \lpcg/\lpic<\lpcg^\star$, and $0<\lpic/\theta<\lpic^\star$, it holds that
\begin{align}\label{eq:opt:theta'}
0<\theta':=\frac{\theta+\Cstab\Big( (1+\Cpic)\Cpcg\lpcg + \big[\Cpic + (1+\Cpic)\Cpcg\lpcg\big]\lpic\Big)}{1-\lpic\,/\lpic^\star}<\theta_{\rm opt},
\end{align}
where the constant $\theta^\star>0$ depends only on $\Cstab$, $\qpic$, and $\qpcg$.  Let $s>0$ and define
\begin{align}\label{eq:As}
\norm{u^\star}{\mathbb{A}_s}:=\sup_{N\in\N_0}\Big((N+1)^s\inf_{\TT_{\rm opt}\in\T(N)}\big[\,\enorm{u^\star-u_{\rm opt}^\star}+\eta_{\rm opt}(u_{\rm opt}^\star)\,\big]\Big) \in \R_{\ge 0}\cup\{\infty\},
\end{align}
where
\begin{align*}
\T(N) := \set{\TT\in\T}{\#\TT - \#\TT_0 \leq N}.
\end{align*}
Then, there exist $\copt, \Copt > 0$ such that
\begin{align} \label{eq:opt_rate}
\copt^{-1} \, \norm{u^\star}{\mathbb{A}_s}
		&\le \sup_{(\ell,k,j) \in \QQ} ( \#\TT_{\ell} - \#\TT_0 + 1 )^s \Delta_{\ell}^{k,j}
		\le \Copt \, \max\{\norm{u^\star}{\mathbb{A}_s},\Delta_0^{0,0}\}.
\end{align}
The constant $\copt > 0$ depends only on~$\Ccea=L/\alpha$, $\Cstab$, $\Crel$, $\Cson$, $\#\TT_0$, $s$, and, if $\underline\ell<\infty$, additionally on $\underline\ell$.
The constant $\Copt > 0$ depends only on~$\Cstab$, $\Crel$, $\Cmark$, $1-\lpic/\lpic^\star$, $\Ccea=L/\alpha$, $\Crel'$, $\Cmesh$, $\Clin$, $\qlin$, $\#\TT_0$, and $s$. The maximum in the right inequality is only needed if $\ell=0$. If $\ell\ge 1$, the maximum $\max\{\norm{u^\star}{\mathbb{A}_s},\Delta_0^{0,0}\}$ can be replaced by $\norm{u^\star}{\mathbb{A}_s}$.
\end{theorem}

Note that $\Delta_0^{0,0}$ can be arbitrarily bad with bad initial guess $u_0^{0,0}$. However, $\norm{u^\star}{\mathbb{A}_s}$ as well as the constant $\Copt$ are independent of the initial guess, so that the upper bound in~\eqref{eq:opt_rate} cannot avoid $\max\{\norm{u^\star}{\A_s},\Delta_0^{0,0}\}$ for the case $\ell=0$. Such a phenomenon does not appear at later stages, since the stopping criteria~\eqref{eq:st_crit_pcg} and~\eqref{eq:st_crit_pic} ensure that, though $u_{\ell}^{\k,\j}$ does not in general coincide with $u_\ell^{\star}$, it is sufficiently accurate. If one restricts the indices to $(\ell,k,j)\in\QQ$ with $\ell\ge 1$, then the upper bound in~\eqref{eq:opt_rate} may omit $\Delta_0^{0,0}$.

Our last main result states that Algorithm~\ref{algorithm} drives the quasi-error down at each possible rate $s$ not only with respect to the number of degrees of freedom added in the space $\XX_\ell$ in comparison with $\XX_0$, but actually also with respect to the overall computational cost expressed as a cumulated sum of the number of degrees of freedom. This is an important improvement of Theorem~\ref{theorem:rate}. More precisely, under the same conditions as above, i.e., if the unknown weak solution $u$ of~\eqref{eq:exact_solution} can be approximated at algebraic decay rate $s$ with respect to the number of mesh elements added in the refinement of $\TT_0$ (plus one), then Algorithm~\ref{algorithm} generates a sequence of triple-$(\ell,k,j)$-indexed approximations (mesh, linearization, algebraic solver) such that the quasi-error decays down at rate $s$ with respect to the overall algorithmic cost expressed as the sum of the number of simplices $\#\TT_{\ell}$ over all steps $(\ell,k,j)\in\QQ$ effectuated by Algorithm~\ref{algorithm}. The proof of the following Theorem~\ref{theorem:cost} is given in Section~\ref{section:optimal_cost}.

\begin{theorem}[optimal decay rate wrt.\ overall computational cost]\label{theorem:cost} Let the assumptions of Theorem~\ref{theorem:rate} be verified. Then
\begin{align} \label{eq:opt_cost}
\copt^{-1} \, \norm{u^\star}{\mathbb{A}_s}
		&\le \sup_{(\ell',k',j') \in \QQ} \bigg( \sum_{\substack{(\ell,k,j)\in\QQ \\ (\ell,k,j) \le (\ell',k',j')}} \# \TT_\ell \bigg)^s \Delta_{\ell'}^{k',j'}
				\le \Copt' \, \max\{\norm{u^\star}{\mathbb{A}_s},\Delta_0^{0,0}\}.
\end{align}
The maximum in the right inequality is only needed if $\ell=0$. If $\ell\ge 1$, the maximum $\max\{\norm{u^\star}{\mathbb{A}_s},\Delta_0^{0,0}\}$ can be replaced by $\norm{u^\star}{\mathbb{A}_s}$. While $\copt>0$ is the constant of Theorem~\ref{theorem:rate}, the constant $\Copt'>0$ reads $\Copt':=(\#\TT_0)^s\,\Copt\,\Clin\,\big(1-\qlin^{1/s}\big)^{-s}$.\end{theorem}

Analogously to the comments after Theorem~\ref{theorem:rate}, the upper estimate in~\eqref{eq:opt_cost} cannot avoid $\max\{\norm{u^\star}{\A_s},\Delta_0^{0,0}\}$ for the case $\ell'=\ell=0$. As above, if one restricts the indices to $(\ell',k',j'),(\ell,k,j)\in\QQ$ with $\ell',\ell\ge 1$, then the upper bound in~\eqref{eq:opt_cost} may omit $\Delta_0^{0,0}$.

\section{Auxiliary results}
\label{section:aux}

\subsection{Some observations on Algorithm~\ref{algorithm}}

This section collects some elementary observations on Algorithm~\ref{algorithm} in what concerns nested iteration and stopping criteria. The given initial value of Algorithm~\ref{algorithm} reads
\begin{align}
 u_0^{0,0} = u_0^{0,\j} =  u_0^{0,\star}\in \XX_0.
\end{align}
If $(\ell,0,0) \in \QQ$ with $\ell \geq 1$, then
\begin{align}\label{eq:init:**}
 u_\ell^{0,\star}:=u_\ell^{0,0} := u_\ell^{0,\j} := u_{\ell-1}^{\k,\j} \in \XX_{\ell-1} \subseteq \XX_\ell.
\end{align}
If $(\ell,k,0) \in \QQ$, then the initial guess for the algebraic solver reads
\begin{align}\label{eq:init:*}
 u_\ell^{k,0} = \begin{cases}
 u_0^{0,0}\quad &\text{for $\ell=0$},\\
 u_{\ell-1}^{\k,\j} \quad &\text{if $k = 0$ and $\ell \geq 1$},\\
 u_\ell^{k-1,\j} \quad &\text{if $k > 0$},
 \end{cases}
\end{align}
i.e., the algebraic solver employs \emph{nested iteration}.
The stopping criterion~\eqref{eq:st_crit_pcg} of Algorithm~\ref{algorithm} guarantees that $\j(\ell,k) \ge 1$ if $k>0$ and, for all $(\ell, k, j) \in \QQ$, it holds that
\begin{align}\label{eq:pcg:stopping}
 \enorm{u_\ell^{k,\j} - u_\ell^{k,\j-1} }
 &\le  \lpcg \, \big[\, \eta_\ell(u_\ell^{k,\j}) + \enorm{u_\ell^{k,\j} - u_\ell^{k-1,\j}} \,\big]
 \quad \text{for } j = \j(\ell,k),
 \\ \label{eq:pcg:stopping_neg}
 \enorm{u_\ell^{k,j} - u_\ell^{k,j-1}}
 &>  \lpcg \, \big[\, \eta_\ell(u_\ell^{k,j}) + \enorm{u_\ell^{k,j} - u_\ell^{k-1,\j}} \,\big]
 \quad \text{for } j < \j(\ell,k),
\end{align}
i.e., the algebraic error estimate $\enorm{u_\ell^{k,j} - u_\ell^{k,j-1}}$ only drops below the discretization plus linearization error estimate at the stopping iteration $\j = \j(\ell,k)$.

The final iterates $u_\ell^{k,\j}$ of the algebraic solver are used to obtain the perturbed Banach--Picard iterates $u_\ell^{k+1,\j}$ for $k > 0$; see~\eqref{eq:banach_non_pert}. The stopping criterion~\eqref{eq:st_crit_pic} of Algorithm~\ref{algorithm} guarantees that $\k(\ell) \ge 1$ and, for all $(\ell, k, \j) \in \QQ$, it holds that
\begin{align}\label{eq:picard:stopping}
 \enorm{u_\ell^{\k,\j} - u_\ell^{\k-1,\j}}
 &\le \lpic \, \eta_\ell(u_\ell^{\k,\j})
 \quad \text{for } k = \k(\ell),
 \\ \label{eq:picard:stopping_neg}
 \enorm{u_\ell^{k,\j} - u_\ell^{k-1,\j}}
 &> \lpic \, \eta_\ell(u_\ell^{k,\j})
 \quad \text{for } k < \k(\ell),
\end{align}
i.e., the linearization error estimate $\enorm{u_\ell^{k,\j} - u_\ell^{k-1,\j}}$ only drops below the discretization error estimate at the stopping iteration $\k = \k(\ell)$.

\subsection{Contraction of the perturbed Banach--Picard iteration}

Assumption~\eqref{eq:alg:contraction} immediately implies the algebraic solver contraction~\eqref{eq:pcg:contraction} and reliability~\eqref{eq2:pcg:contraction} of the algebraic error estimate $\enorm{u_\ell^{k,j} - u_\ell^{k,j-1}}$. Similarly, one step of the non-perturbed Banach--Picard iteration~\eqref{eq:banach_non_pert} (i.e., with an exact algebraic solve of problem~\eqref{eq:banach_non_pert} with the datum $u_\ell^{k-1,\j}$) leads to contraction~\eqref{eq:picard:contraction0} and consequently to the reliability
\begin{align}\label{eq2:picard:contraction0}
 \frac{1 - \qpic}{\qpic} \, \enorm{u_\ell^{\star} - u_\ell^{k,\star}}
 \le \enorm{u_\ell^{k,\star} - u_\ell^{k-1,\j}}
 \le (1 + \qpic) \, \enorm{u_\ell^{\star} - u_\ell^{k-1,\j}}
\end{align}
of the unavailable linearization error estimate $\enorm{u_\ell^{k,\star} - u_\ell^{k-1,\j}}$.
As our first result, we now show that, for sufficiently small stopping parameters $0 < \lpcg$ in~\eqref{eq:st_crit_pcg}, we also get that the {\em perturbed} Banach--Picard iteration is a {\em contraction}. Recall that $u_\ell^\star \in \XX_\ell$ is the (unavailable) exact discrete solution given by~\eqref{eq:exact_solution_ell}, that $u_\ell^{k,\star} \in \XX_\ell$ is the (unavailable) exact linearization solution given by~\eqref{eq:banach_non_pert}, and that $u_\ell^{k,\j} \in \XX_\ell$ is the computed solution for which the algebraic solver is stopped; see~\eqref{eq:st_crit_pcg} (resp.~\eqref{eq:pcg:stopping}--\eqref{eq:pcg:stopping_neg}) for the stopping criterion.

\begin{lemma}\label{lemma:picard:contraction}
There exists $\lpcg^\star > 0$ only depending on $\qpcg$ and $\qpic$ such that
\begin{align}\label{eq:qpic'}
 0 < \qpic' := \frac{\qpic + \frac{\qpcg}{1-\qpcg} \, \lpcg^\star}{1 - \frac{\qpcg}{1-\qpcg} \, \lpcg^\star} < 1.
\end{align}
Moreover, for all stopping parameters $0<\lpcg<1$ and $0<\lpic$ from~\eqref{eq:st_crit_pcg}--\eqref{eq:st_crit_pic} such that $0 <\lpcg + \lpcg/\lpic < \lpcg^\star$, it holds that
\begin{align}\label{eq:picard:contraction}
 \enorm{u_\ell^\star - u_\ell^{k,\j}} \le \qpic' \, \enorm{u_\ell^\star - u_\ell^{k-1,\j}}
 \quad \text{for all } 1 \le k < \k(\ell).
\end{align}
This also implies that
\begin{align}\label{eq2:picard:contraction}
 \frac{1 - \qpic'}{\qpic'} \, \enorm{u_\ell^\star - u_\ell^{k,\j}}
 \le \enorm{u_\ell^{k,\j} - u_\ell^{k-1,\j}}
 \le (1 + \qpic') \, \enorm{u_\ell^\star - u_\ell^{k-1,\j}}.
\end{align}
\end{lemma}

\begin{proof}
Clearly, \eqref{eq2:picard:contraction} follows from~\eqref{eq:picard:contraction} by the triangle inequality as in~\eqref{eq2:pcg:contraction} and~\eqref{eq2:picard:contraction0}. Moreover,  \eqref{eq:qpic'} is obvious for sufficiently small $\lpcg^\star$, since $\qpic = (1-\alpha^2/L^2)^{1/2} < 1$ from~\eqref{eq:picard:contraction0} and $0 < \qpcg < 1$ is fixed from~\eqref{eq:alg:contraction}. To see~\eqref{eq:picard:contraction}, first note that
\begin{align*}
 \enorm{u_\ell^\star - u_\ell^{k,\j}}
 \le \enorm{u_\ell^\star - u_\ell^{k,\star}} + \enorm{u_\ell^{k,\star} - u_\ell^{k,\j}}
 \reff{eq:picard:contraction0}\le \qpic \enorm{u_\ell^\star - u_\ell^{k-1,\j}} + \enorm{u_\ell^{k,\star} - u_\ell^{k,\j}},
\end{align*}
where the first term corresponds to the unperturbed Banach--Picard iteration~\eqref{eq:banach_non_pert} and the second to the algebraic error.
Second, note that, since $1 \le k < \k(\ell)$,
\begin{align*}
 &\enorm{u_\ell^{k,\star} - u_\ell^{k,\j}}
 		\reff{eq2:pcg:contraction}\le \frac{\qpcg}{1-\qpcg} \, \enorm{u_\ell^{k,\j} - u_\ell^{k,\j-1} }
 		\reff{eq:pcg:stopping}\le \frac{\qpcg}{1-\qpcg} \, \lpcg \big[\, \eta_\ell(u_\ell^{k,\j}) + \enorm{u_\ell^{k,\j} - u_\ell^{k-1,\j}} \,\big]
 \\& \quad
		\reffleft{eq:picard:stopping_neg}{<} \frac{\qpcg}{1-\qpcg} \, (\lpcg+\lpcg/\lpic) \, \enorm{u_\ell^{k,\j} - u_\ell^{k-1,\j}}
 \\& \quad
 		\le \frac{\qpcg}{1-\qpcg} \, (\lpcg+\lpcg/\lpic) \, \big[\, \enorm{u_\ell^\star - u_\ell^{k,\j}} + \enorm{u_\ell^\star - u_\ell^{k-1,\j}} \,\big].
\end{align*}
Combining the latter estimates with the assumption $\lpcg + \lpcg/\lpic < \lpcg^\star$, we see that
\begin{align*}
 \enorm{u_\ell^\star - u_\ell^{k,\j}}
 \le (\qpic + \frac{\qpcg}{1-\qpcg} \,\lpcg^\star) \, \enorm{u_\ell^\star - u_\ell^{k-1,\j}} + \frac{\qpcg}{1-\qpcg} \,\lpcg^\star \, \enorm{u_\ell^\star - u_\ell^{k,\j}}.
\end{align*}
If $0 < \lpcg^\star$ is sufficiently small, it follows that
\begin{align*}
 \enorm{u_\ell^\star - u_\ell^{k,\j}}
 \le \frac{\qpic + \frac{\qpcg}{1-\qpcg} \,\lpcg^\star}{1 - \frac{\qpcg}{1-\qpcg} \,\lpcg^\star } \, \enorm{u_\ell^\star - u_\ell^{k-1,\j}}
 = \qpic'\enorm{u_\ell^\star-u_\ell^{k-1,\j}}
 \quad \text{for all } 1 \le k < \k(\ell).
\end{align*}
This concludes the proof.
\end{proof}

\subsection{Proof of Proposition~\ref{proposition:reliability} (reliable error control in Algorithm~\ref{algorithm})}
\label{sec:proof_error_control}

We are now ready to prove the estimates~\eqref{eq:reliability}.

\begin{proof}[{\bfseries Proof of Proposition~\ref{proposition:reliability}}]
First, let $(\ell,k,j)\in\QQ$ with $0<k\le\k(\ell)$ and $0<j\le\j(\ell,k)$. Due to stability~\eqref{axiom:stability}, reliability~\eqref{axiom:reliability}, and the contraction properties~\eqref{eq2:pcg:contraction} resp.~\eqref{eq2:picard:contraction0}, it holds that
\begin{align}\label{eq:rel_aux}
\begin{split}
&\enorm{u^\star - u_\ell^{k,j}}
 		\le \enorm{u^\star - u_\ell^\star} + \enorm{u_\ell^\star - u_\ell^{k,j}}
 		\reff{axiom:reliability}\lesssim \eta_\ell(u_\ell^\star) + \enorm{u_\ell^\star - u_\ell^{k,j}}\\
& \qquad \reffleft{axiom:stability}\lesssim \eta_\ell(u_\ell^{k,j}) + \enorm{u_\ell^\star - u_\ell^{k,j}}
		\le \,\eta_\ell(u_\ell^{k,j})
		+ \enorm{u_\ell^\star - u_\ell^{k,\star}}
		+ \enorm{u_\ell^{k,\star} - u_\ell^{k,j}}\\
& \qquad \reffleft{eq2:picard:contraction0}{\lesssim}\eta_\ell(u_\ell^{k,j})
		+ \enorm{u_\ell^{k,\star}-u_\ell^{k-1,\j}}
		+ \enorm{u_\ell^{k,\star} - u_\ell^{k,j}}\\
& \qquad \le\,\eta_\ell(u_\ell^{k,j})
		+ \enorm{u_\ell^{k,j}-u_\ell^{k-1,\j}}
		+ 2\enorm{u_\ell^{k,\star} - u_\ell^{k,j}}\\
& \qquad \reffleft{eq2:pcg:contraction}{\lesssim}\,\eta_\ell(u_\ell^{k,j})
		+ \enorm{u_\ell^{k,j}-u_\ell^{k-1,\j}}
		+ \enorm{u_\ell^{k,j} - u_\ell^{k,j-1}}.
\end{split}
\end{align}
This proves~\eqref{eq:reliability} for the case $0<k\le\k(\ell)$ and $0<j\le\j(\ell,k)$.

If $j=\j(\ell,k)$, we can improve this estimate using the stopping criterion~\eqref{eq:pcg:stopping} which yields that
\begin{align}\label{eq:rel_aux3}
\enorm{u_\ell^{k,\j}-u_\ell^{k,\j-1}}  \reff{eq:pcg:stopping}{\lesssim} \eta_\ell(u_\ell^{k,\j}) + \enorm{u_\ell^{k,\j}-u_\ell^{k-1,\j}}.
\end{align}
Combined with~\eqref{eq:rel_aux}, this proves~\eqref{eq:reliability} for $j=\j(\ell,k)$. If additionally $k=\k(\ell)$, the stopping criterion~\eqref{eq:picard:stopping} and the previous estimate~\eqref{eq:rel_aux3} provide that
\begin{align}
\enorm{u_\ell^{\k,\j}-u_\ell^{\k,\j-1}}  \reff{eq:rel_aux3}{\lesssim} \eta_\ell(u_\ell^{\k,\j}) + \enorm{u_\ell^{\k,\j}-u_\ell^{\k-1,\j}}  \reff{eq:picard:stopping}{\lesssim}\eta_\ell(u_\ell^{\k,\j}),
\end{align}
which proves~\eqref{eq:reliability} for this case. Finally, for $k=0$, $\ell>0$ and hence $j=\j=0$, it directly follows from nested iteration~\eqref{eq:init:**} and the previous case $k=\k(\ell-1)$ resp. $j=\j(\ell-1,\k)$ that
\begin{align}
\enorm{u^\star-u_\ell^{0,0}}  =  \enorm{u^\star-u_{\ell-1}^{\k,\j}}  \lesssim  \eta_{\ell-1}(u_{\ell-1}^{\k,\j}).
\end{align}
This concludes the proof.
\end{proof}

\subsection{An auxiliary adaptive algorithm}

Due to Lemma~\ref{lemma:picard:contraction}, the iterates $u_\ell^{k,\j}$ are contractive in the index $k$. Consequently, Algorithm~\ref{algorithm} fits into the framework of~\cite{banach} upon defining $u_\ell$ from~\cite{banach} as $u_\ell := u_\ell^{\k,\j}$ for the case where $\k(\ell) < \infty$ and $\j(\ell,\k) < \infty$, i.e., both the algebraic and the linearization solvers are stopped by~\eqref{eq:st_crit_pcg}--\eqref{eq:st_crit_pic} on the mesh $\TT_\ell$. Note that the assumption $(\ell + n+1, 0,0)\in\QQ$ below ensures this for all meshes $\TT_{\ell'}$ with $0\le \ell' \le \ell+n$. Then, we can rewrite~\cite[Lemma~4.9, eq.~(4.10)]{banach} and \cite[Theorem~5.3, eq.~(5.5)]{banach} in the current setting to conclude two important properties: First, the estimators $\eta_\ell(u_\ell^{\k,\j})$ available at step (iv) of Algorithm~\ref{algorithm} are, up to a constant, equivalent to the estimators $\eta_\ell(u_\ell^\star)$ corresponding to the unavailable exact linearization $u_\ell^\star$ of~\eqref{eq:exact_solution_ell}. And second, the estimators $\eta_\ell(u_\ell^{\k,\j})$ are linearly convergent.

\begin{lemma}[{{\cite[Lemma~4.9, Theorem~5.3]{banach}}}]\label{lemma:banach2:equi_linconv}
Recall $\lpcg^\star>0$ and $0<\qpic'<1$ from Lemma~\ref{lemma:picard:contraction}. Define $\lpic^\star:= \frac{1-\qpic'}{\qpic'\Cstab}>0$ and note that it  depends only on $\qpic$, $\qpcg$, and $\Cstab$. Then, for all $0<\theta\le 1$, all $0<\lpcg<1$ and $0<\lpic$ with $0<\lpcg + \lpcg/\lpic<\lpcg^\star$ and $0 < \lpic/\theta < \lpic^\star$, and all $(\ell,\k,\j) \in \QQ$ with $\k < \infty$ and $\j < \infty$, it holds that
\begin{align}\label{eq:eta-star}
 (1 - \lpic/\lpic^\star) \, \eta_\ell(u_\ell^{\k,\j})
 \le \eta_\ell(u_\ell^\star)
 \le (1 + \lpic/\lpic^\star) \, \eta_\ell(u_\ell^{\k,\j}).
\end{align}
Moreover, there exist $\Cghps > 0$ and $0 < \qghps < 1$ such that
\begin{align}\label{eq:linearconvergence}
 \eta_{\ell + n}(u_{\ell+n}^{\k,\j})
 \le \Cghps \, \qghps^n \, \eta_\ell(u_\ell^{\k,\j})
 \quad \text{for all } (\ell + n+1, 0,0)\in\QQ.
\end{align}
The constants $\Cghps$ and $\qghps$ depend only on $L$, $\alpha$, $\Crel$, $\Cstab$, $\qred$, $\qpcg$, and $\qpic$, as well as on the adaptivity parameters $\theta$, $\lpcg$, and $\lpic$. \qed
\end{lemma}

As a result of Lemma~\ref{lemma:banach2:equi_linconv} and Proposition~\ref{proposition:reliability}, we get the following lemma. Please note that when $\underline\ell < \infty$, the summation below only goes to $\underline\ell - 1$, as the arguments rely on~\eqref{eq:linearconvergence} which needs finite stopping indices $\k$ and $\j$ on each mesh $\TT_\ell$.

\begin{lemma}\label{lemma:banach2}
Suppose that $0<\lpcg + \lpcg/\lpic<\lpcg^\star$ (from Lemma~\ref{lemma:picard:contraction}) as well as $0<\theta\le 1$ and $0<\lpic/\theta<\lpic^\star$ (from Lemma~\ref{lemma:banach2:equi_linconv}). With the convention $\underline\ell-1=\infty$ if $\underline\ell=\infty$, there holds summability
\begin{align}\label{eq:banach2}
\sum_{\ell=\ell'+1}^{\underline\ell-1} \Delta_\ell^{\k,\j}
		\le C \, \Delta_{\ell'}^{\k,\j}
		\quad\text{for all } (\ell', \k, \j) \in \QQ,
\end{align}
where $C > 0$ depends only on $L$, $\alpha$, $\Crel$, $\Cstab$, $\qred$, $\theta$, $\qpcg$, $\qpic$, $\lpcg$, and $\lpic$.
\end{lemma}

\begin{proof}
Define $\widetilde\Delta_\ell^k := \enorm{u^\star-u_\ell^{k,\j}} + \eta_\ell(u_\ell^{k,\j})$ as the sum of overall error plus error estimator. In comparison with~\eqref{eq:def:Delta}, $\widetilde\Delta_\ell^k$ omits the algebraic error term.
With Proposition~\ref{proposition:reliability} and the linear convergence~\eqref{eq:linearconvergence}, we get that
\begin{align*}
 \sum_{\ell=\ell'+1}^{\underline\ell-1} \widetilde\Delta_\ell^\k
		\reff{eq:reliability}\lesssim \sum_{\ell=\ell'+1}^{\underline\ell-1} \eta_\ell(u_\ell^{\k,\j})
		\reff{eq:linearconvergence}\lesssim \eta_{\ell'}(u_{\ell'}^{\k,\j})\sum_{\ell=\ell'+1}^{\underline\ell-1}\qghps^{\ell-\ell'}
 \lesssim \widetilde\Delta_{\ell'}^\k.
\end{align*}
Hence, it only remains to prove that
\begin{align}\label{eq:banach2_aux1}
 \Delta_{\ell'}^{\k,\j} \simeq \widetilde\Delta_{\ell'}^{\k} \quad \text{for all } (\ell', \k,\j) \in \QQ.
\end{align}
By definition~\eqref{eq:def:Delta}, it holds that
\begin{align*}
\Delta_{\ell'}^{k,\j}
 		= \enorm{u^\star - u_{\ell'}^{k,\j}} + \enorm{u_{\ell'}^{k,\star} - u_{\ell'}^{k,\j}} + \eta_{\ell'}(u_{\ell'}^{k,\j})
 		= \widetilde\Delta_{\ell'}^k + \enorm{u_{\ell'}^{k,\star} - u_{\ell'}^{k,\j}}.
\end{align*}
Hence, it only remains to show that $\enorm{u_{\ell'}^{\k,\star} - u_{\ell'}^{\k,\j}} \lesssim \widetilde\Delta_{\ell'}^{\k}$.
To this end, note that
\begin{align*}
 \enorm{u_{\ell'}^{\k,\star} - u_{\ell'}^{\k,\j}}
 		\reff{eq2:pcg:contraction}\lesssim  \enorm{u_{\ell'}^{\k,\j} - u_{\ell'}^{\k,\j-1}}
 		\reff{eq:pcg:stopping}\lesssim \eta_{\ell'}(u_{\ell'}^{\k,\j}) + \enorm{u_{\ell'}^{\k,\j}-u_{\ell'}^{\k-1,\j}}
 		\reff{eq:picard:stopping}\lesssim \eta_{\ell'}(u_{\ell'}^{\k,\j}) \le \widetilde\Delta_{\ell'}^{\k}.
\end{align*}
This proves \eqref{eq:banach2_aux1} and concludes the proof.
\end{proof}

\section{Proof of Theorem~\ref{theorem:linconv} (linear convergence)}
\label{section:lin_cvg}

This section is dedicated to the proof of Theorem~\ref{theorem:linconv}. The core is the following lemma that extends Lemma~\ref{lemma:banach2} to our setting with the triple indices.

\begin{lemma}\label{lemma:ell-k-j}
Suppose that $0<\lpcg + \lpcg/\lpic<\lpcg^\star$ (from Lemma~\ref{lemma:picard:contraction}) as well as $0<\theta\le 1$ and $0<\lpic/\theta<\lpic^\star$ (from Lemma~\ref{lemma:banach2:equi_linconv}). Then, there exists $\Csum > 0$ such that
\begin{align} \label{eq_ell-k-j}
 \sum_{\substack{(\ell,k,j)\in\QQ \\ (\ell,k,j) > (\ell',k',j')}} \Delta_\ell^{k,j}
 \le \Csum \, \Delta_{\ell'}^{k',j'}
 \quad \text{for all } (\ell',k',j') \in \QQ.
\end{align}
The constant $\Csum$ depends only on $\Crel$, $\Cstab$, $\qred$, $\theta$, $\qpcg$, $\lpcg$, $\qpic$, $\lpic$, $\alpha$, and $L$.
\end{lemma}

\begin{proof}
{\bf Step~1.}
We prove that
\begin{align}\label{eq:Alpha-Delta}
 \boxed{\Alpha_\ell^{k,j}
 := \enorm{u_\ell^\star - u_\ell^{k,j}} + \enorm{u_\ell^{k,\star} - u_\ell^{k,j}} + \eta_\ell(u_\ell^{k,j})
 \simeq \Delta_\ell^{k,j}
 \quad \text{for all } (\ell, k, j) \in \QQ.}
\end{align}
Note that $\Alpha_\ell^{k,j}$ and $\Delta_\ell^{k,j}$ only differ in the first term, where the overall error is replaced by the (inexact) linearization error.
According to the C\'ea lemma~\eqref{eq:cea}, it holds that
\begin{align*}
 \enorm{u_\ell^\star - u_\ell^{k,j}}
 \le \enorm{u^\star - u_\ell^{k,j}} + \enorm{u^\star - u_\ell^\star}
 \reff{eq:cea}\lesssim \enorm{u^\star - u_\ell^{k,j}} \le \Delta_\ell^{k,j}.
\end{align*}
This implies that $\Alpha_\ell^{k,j} \lesssim \Delta_\ell^{k,j}$. To see the converse inequality, note that
\begin{align*}
 \enorm{u^\star - u_\ell^{k,j}}
 \le \enorm{u^\star - u_\ell^\star} + \enorm{u_\ell^\star - u_\ell^{k,j}}
 &\reff{axiom:reliability}\lesssim \eta_\ell(u_\ell^\star) + \enorm{u_\ell^\star - u_\ell^{k,j}}
 \\&
 \reff{axiom:stability}\lesssim \eta_\ell(u_\ell^{k,j}) + \enorm{u_\ell^\star - u_\ell^{k,j}}
 \le \Alpha_\ell^{k,j}.
\end{align*}
This proves $\Delta_\ell^{k,j} \lesssim \Alpha_\ell^{k,j}$ and concludes this step.

{\bf Step~2.} We prove some auxiliary estimates.
First, we prove that the algebraic error $\enorm{u_\ell^{k,\star} - u_\ell^{k,j-1}}$ dominates the modified total error $\Alpha_\ell^{k,j}$, before the algebraic stopping criterion~\eqref{eq:st_crit_pcg} is reached, i.e.,
\begin{align}\label{eq1:ell-k-j}
 \boxed{ \Alpha_\ell^{k,j} \lesssim \enorm{u_\ell^{k,\star} - u_\ell^{k,j-1}}
 \quad \text{for all }
 (\ell,k,j) \in \QQ \text{ with } k\ge 1 \text{ and } 1 \le j < \j(\ell,k).}
\end{align}

To this end, note that
\begin{align*}
 \enorm{u_\ell^{\star} - u_\ell^{k,j}}
 \le \enorm{u_\ell^{\star} - u_\ell^{k,\star}} + \enorm{u_\ell^{k,\star} - u_\ell^{k,j}}
 \reff{eq2:picard:contraction0}\lesssim \enorm{u_\ell^{k,\star} - u_\ell^{k-1,\j}} + \enorm{u_\ell^{k,\star} - u_\ell^{k,j}}.
\end{align*}
Since $1 \le j < \j(\ell,k)$, we obtain that
\begin{align*}
 \Alpha_\ell^{k,j}
 &= \enorm{u_\ell^{\star} - u_\ell^{k,j}} + \enorm{u_\ell^{k,\star} - u_\ell^{k,j}} + \eta_\ell(u_\ell^{k,j})
\\&\lesssim \enorm{u_\ell^{k,\star} - u_\ell^{k-1,\j}} + \enorm{u_\ell^{k,\star} - u_\ell^{k,j}} + \eta_\ell(u_\ell^{k,j})
 \\&
 \le 2 \, \enorm{u_\ell^{k,\star} - u_\ell^{k,j}} + \enorm{u_\ell^{k,j} - u_\ell^{k-1,\j}} + \eta_\ell(u_\ell^{k,j})
 \\&
 \reffleft{eq2:pcg:contraction} \lesssim \enorm{u_\ell^{k,j} - u_\ell^{k,j-1}} + \enorm{u_\ell^{k,j} - u_\ell^{k-1,\j}} + \eta_\ell(u_\ell^{k,j})
 \\&
 \reffleft{eq:pcg:stopping_neg}\lesssim \enorm{u_\ell^{k,j} - u_\ell^{k,j-1}}
		\reff{eq2:pcg:contraction}\lesssim \enorm{u_\ell^{k,\star} - u_\ell^{k,j-1}}.
\end{align*}
This proves~\eqref{eq1:ell-k-j}.

Second, we consider the use of nested iteration when passing to the next perturbed Banach--Picard step. We prove that
\begin{align}\label{eq2:ell-k-j}
 \boxed{ \enorm{u_\ell^{k,\star} - u_\ell^{k,0}} \lesssim \Alpha_\ell^{k-1,\j}
 \quad \text{for all }
 (\ell, k, 0) \in \QQ \text{ with } k \ge 1,}
\end{align}
To this end, note that
\begin{align*}
 \enorm{u_\ell^{k,\star} - u_\ell^{k,0}}
\reff{eq:init:*}= \enorm{u_\ell^{k,\star} - u_\ell^{k-1,\j}}
 \reff{eq2:picard:contraction0}\lesssim \enorm{u_\ell^{\star} - u_\ell^{k-1,\j}}
 \le \Alpha_\ell^{k-1,\j}.
\end{align*}
This proves~\eqref{eq2:ell-k-j}.

Third, we prove that
\begin{align}\label{eq6:ell-k-j}
 \boxed{ \Alpha_\ell^{k,\j} \lesssim \Alpha_\ell^{k,j}
 \quad \text{for all }
 (\ell,k,j) \in \QQ,}
\end{align}
related to the algebraic error contraction. Note that $k=0$ implies $\j=0$, so that~\eqref{eq6:ell-k-j} trivially holds for $k = 0$ in the form of equality. Let now $k \geq 1$.
We first consider the last but one algebraic iteration step $j = \j(\ell,k) - 1 \geq 0$. There holds that
\begin{align*}
 \Alpha_\ell^{k,\j}
 &= \enorm{u_\ell^{\star} - u_\ell^{k,\j}} + \enorm{u_\ell^{k,\star} - u_\ell^{k,\j}} + \eta_\ell(u_\ell^{k,\j})
 \\&
 \le \enorm{u_\ell^{\star} - u_\ell^{k,\j-1} } + \enorm{u_\ell^{k,\star} - u_\ell^{k,\j-1} } + \eta_\ell(u_\ell^{k,\j}) + 2 \, \enorm{u_\ell^{k,\j} - u_\ell^{k,\j-1} }
 \\&
 \reffleft{axiom:stability}\lesssim \Alpha_\ell^{k,\j-1} + \enorm{u_\ell^{k,\j} - u_\ell^{k,\j-1} }
 		\reff{eq2:pcg:contraction}\lesssim \Alpha_\ell^{k,\j-1} + \enorm{u_\ell^{k,\star} - u_\ell^{k,\j-1} }
\simeq \Alpha_\ell^{k,\j-1}.
\end{align*}
This proves~\eqref{eq6:ell-k-j} for $j = \j(\ell,k) -1 \geq 0$. Note that this argument also applies when $\j=1$.
If $0 \le j \le \j(\ell,k) - 2$, then
\begin{align*}
\Alpha_\ell^{k,\j} \lesssim \Alpha_\ell^{k,\j-1}
 \reff{eq1:ell-k-j}\lesssim  \enorm{u_\ell^{k,\star} - u_\ell^{k,\j-2}}
 \reff{eq:pcg:contraction}\le \enorm{u_\ell^{k,\star} - u_\ell^{k,j}}
 \le \Alpha_\ell^{k,j},
\end{align*}
also using that $\qpcg \leq 1$. This concludes the proof of~\eqref{eq6:ell-k-j}.

Fourth, we prove that the linearization error $\enorm{u_\ell^\star - u_\ell^{k-1,\j}}$ dominates the modified total error $\Alpha_\ell^{k,\j}$, before the linearization stopping criterion~\eqref{eq:st_crit_pic} is reached, i.e.,
\begin{align}\label{eq4:ell-k-j}
 \boxed{ \Alpha_\ell^{k,\j} \lesssim \enorm{u_\ell^\star - u_\ell^{k-1,\j}}
 \quad \text{for all }
 (\ell,k,\j) \in \QQ \text{ with } 1 \le k < \k(\ell).}
\end{align}
Since $1\le k<\k(\ell)$, we obtain that
\begin{align*}
&\Alpha_\ell^{k,\j}
		= \enorm{u_\ell^{\star} - u_\ell^{k,\j}} + \enorm{u_\ell^{k,\star} - u_\ell^{k,\j}} + \eta_\ell(u_\ell^{k,\j})
 		\reff{eq2:pcg:contraction}\lesssim \enorm{u_\ell^{\star} - u_\ell^{k,\j}} + \enorm{u_\ell^{k,\j} - u_\ell^{k,\j-1} } + \eta_\ell(u_\ell^{k,\j})
 \\& \quad
 \reffleft{eq:pcg:stopping}\lesssim \enorm{u_\ell^{\star} - u_\ell^{k,\j}} + \enorm{u_\ell^{k,\j} - u_\ell^{k-1,\j}} + \eta_\ell(u_\ell^{k,\j})
 		\reff{eq2:picard:contraction}\lesssim \enorm{u_\ell^{k,\j} - u_\ell^{k-1,\j}} + \eta_\ell(u_\ell^{k,\j})
 \\& \quad
 \reffleft{eq:picard:stopping_neg}\lesssim \enorm{u_\ell^{k,\j} - u_\ell^{k-1,\j}}
 		\reff{eq2:picard:contraction}\lesssim \enorm{u_\ell^{\star} - u_\ell^{k-1,\j}},
\end{align*}
where we employ Lemma~\ref{lemma:picard:contraction} and hence require $0<\lpcg+ \lpcg/\lpic$ to be sufficiently small. This proves~\eqref{eq4:ell-k-j}.

Fifth, we consider the use of nested iteration when refining the mesh. We prove that
\begin{align}\label{eq3:ell-k-j}
 \boxed{ \Alpha_\ell^{0,\j} \lesssim \eta_{\ell-1}(u_{\ell-1}^{\k,\j}) \le \Alpha_{\ell-1}^{\k,\j}
 \quad \text{for all }
 (\ell,\k,\j) \in \QQ.}
\end{align}
To this end, note that
\begin{align}\label{eq3:ell-k-j_aux1}
 \enorm{u_\ell^\star - u_{\ell-1}^{\k,\j}}
 		\le \enorm{u^\star - u_\ell^\star} + \enorm{u^\star - u_{\ell-1}^{\k,\j}}
 		\reff{eq:cea}\lesssim \enorm{u^\star - u_{\ell-1}^{\k,\j}}
 		\reff{eq:reliability}\lesssim \eta_{\ell-1}(u_{\ell-1}^{\k,\j}).
\end{align}
Next, recall from~\eqref{eq:init:**} that $u_\ell^{0,\star} = u_\ell^{0,\j} = u_{\ell-1}^{\k,\j}$. Hence, it follows from~\eqref{axiom:stability} used on non-refined mesh elements and \eqref{axiom:reduction} used on refined mesh elements that
\begin{align*}
 \Alpha_\ell^{0,\j} = \enorm{u_\ell^\star - u_\ell^{0,\j}} + \eta_\ell(u_\ell^{0,\j})
 		\reff{eq:init:**}= \enorm{u_\ell^\star - u_{\ell-1}^{\k,\j}} + \eta_\ell(u_{\ell-1}^{\k,\j})
 		&\reff{eq3:ell-k-j_aux1}\lesssim \eta_{\ell-1}(u_{\ell-1}^{\k,\j}) + \eta_\ell(u_{\ell-1}^{\k,\j})\\
 &\le 2 \, \eta_{\ell-1}(u_{\ell-1}^{\k,\j}).
\end{align*}
This proves~\eqref{eq3:ell-k-j}.

Sixth, we prove that
\begin{align}\label{eq5:ell-k-j}
 \boxed{ \Alpha_\ell^{\k,\j} \lesssim \Alpha_\ell^{k,\j}
 \quad \text{for all }
 (\ell,k,\j) \in \QQ  ,}
\end{align}
related to the linearization error contraction.
We first consider $k = \k(\ell) - 1 \geq 0$. Note that
\begin{align} \label{eq5:ell-k-j_aux1}
 &\enorm{u_\ell^{\k,\star} - u_\ell^{\k-1,\j}}
 		\le \enorm{u_\ell^\star - u_\ell^{\k,\star}} + \enorm{u_\ell^\star - u_\ell^{\k-1,\j}}
 		\reff{eq:picard:contraction0}\lesssim \enorm{u_\ell^\star - u_\ell^{\k-1,\j}}
 		\le \Alpha_\ell^{\k-1,\j}.
\end{align}
Hence, the triangle inequality leads to
\begin{align*}
 \Alpha_\ell^{\k,\j}
 &= \enorm{u_\ell^{\star} - u_\ell^{\k,\j}}
 		+ \enorm{u_\ell^{\k,\star} - u_\ell^{\k,\j}}
 		+ \eta_\ell(u_\ell^{\k,\j})
 \\&
 \le \enorm{u_\ell^\star - u_\ell^{\k-1,\j}}
 		+ \enorm{u_\ell^{\k,\star}-u_\ell^{\k-1,\j}}
 		+ 2 \, \enorm{u_\ell^{\k,\j}-u_\ell^{\k-1,\j}}
 		+ \eta_\ell(u_\ell^{\k,\j})
 \\&
 \reffleft{eq5:ell-k-j_aux1}\lesssim \Alpha_\ell^{\k-1,\j}
 		+ \enorm{u_\ell^{\k,\j}-u_\ell^{\k-1,\j}}
 		+ \eta_\ell(u_\ell^{\k,\j})
 \\&
 \reffleft{axiom:stability}\lesssim \Alpha_\ell^{\k-1,\j}
 		+ \enorm{u_\ell^{\k,\j} - u_\ell^{\k-1,\j}}
 		\reff{eq2:picard:contraction}\lesssim \Alpha_\ell^{\k-1,\j} + \enorm{u_\ell^\star-u_\ell^{\k-1,\j}}
 		\le 2 \, \Alpha_\ell^{\k-1,\j}.
\end{align*}
This proves~\eqref{eq5:ell-k-j} for $k = \k(\ell) - 1$. Note that the same argument also applies when $\k=1$.  If $0 \le k \le \k(\ell) - 2$, then
\begin{align*}
 \Alpha_\ell^{\k,\j} \lesssim \Alpha_\ell^{\k-1,\j} \reff{eq4:ell-k-j}\lesssim \enorm{u_\ell^\star - u_\ell^{\k-2,\j}}
 \reff{eq:picard:contraction}\le \enorm{u_\ell^\star - u_\ell^{k,\j}}
 \le \Alpha_\ell^{k,\j},
\end{align*}
also using that $\qpic' \leq 1$.
This concludes the proof of~\eqref{eq5:ell-k-j}.

Seventh, we consider the use of nested iteration when passing to the next perturbed Banach--Picard step. We prove that
\begin{align}\label{eq3:ell-k-0}
 \boxed{ \Alpha_\ell^{k,0} \lesssim \Alpha_{\ell}^{k-1,\j}
 \quad \text{for all }
 (\ell,k,0) \in \QQ \text{ with } k \ge 1.} 
\end{align}
Using~\eqref{eq2:ell-k-j} and recalling the definition $u_\ell^{k,0} = u_\ell^{k-1,\j}$, it holds that
\begin{align*}
\Alpha_\ell^{k,0} = {} & \enorm{u_\ell^\star - u_\ell^{k-1,\j}} + \enorm{u_\ell^{k,\star} - u_\ell^{k,0}} + \eta_\ell(u_\ell^{k-1,\j}) \lesssim \Alpha_\ell^{k-1,\j},
\end{align*}
which is the claim~\eqref{eq3:ell-k-0}.

{\bf Step~3.} This step collects auxiliary estimates following from the geometric series and the contraction properties of the linearization and the algebraic solver.
First, it holds that
\begin{align}
\begin{split}\label{eq:geometric:j}
 \boxed{
 \sum_{j = i+1}^{\j(\ell,k)-1} \Alpha_\ell^{k,j}
 \lesssim \enorm{u_\ell^{k,\star} - u_{\ell}^{k,i}}
 \le \Alpha_\ell^{k,i}
 \quad \text{for all } (\ell,k,i) \in \QQ \text{ with } k \ge 1.
 }
\end{split}
\end{align}
This follows immediately from
$$
 \sum_{j = i+1}^{\j(\ell,k)-1} \Alpha_\ell^{k,j}
 		\reff{eq1:ell-k-j}\lesssim
 		\sum_{j = i+1}^{\j(\ell,k)-1} \enorm{u_\ell^{k,\star} - u_{\ell}^{k,j-1}}
 		\reff{eq:pcg:contraction}\le \enorm{u_\ell^{k,\star} - u_{\ell}^{k,i}}
 		\sum_{j = i}^\infty \qpcg^{j-i}
 		\lesssim \enorm{u_\ell^{k,\star} - u_{\ell}^{k,i}}.
$$
We note that~\eqref{eq:geometric:j} also holds for $\j(\ell,k) = \infty$ (with the convention that then $\j(\ell,k)-1=\infty$).

Analogously, the contraction~\eqref{eq:picard:contraction} of the perturbed Banach--Picard iteration leads to
\begin{align}
\begin{split}\label{eq:geometric:k}
 \boxed{
 \sum_{k = i+1}^{\k(\ell)-1} \Alpha_\ell^{k,\j}
 		\lesssim \enorm{u_\ell^\star - u_\ell^{i,\j}}
 		\le \Alpha_\ell^{i,\j}
 		\quad \text{for all } (\ell,i,\j) \in \QQ.
 }
\end{split}
\end{align}
This follows immediately from
$$
 \sum_{k = i+1}^{\k(\ell)-1} \Alpha_\ell^{k,\j}
 \reff{eq4:ell-k-j}\lesssim \sum_{k = i+1}^{\k(\ell)-1} \enorm{u_\ell^\star - u_\ell^{k-1,\j}}
 \reff{eq:picard:contraction}\lesssim \enorm{u_\ell^\star - u_\ell^{i,\j}}
 \sum_{k=i}^\infty (\qpic')^{k-i}
 \lesssim \enorm{u_\ell^\star - u_\ell^{i,\j}}.
$$
We note that~\eqref{eq:geometric:k} also holds for $\k(\ell) = \infty$ (with the convention that then $\k(\ell)-1=\infty$).

With the analogous convention $\underline \ell - 1 = \infty$ when $\underline \ell = \infty$, we finally prove that
\begin{align}\label{eq:geometric:ell}
\boxed{
 \sum_{\ell = i + 1}^{\underline\ell-1} \Alpha_\ell^{\k,\j}
 		\lesssim \Alpha_i^{\k,\j}
 		\quad \text{for all } (i,\k,\j) \in \QQ.
}
\end{align}
This follows from Step~1 and
$$
 \sum_{\ell = i + 1}^{\underline\ell-1} \Alpha_\ell^{\k,\j}
 \reff{eq:Alpha-Delta}\simeq \sum_{\ell = i + 1}^{\underline\ell-1} \Delta_\ell^{\k,\j}
 \reff{eq:banach2}\lesssim \Delta_i^{\k,\j}
 \reff{eq:Alpha-Delta}\simeq  \Alpha_i^{\k,\j}.
$$

{\bf Step~4.} From now on, let $(\ell',k',j') \in \QQ$ be arbitrary. Suppose first that $\underline\ell = \infty$, i.e., both algebraic and linearization solvers terminate at some finite values $\k(\ell)$ for all $\ell \geq 0$ and $\j(\ell,k)$ for all $\ell \geq 0$ and all $k \leq \k(\ell)$, whereas infinitely many steps of mesh-refinement take place. By the definition of our index set $\QQ$ in~\eqref{eq_QQ} (which in particular features nested iterates), it holds that
\begin{align}\label{eq:sums} \begin{split}
 \sum_{\substack{(\ell,k,j)\in\QQ \\ (\ell,k,j) > (\ell',k',j')}} \Alpha_\ell^{k,j}
 = {} & \sum_{\ell = \ell' + 1}^\infty \Bigg(\Alpha_{\ell}^{0,0} + \sum_{k=1}^{\k(\ell)} \Big(\Alpha_{\ell}^{k,0} +  \sum_{j=1}^{\j(\ell,k)} \Alpha_\ell^{k,j} \Big)\Bigg) \\
 {} & + \sum_{k=k' + 1}^{\k(\ell')} \Big(\Alpha_{\ell'}^{k,0} +  \sum_{j=1}^{\j(\ell',k)} \Alpha_{\ell'}^{k,j} \Big)
 + \sum_{j=j' + 1}^{\j(\ell',k')} \Alpha_{\ell'}^{k',j}\\
\lesssim {} & \sum_{\ell = \ell' + 1}^\infty \sum_{k=1}^{\k(\ell)} \sum_{j=1}^{\j(\ell,k)} \Alpha_\ell^{k,j}
 + \sum_{k=k' + 1}^{\k(\ell')} \sum_{j=1}^{\j(\ell',k)} \Alpha_{\ell'}^{k,j}
 + \sum_{j=j' + 1}^{\j(\ell',k')} \Alpha_{\ell'}^{k',j},
\end{split}\end{align}
where we have employed estimates~\eqref{eq3:ell-k-j} and~\eqref{eq3:ell-k-0} in order to start all the summations from $k=1$ and $j=1$.

We consider the three summands in~\eqref{eq:sums} separately.
For the first sum, we infer that
\begin{align}\nonumber \label{eq:sums_aux1}
 &\sum_{\ell = \ell' + 1}^\infty \sum_{k=1}^{\k(\ell)} \sum_{j=1}^{\j(\ell,k)} \Alpha_\ell^{k,j}
 		\reff{eq:geometric:j}\lesssim \sum_{\ell = \ell' + 1}^\infty \sum_{k=1}^{\k(\ell)} ( \Alpha_\ell^{k,\j} + \enorm{u_\ell^{k,\star} - u_\ell^{k,0}} )
 		\reff{eq2:ell-k-j}\lesssim \sum_{\ell = \ell' + 1}^\infty \sum_{k=1}^{\k(\ell)} ( \Alpha_\ell^{k,\j} + \Alpha_\ell^{k-1,\j} )
 \\& \quad \nonumber
 \lesssim \sum_{\ell = \ell' + 1}^\infty \Big( \Alpha_\ell^{0,\j} + \sum_{k=1}^{\k(\ell)} \Alpha_\ell^{k,\j} \Big)
 		\reff{eq:geometric:k}\lesssim  \sum_{\ell = \ell' + 1}^\infty \big( \Alpha_\ell^{0,\j} + \Alpha_\ell^{\k,\j} \big)
 		 \reffleft{eq3:ell-k-j}\lesssim \sum_{\ell = \ell' + 1}^\infty \big( \Alpha_{\ell-1}^{\k,\j} + \Alpha_\ell^{\k,\j} \big)
 \\& \quad
 \lesssim \Alpha_{\ell'}^{\k,\j} + \sum_{\ell = \ell' + 1}^\infty \Alpha_\ell^{\k,\j}
  		\reff{eq:geometric:ell}\lesssim\Alpha_{\ell'}^{\k,\j}
 		\reff{eq5:ell-k-j}\lesssim \Alpha_{\ell'}^{k',\j}
 		\reff{eq6:ell-k-j}\lesssim \Alpha_{\ell'}^{k',j'}.
\end{align}
If $k' = \k(\ell')$, the second sum in the bound~\eqref{eq:sums} disappears. If $k' < \k(\ell')$, we infer that
\begin{align}\label{eq:sums_aux2}
\begin{split}
 &\sum_{k=k' + 1}^{\k(\ell')} \sum_{j=1}^{\j(\ell',k)} \Alpha_{\ell'}^{k,j}
 	\reff{eq:geometric:j}\lesssim \sum_{k=k' + 1}^{\k(\ell')} ( \Alpha_{\ell'}^{k,\j} + \enorm{u_{\ell'}^{k,\star} - u_{\ell'}^{k,0}} )
 	\reff{eq2:ell-k-j}\lesssim \sum_{k=k' + 1}^{\k(\ell')} ( \Alpha_{\ell'}^{k,\j} + \Alpha_{\ell'}^{k-1,\j} ) \\& \qquad
 \lesssim \Alpha_{\ell'}^{k',\j} + \sum_{k=k' + 1}^{\k(\ell')} \Alpha_{\ell'}^{k,\j}
 \reff{eq:geometric:k}\lesssim \Alpha_{\ell'}^{k',\j} + \Alpha_{\ell'}^{\k,\j}
 	\reff{eq5:ell-k-j}\le \Alpha_{\ell'}^{k',\j}
 	\reff{eq6:ell-k-j}\lesssim \Alpha_{\ell'}^{k',j'}.
\end{split}
\end{align}
If $j' = \j(\ell',k')$, the third sum in the bound~\eqref{eq:sums} disappears. If $j' < \j(\ell',k')$, we infer that
\begin{align}\label{eq:sums_aux3}
 \sum_{j=j' + 1}^{\j(\ell',k')} \Alpha_{\ell'}^{k',j}
 \reff{eq:geometric:j}\le \Alpha_{\ell'}^{k',\j} + \Alpha_{\ell'}^{k',j'}
 \reff{eq6:ell-k-j}\lesssim \Alpha_{\ell'}^{k',j'}.
\end{align}
Summing up~\eqref{eq:sums}--\eqref{eq:sums_aux3}, we see that
\begin{align*}
 \boxed{\sum_{\substack{(\ell,k,j)\in\QQ \\ (\ell,k,j) > (\ell',k',j')}} \Alpha_\ell^{k,j}
 \lesssim \Alpha_{\ell'}^{k',j'}
 \quad\text{provided that }\underline\ell = \infty.}
\end{align*}

{\bf Step~5.} Suppose that $\underline\ell < \infty$ and $\k(\underline\ell) = \infty$, i.e., for the mesh $\TT_{\underline\ell}$, the linearization loop does not terminate, and, moreover, $\ell' < \underline\ell$. Then, it holds that
\begin{align}\label{eq:linconv_aux4}
\hspace{-4pt}\sum_{\substack{(\ell,k,j)\in\QQ \\ (\ell,k,j)>(\ell',k',j')}} \!\!\!\!\! \Alpha_\ell^{k,j}
 \lesssim \sum_{k=1}^\infty\sum_{j=1}^{\j(\underline\ell,k)}\Alpha_{\underline\ell}^{k,j}
 +\sum_{\ell=\ell'+1}^{\underline\ell-1}\sum_{k=1}^{\k(\ell)}\sum_{j=1}^{\j(\ell,k)}\Alpha_{\ell}^{k,j}
 + \sum_{k=k'+1}^{\k(\ell')}\sum_{j=1}^{\j(\ell',k)}\Alpha_{\ell'}^{k,j}
 + \sum_{j=j'+1}^{\j(\ell',k')}\Alpha_{\ell'}^{k',j} .
\end{align}
We argue as before to see that
\begin{align}\label{eq:linconv_aux4bis}
 \hspace{-4pt}\sum_{\ell=\ell'+1}^{\underline\ell-1}\sum_{k=1}^{\k(\ell)}\sum_{j=1}^{\j(\ell,k)}\Alpha_{\ell}^{k,j}
 \reff{eq:sums_aux1}\lesssim \Alpha_{\ell'}^{k',j'},
 \quad
 \sum_{k=k'+1}^{\k(\ell')}\sum_{j=1}^{\j(\ell',k)}\Alpha_{\ell'}^{k,j}
 \reff{eq:sums_aux2}\lesssim \Alpha_{\ell'}^{k',j'},
 \quad \text{and} 
 \sum_{j=j'+1}^{\j(\ell',k')}\Alpha_{\ell'}^{k',j}
 \reff{eq:sums_aux3}\lesssim \Alpha_{\ell'}^{k',j'}.
\end{align}
It only remains to estimate
\begin{align}\label{eq:sums_aux5}
\begin{split}
\sum_{k=1}^\infty\sum_{j=1}^{\j(\underline\ell,k)}\Alpha_{\underline\ell}^{k,j}
		&\reff{eq:geometric:j}\lesssim \sum_{k=1}^{\infty}\big(\Alpha_{\underline\ell}^{k,\j} + \enorm{u_{\underline\ell}^{k,\star}-u_{\underline\ell}^{k,0}} \big)
		\reff{eq2:ell-k-j}\lesssim \Alpha_{\underline\ell}^{0,\j} + \sum_{k=1}^\infty\Alpha_{\underline\ell}^{k,\j}
		\reff{eq:geometric:k}\lesssim \Alpha_{\underline\ell}^{0,\j}
\\&
\reff{eq3:ell-k-j}\lesssim \Alpha_{\underline\ell-1}^{\k,\j}
		\le \Alpha_{\ell'}^{\k,\j} + \sum_{\ell=\ell'+1}^{\underline{\ell}-1}\Alpha_\ell^{\k,\j}
		\reff{eq:geometric:ell}\lesssim \Alpha_{\ell'}^{\k,\j}
		\reff{eq5:ell-k-j}\lesssim \Alpha_{\ell'}^{k',\j}
		\reff{eq6:ell-k-j}\lesssim \Alpha_{\ell'}^{k',j'}.
\end{split}
\end{align}
Altogether, we hence obtain that
\begin{align*}
 \boxed{\sum_{\substack{(\ell,k,j)\in\QQ \\ (\ell,k,j) > (\ell',k',j')}} \Alpha_\ell^{k,j}
 \lesssim \Alpha_{\ell'}^{k',j'}
 \quad\text{provided that }\ell' < \underline\ell < \infty \text{ and } \k(\underline\ell) = \infty.}
\end{align*}

{\bf Step~6.} Suppose that $\underline\ell < \infty$ and $\k(\underline \ell) = \infty$, i.e., for the mesh $\TT_{\underline \ell}$, the linearization loop does not terminate, and moreover, $\ell' = \underline\ell$. Arguing as in~\eqref{eq:sums_aux5} and~\eqref{eq:sums_aux3}, it holds that
\begin{align}\label{eq:linconv_aux1}
\boxed{\sum_{\substack{(\ell,k,j)\in\QQ \\ (\ell,k,j) > (\ell',k',j')}} \Alpha_\ell^{k,j} \lesssim \sum_{k=k'+1}^\infty\sum_{j=1}^{\j(\ell',k)}\Alpha_{\ell'}^{k,j} + \sum_{j=j'+1}^{\j(\ell',k')}\Alpha_{\ell'}^{k',j}
\lesssim \Alpha_{\ell'}^{k',j'}.}
\end{align}

{\bf Step~7.} Suppose that $\underline\ell < \infty$, where $\k(\underline\ell) < \infty$ and hence $\j(\underline\ell,\k) = \infty$, i.e., the linear solver does not terminate for the linearization step $\k(\underline\ell)$. Suppose moreover $\ell' < \underline\ell$.
Then, it holds that
\begin{align}\label{eq:alg_solv_inf_1}
\begin{split}
\sum_{\substack{(\ell,k,j)\in\QQ \\ (\ell,k,j) > (\ell',k',j')}} \Alpha_\ell^{k,j}
		\lesssim 
				\sum_{j=1}^\infty\Alpha_{\underline\ell}^{\k,j}
				+ \sum_{k=1}^{\k(\underline\ell)-1}\sum_{j=1}^{\j(\underline\ell,k)}\Alpha_{\underline\ell}^{k,j}
&+\sum_{\ell=\ell'+1}^{\underline\ell-1}\sum_{k=1}^{\k(\ell)}\sum_{j=1}^{\j(\ell,k)}\Alpha_\ell^{k,j}\\
&+ \sum_{k=k'+1}^{\k(\ell')}\sum_{j=1}^{\j(\ell',k)}\Alpha_{\ell'}^{k,j}
				+ \sum_{j=j'+1}^{\j(\ell',k')}\Alpha_{\ell'}^{k',j}.
\end{split}
\end{align}
We argue as before to see that
\begin{align*}
\sum_{\ell=\ell'+1}^{\underline\ell-1}\sum_{k=1}^{\k(\ell)}\sum_{j=1}^{\j(\ell,k)}\Alpha_{\ell}^{k,j}
			\reff{eq:sums_aux1}\lesssim \Alpha_{\ell'}^{k',j'},
		\quad
 		\sum_{k=k'+1}^{\k(\ell')}\sum_{j=1}^{\j(\ell',k)}\Alpha_{\ell'}^{k,j}
 			\reff{eq:sums_aux2}\lesssim \Alpha_{\ell'}^{k',j'},
 		\quad \text{and} \quad
 		\sum_{j=j'+1}^{\j(\ell',k')}\Alpha_{\ell'}^{k',j}
 			\reff{eq:sums_aux3}\lesssim \Alpha_{\ell'}^{k',j'}.
\end{align*}
For the first sum in~\eqref{eq:alg_solv_inf_1}, we get that
\begin{align}\label{eq:alg_solv_inf_1_aux1}
\sum_{j=1}^\infty\Alpha_{\underline\ell}^{\k,j}
		\reff{eq:geometric:j}\lesssim \enorm{u_{\underline\ell}^{\k,\star}-u_{\underline\ell}^{\k,0}}
		\reff{eq2:ell-k-j}\lesssim\Alpha_{\underline\ell}^{\k-1,\j}
		\reff{eq:sums_aux1}{\lesssim} \Alpha_{\ell'}^{k',j'}.
\end{align}
Hence, it only remains to estimate to estimate the second sum in~\eqref{eq:alg_solv_inf_1}, which can be treated analogously to~\eqref{eq:sums_aux5} in Step~5. This proves that
\begin{align*}
\sum_{k=1}^{\k(\underline\ell)-1}\sum_{j=1}^{\j(\underline\ell,k)}\Alpha_{\underline\ell}^{k,j}
		\reff{eq:sums_aux5}\lesssim\Alpha_{\ell'}^{k',j'}.
\end{align*}

Altogether, we obtain that
\begin{align*}
 \boxed{\sum_{\substack{(\ell,k,j)\in\QQ \\ (\ell,k,j) > (\ell',k',j')}} \Alpha_\ell^{k,j}
 		\lesssim \Alpha_{\ell'}^{k',j'}
 		\quad\text{provided that }\ell' < \underline\ell < \infty, \, \k(\underline\ell) < \infty, \text{ and }\j(\underline\ell,\k)=\infty.}
\end{align*}

{\bf Step~8.} Suppose that $\underline\ell < \infty$, where $\k(\underline\ell) < \infty$ and hence $\j(\underline\ell,\k) = \infty$, i.e., the linear solver does not terminate for the linearization step $\k(\underline\ell)$. Suppose moreover $\ell' = \underline\ell$ but $k' < \k(\ell')$.
Then, it holds that
\begin{align}\label{eq:alg_solv_inf_2}
\begin{split}
\sum_{\substack{(\ell,k,j)\in\QQ \\ (\ell,k,j) > (\ell',k',j')}} \Alpha_\ell^{k,j}
		\lesssim \sum_{j=1}^\infty\Alpha_{\ell'}^{\k,j}
                + \sum_{k=k'+1}^{\k(\ell')-1}\sum_{j=1}^{\j(\ell',k)}\Alpha_{\ell'}^{k,j}
				+ \sum_{j=j'+1}^{\j(\ell',k')}\Alpha_{\ell'}^{k',j}.
\end{split}
\end{align}
We argue as before to see that
\begin{align*}
 		\sum_{j=1}^{\infty}\Alpha_{\ell'}^{\k,j}
 			\reff{eq:alg_solv_inf_1_aux1}\lesssim\Alpha_{\ell'}^{k',j'},
 		\quad
 \sum_{k=k'+1}^{\k(\ell')-1}\sum_{j=1}^{\j(\ell',k)}\Alpha_{\ell'}^{k,j}
 			\reff{eq:sums_aux2}\lesssim \Alpha_{\ell'}^{k',j'},
 		\quad \text{and} \quad
 		\sum_{j=j'+1}^{\j(\ell',k')}\Alpha_{\ell'}^{k',j}
 			\reff{eq:sums_aux3}\lesssim \Alpha_{\ell'}^{k',j'}.
\end{align*}

Hence, we obtain that
\begin{align*}
 \boxed{\sum_{\substack{(\ell,k,j)\in\QQ \\ (\ell,k,j) > (\ell',k',j')}} \Alpha_\ell^{k,j}
 		\lesssim \Alpha_{\ell'}^{k',j'}
 		\quad\text{provided that }\ell' = \underline\ell < \infty, \, k'<\k(\ell') < \infty \text{, and }\j(\ell',\k)=\infty.}
\end{align*}

{\bf Step~9.} Suppose that $\underline\ell < \infty$, where $\k(\underline\ell) < \infty$ and hence $\j(\underline\ell,\k) = \infty$, i.e., the linear solver does not terminate for the linearization step $\k(\underline\ell)$. Suppose $\ell' = \underline\ell$ and $k' = \k(\ell')$.
 Then, it holds that
\begin{align}
\begin{split}
\boxed{\sum_{\substack{(\ell,k,j)\in\QQ \\ (\ell,k,j) > (\ell',k',j')}} \Alpha_\ell^{k,j}
		=  \sum_{j=j'+1}^{\infty}\Alpha_{\ell'}^{k',j}
			\reff{eq:geometric:j}\lesssim\Alpha_{\ell'}^{k',j'}.}
\end{split}
\end{align}

{\bf Step~10.} Suppose that $\underline\ell, \k(\underline\ell), \j(\underline\ell,\k(\underline\ell)) < \infty$, so that Algorithm~\ref{algorithm} finished on step~(iii) when $\eta_{\underline\ell}(u_{\underline\ell}^{\k,\j}) = 0$. From~\eqref{eq:reliability}, we see that $\eta_{\underline\ell}(u_{\underline\ell}^{\k,\j}) = 0$ implies $u^\star = u_{\underline\ell}^{\k,\j}$, i.e., the exact solution was found. Moreover, through the stopping criteria~\eqref{eq:st_crit_pic} and~\eqref{eq:st_crit_pcg}, we see that $u_{\underline\ell}^{\k-1,\j} = u_{\underline\ell}^{\k,\j-1} = u_{\underline\ell}^{\k,\j}$, so that~\eqref{eq2:picard:contraction} gives $u_{\underline\ell}^\star = u_{\underline\ell}^{\k,\j}$, and finally~\eqref{eq:banach_non_pert} gives $u_{\underline\ell}^{\k,\star} = u_{\underline\ell}^{\k,\j}$. Thus $\Alpha_{\underline\ell}^{\k,\j} = 0$.

Let $\ell' < \underline \ell$. Then, as in~\eqref{eq:linconv_aux4},
\begin{align*}
\hspace{-4pt}\sum_{\substack{(\ell,k,j)\in\QQ \\ (\ell,k,j)>(\ell',k',j')}} \!\!\!\!\! \Alpha_\ell^{k,j}
 \lesssim \sum_{k=1}^{\k(\underline\ell)}\sum_{j=1}^{\j(\underline\ell,k)}\Alpha_{\underline\ell}^{k,j}
 +\sum_{\ell=\ell'+1}^{\underline\ell-1}\sum_{k=1}^{\k(\ell)}\sum_{j=1}^{\j(\ell,k)}\Alpha_{\ell}^{k,j}
 + \sum_{k=k'+1}^{\k(\ell')}\sum_{j=1}^{\j(\ell',k)}\Alpha_{\ell'}^{k,j}
 + \sum_{j=j'+1}^{\j(\ell',k')}\Alpha_{\ell'}^{k',j}.
\end{align*}
Here, the last three terms are estimated as in~\eqref{eq:linconv_aux4bis}, whereas for the first one, we can proceed as in~\eqref{eq:sums_aux5}, crucially noting that the last summand $\Alpha_{\underline\ell}^{\k,\j}$ is zero.

If $\ell' = \underline \ell$, three cases are possible. The first case is $k' < \k$. Then
\begin{align*}
\sum_{\substack{(\ell,k,j)\in\QQ \\ (\ell,k,j)>(\ell',k',j')}} \!\!\!\!\! \Alpha_\ell^{k,j}
 \lesssim \sum_{k=k'+1}^{\k(\ell')}\sum_{j=1}^{\j(\ell',k)}\Alpha_{\ell'}^{k,j}
 + \sum_{j=j'+1}^{\j(\ell',k')}\Alpha_{\ell'}^{k',j},
\end{align*}
which is controlled as in~\eqref{eq:linconv_aux4bis}. The second case is $k' = \k$ but $j' < \j$, where directly
\begin{align*}
\sum_{\substack{(\ell,k,j)\in\QQ \\ (\ell,k,j)>(\ell',k',j')}} \!\!\!\!\! \Alpha_\ell^{k,j}
 \leq \sum_{j=j'+1}^{\j(\ell',k')}\Alpha_{\ell'}^{k',j} \reff{eq:geometric:j}\lesssim\Alpha_{\ell'}^{k',j'},
\end{align*}
again using $\Alpha_{\ell'}^{k',\j} = 0$. Finally, in the third case, $k' = \k$ and $j' = \j$, but then the sum is void. Altogether
\begin{align}
\begin{split}
\boxed{\sum_{\substack{(\ell,k,j)\in\QQ \\ (\ell,k,j) > (\ell',k',j')}} \Alpha_\ell^{k,j}
		\lesssim\Alpha_{\ell'}^{k',j'}}
\end{split}
\end{align}
also holds in this case.

\bigskip

{\bf Step~11.} Combining Steps~4--10 that cover all possible runs of Algorithm~\ref{algorithm} with Step~1, we finally see that
\begin{align*}
 \sum_{\substack{(\ell,k,j)\in\QQ \\ (\ell,k,j) > (\ell',k',j')}} \Delta_\ell^{k,j}
 \reff{eq:Alpha-Delta}\simeq
 \sum_{\substack{(\ell,k,j)\in\QQ \\ (\ell,k,j) > (\ell',k',j')}} \Alpha_\ell^{k,j}
 \lesssim \Alpha_{\ell'}^{k',j'}
 \reff{eq:Alpha-Delta}\simeq \Delta_{\ell'}^{k',j'}
 \quad \text{for all } (\ell',k',j') \in \QQ.
\end{align*}
This concludes the proof of~\eqref{eq_ell-k-j}.
\end{proof}

\begin{proof}[{\bfseries Proof of Theorem~\ref{theorem:linconv}}]
The proof is split into two steps.

{\bf Step~1.} For the convenience of the reader, we recall an argument from the proof of~\cite[Lemma 4.9]{axioms}: For $M\in\N\cup\{\infty\}$, let $C > 0$ and $\alpha_n \ge 0$ satisfy that
\begin{align*}
 \sum_{n = N + 1}^M \alpha_n \le C \, \alpha_N
 \quad \text{for all } N \in \N_0 \text{ with } N<\min\{M,\infty\}.
\end{align*}
Then,
\begin{align*}
 (1+C^{-1}) \, \sum_{n = N + 1}^M \alpha_n \le \sum_{n = N + 1}^M \alpha_n + \alpha_N = \sum_{n = N}^{M} \alpha_n
 \quad \text{for all } N \in \N_0.
\end{align*}
Inductively, it follows for all $N, m \in \N_0$ with $N+m<\min\{M+1,\infty\}$ that
\begin{align*}
 (1+C^{-1})^m \, \sum_{n = N + m}^M \alpha_n \le \sum_{n = N + 1}^M \alpha_n + \alpha_N = \sum_{n = N}^M \alpha_n.
\end{align*}
We thus conclude for all $N, m \in \N_0$ with $N+m<\min\{M+1,\infty\}$ that
\begin{align*}
 \alpha_{N+m} \le \sum_{n = N + m}^M \alpha_n \le (1 + C^{-1})^{-m} \sum_{n = N}^M \alpha_n
 \le (1 + C) \, (1 + C^{-1})^{-m} \alpha_N.
\end{align*}

{\bf Step~2.} Since the index set $\QQ$ is linearly ordered with respect to the total step counter $|(\cdot,\cdot,\cdot)|$, Lemma~\ref{lemma:ell-k-j} and Step~1 imply that
\begin{align*}
\Delta_{\ell'}^{k',j'}\leq\Clin\,\qlin^{|(\ell',k',j')|-|(\ell,k,j)|}\,\Delta_\ell^{k,j}\quad\textrm{for all }(\ell,k,j),(\ell',k',j')\in\QQ\textrm{ with }(\ell',k',j') \ge (\ell,k,j),
\end{align*}
where $\Clin = 1 + \Csum$ and $\qlin = \Csum / (\Csum + 1)$.
This concludes the proof.
\end{proof}

\section{Proof of Theorem~\ref{theorem:rate} (optimal decay rate wrt.\ degrees of freedom)}
\label{section:optimal_rate_DoFs}

The first result of this section proves the left inequality in~\eqref{eq:opt_rate}:

\begin{lemma}\label{lemma:opt_conv:lower_bound}
	Suppose \eqref{axiom:sons} as well as \eqref{axiom:stability}, \eqref{axiom:reduction}, and \eqref{axiom:discrete_reliability}. Let $s>0$ and assume $\norm{u^\star}{\mathbb{A}_s}>0$. Then, it holds that
	\begin{align}
		\norm{u^\star}{\mathbb{A}_s} \le \copt\sup_{(\ell',k',j')\in\QQ} (\#\TT_{\ell'}-\#\TT_0+1)^s\Delta_{\ell'}^{k',j'},
	\end{align}
	where the constant $\copt>0$ depends only on $\Ccea=L/\alpha$, $\Cstab$, $\Crel$, $\Cson$, $\#\TT_0$, $s$, and, if $\underline\ell<\infty$, additionally on $\underline\ell$.
\end{lemma}

\begin{proof}
	The proof is split into three steps.
	First, we recall from~\cite[Lemma~22]{bhp2017} that
	\begin{align}\label{eq:TT0}
		\#\TT_\fine / \#\TT_\coarse
		\le \#\TT_\fine - \#\TT_\coarse + 1
		\le \#\TT_\fine
		\quad \text{for all } \TT_\coarse \in \T \text{ and all } \TT_\fine \in \refine(\TT_\coarse).
	\end{align}

	{\bf Step~1.}

	Let $\underline\ell<\infty$ and $\k(\underline\ell)<\infty$ but $\j(\underline\ell,\k)=\infty$, i.e., the algebraic solver does not stop. According to Theorem~\ref{theorem:linconv}, it holds that
	\begin{align*}
		\Delta_{\underline\ell}^{\k,j}
		= \enorm{u^\star-u_{\underline\ell}^{\k,j}}
		+ \enorm{u_{\underline\ell}^{\k,\star}-u_{\underline\ell}^{\k,j}}
		+ \eta_{\underline\ell}(u_{\underline\ell}^{\k,j})
		\to 0
		\quad\text{as}\quad j\to\infty.
	\end{align*}
	Due to the uniqueness of the limit and the  C\'ea lemma~\eqref{eq:cea}, we obtain that $u^\star=u_{\underline\ell}^\star=u_{\underline\ell}^{\k,\star}$. From stability~\eqref{axiom:stability}, it follows that
	\begin{align*}
		0 \le \eta_{\underline\ell}(u_{\underline\ell}^{\k,\star})
		\reff{axiom:stability}{\lesssim} \eta_{\underline\ell}(u_{\underline\ell}^{\k,j}) + \enorm{u_{\underline\ell}^{\k,\star}-u_{\underline\ell}^{\k,j}}
		\to 0
		\quad\text{as}\quad j\to \infty.
	\end{align*}
	Hence, we see that $\eta_{\underline\ell}(u_{\underline\ell}^\star)=\eta_{\underline\ell}(u_{\underline\ell}^{\k,\star})=0$.

	For the last case, let $\underline\ell<\infty$ and $\k(\underline\ell)=\infty$, i.e., the linearization solver does not stop. Analogously to the previous case, we obtain that
	\begin{align*}
		\Delta_{\underline\ell}^{k,\j}
		= \enorm{u^\star-u_{\underline\ell}^{k,\j}}
		+ \enorm{u_{\underline\ell}^{k,\star}-u_{\underline\ell}^{k,\j}}
		+ \eta_{\underline\ell}(u_{\underline\ell}^{k,\j})
		\to 0
		\quad\text{as}\quad k\to\infty.
	\end{align*}
	With the C\'ea lemma~\eqref{eq:cea}, this leads to
	\begin{align*}
		0 \le \enorm{u_{\underline\ell}^\star-u_{\underline\ell}^{k,\j}}
		\reff{eq:cea}{\le} (1+\Ccea)\enorm{u^\star-u_{\underline\ell}^{k,\j}}
		\to 0
		\quad\text{as}\quad k\to\infty.
	\end{align*}
	Hence, we get that $u^\star = u_{\underline\ell}^\star$. Again, stability~\eqref{axiom:stability} yields that $\eta_{\underline\ell}(u_{\underline\ell}^\star)=0$.

	This implies in any case that $\enorm{u^\star-u_{\underline\ell}^\star} + \eta_{\underline\ell}(u_{\underline\ell}^\star)=0$ and hence that
	\begin{align*}
		\norm{u^\star}{\mathbb{A}_s}
		= \sup_{0\le N<\#\TT_{\underline\ell}-\#\TT_0}\Big(
		(N+1)^s\inf_{\TT_{\rm opt}\in\T(N)}\big[
		\enorm{u^\star-u_{\rm opt}^\star}+\eta_{\rm opt}(u_{\rm opt}^\star)
		\big]\Big)
	\end{align*}
	The term $N+1$ within the supremum can be estimated by
	\begin{align*}
		N+1\le \#\TT_{\underline\ell}-\#\TT_0\reff{axiom:sons}
		\le (\Cson^{\underline\ell}-1)\,\#\TT_0.
	\end{align*}
	The C\'ea lemma \eqref{eq:cea} and~\eqref{axiom:stability}, \eqref{axiom:reduction}, and \eqref{axiom:discrete_reliability} give that $\enorm{u^\star-u_{\rm opt}^\star}\lesssim \enorm{u^\star-u_0^\star}$ and  $\eta_{\rm opt}(u^\star_{\rm opt})\lesssim \eta_0(u^\star_0)$ (see, e.g.,  \cite[Lemma~3.5]{axioms}).
	Altogether, we thus arrive at
	\begin{align}\label{eq:upper bound for As1}
		\norm{u^\star}{\A_s}
		\lesssim
		\enorm{u^\star - u_0^\star} + \eta_0(u_0^\star).
	\end{align}

	{\bf Step~2.}
	We consider the generic case that $\underline\ell = \infty$ and $\eta_{\ell}(u_{\ell}^{\k,\j})>0$ for all $\ell\in\N_0$.
	Algorithm~\ref{algorithm} then guarantees that $\#\TT_\ell\to\infty$ as $\ell\to\infty$.
	Thus, we can argue analogously to the proof of~\cite[Theorem~4.1]{axioms}:
	Let $N \in \N$.
	Choose the maximal $\ell' \in \N_0$ such that
	$ \#\TT_{\ell'} - \#\TT_0 + 1 \le N$.
	Then, $\TT_{\ell'} \in \T(N)$.
	The choice of $N$ guarantees that
	\begin{align}
		\label{eq:upper bound for As2}
		N+1\le \#\TT_{\ell'+1} - \#\TT_0 + 1
		\reff{eq:TT0}\le \#\TT_{\ell'+1}
		\le \Cson \#\TT_{\ell'}
		\reff{eq:TT0}\le \Cson \#\TT_0 \, ( \#\TT_{\ell'} - \#\TT_0 + 1 ).
	\end{align}
	This leads to
	$$
		(N+1)^s\inf_{\TT_{\rm opt}\in\T(N)}\big[\enorm{u^\star-u_{\rm opt}^\star}+\eta_{\rm opt}(u_{\rm opt}^\star)\big]
		\lesssim (\#\TT_{\ell'} - \#\TT_0 + 1)^s \big[ \enorm{u^\star - u_{\ell'}^\star} + \eta_{\ell'}(u_{\ell'}^\star) \big],
	$$%
	and we immediately see that this also holds for $N=0$ with $\ell'=0$.
	Taking the supremum over all $N \in \N_0$, we conclude that
	\begin{align}
		\label{eq:upper bound for As3}
		\norm{u^\star}{\A_s}
		\lesssim \sup_{{\ell'} \in \N_0}(\#\TT_{\ell'} - \#\TT_0 + 1)^s \big[ \enorm{u^\star - u_{\ell'}^\star} + \eta_{\ell'}(u_{\ell'}^\star) \big].
	\end{align}

	{\bf Step~3.} With stability~\eqref{axiom:stability} and the C\'ea lemma~\eqref{eq:cea}, we see for all $(\ell',0,0)\in\QQ$ that
	\begin{align*}
		 & \enorm{u^\star - u_{\ell'}^\star} + \eta_{\ell'}(u_{\ell'}^\star)
		\reff{axiom:stability} \lesssim \enorm{u^\star - u_{\ell'}^\star} + \enorm{u_{\ell'}^\star - u_{\ell'}^{0,0}} + \eta_{\ell'}(u_{\ell'}^{0,0})
		\\ &\quad
		\le 2 \, \enorm{u^\star - u_{\ell'}^\star} + \enorm{u^\star - u_{\ell'}^{0,0}} + \eta_{\ell'}(u_{\ell'}^{0,0})
		\reff{eq:cea} \lesssim \enorm{u^\star - u_{\ell'}^{0,0}} + \eta_{\ell'}(u_{\ell'}^{0,0})
		\le \Delta_{\ell'}^{0,0}.
	\end{align*}
	With \eqref{eq:upper bound for As1} and \eqref{eq:upper bound for As3}, we thus obtain  that
	$$
		\norm{u^\star}{\A_s}
		\lesssim \sup_{(\ell',0,0) \in \QQ} (\#\TT_{\ell'} - \#\TT_0 + 1)^s \, \big[\enorm{u^\star-u_{\ell'}^\star} +\eta_{\ell'}(u_{\ell'}^\star)\big]
		\le \sup_{(\ell', k',j') \in \QQ} (\#\TT_{\ell'} - \#\TT_0 + 1)^s \, \Delta_{\ell'}^{k',j'}.
	$$
	This concludes the proof.
\end{proof}

To prove the upper estimate in~\eqref{eq:opt_rate}, we need the comparison lemma from~\cite[Lemma~4.14]{axioms} for the error estimator of the exact discrete solution $u_\ell^\star\in\XX_\ell$.
\begin{lemma}\label{lemma:comparison}
	Suppose \eqref{axiom:sons}--\eqref{axiom:overlay} as well as \eqref{axiom:stability}, \eqref{axiom:reduction}, and \eqref{axiom:discrete_reliability}. Let $0<\theta'<\theta_{\rm opt}:=(1 + \Cstab^2\Crel^2)^{-1}$. Then, there exist constants $C_1,C_2>0$ such that for all $s>0$ with $0<\norm{u^\star}{\mathbb{A}_s}<\infty$ and all $\TT_H\in\T$, there exists $\RR_H\subseteq\TT_H$ which satisfies
	\begin{align}\label{eq:comparison_lemma0}
		\#\RR_H \le C_1C_2^{-1/s}\norm{u^\star}{\mathbb{A}_s}^{1/s} \, \eta_H(u_H^\star)^{-1/s},
	\end{align}
	as well as the D\"orfler marking criterion
	\begin{align}\label{eq:comparison_lemma}
		\theta'\eta_H(u_H^\star) \le \eta_H(\RR_H,u_H^\star).
	\end{align}
	The constants $C_1,C_2$ depend only on $\Cstab$ and $\Crel$. \qed
\end{lemma}


\begin{proof}[{\bfseries Proof of Theorem~\ref{theorem:rate}}]The proof is split into four steps. 	Without loss of generality, we may assume that $\norm{u^\star}{\mathbb{A}_s}<\infty$.

	{\bf Step~1.}\quad
	Due to the assumptions $\lpcg + \lpcg/\lpic\le\lpcg^\star$ (from Lemma~\ref{lemma:picard:contraction}) and $\lpic/\theta < \lpic^\star$ (from Lemma~\ref{lemma:banach2:equi_linconv}), we get that $\lpcg\le\lpcg^\star\,\lpic\le\lpcg^\star\,\lpic^\star\,\theta$. Hence, it holds that
	\begin{align*}
		\theta' & =  \frac{\theta+\Cstab\Big( (1+\Cpic)\Cpcg\lpcg + \big[\Cpic + (1+\Cpic)\Cpcg\lpcg\big]\lpic\Big)}{1-\lpic\,/\lpic^\star}                                                \\
		        & \le  \frac{\theta+\Cstab\Big( (1+\Cpic)\Cpcg\lpcg^\star\lpic^\star\theta + \big[\Cpic + (1+\Cpic)\Cpcg\lpcg^\star\lpic^\star\theta\big]\lpic^\star\theta\Big)}{1-\theta}
	\end{align*}
	which converges to $0$ as $\theta\to 0$. As a consequence, \eqref{eq:opt:theta'} holds for sufficiently small $\theta$.

	Clearly, the parameters $\lpcg,\lpic,\theta>0$ can be chosen such that all assumptions are fulfilled. First, choose $\theta>0$ such that $0<\theta<\min\{1,\theta^\star\}$. Then, choose $\lpic>0$ such that $0<\lpic/\theta<\lpic^\star$. Finally, choose $0<\lpcg<1$ such that $\lpcg + \lpcg/\lpic<\lpcg^\star$.


	{\bf Step~2.}\quad
	Recall that $\Cpic=\qpic/(1-\qpic)$ and $\Cpcg=\qsolve/(1-\qsolve)$. Provided that $(\ell+1,0,0)\in\QQ$, it follows from the contraction properties~\eqref{eq2:pcg:contraction} resp.~\eqref{eq2:picard:contraction0}, and the stopping criteria~\eqref{eq:pcg:stopping} resp.~\eqref{eq:picard:stopping} that
	\begin{align*}
		 & \enorm{u_\ell^\star-u_\ell^{\k,\j}} \le \enorm{u_\ell^\star-u_\ell^{\k,\star}}  +  \enorm{u_\ell^{\k,\star}-u_\ell^{\k,\j}}                                              \\
		 & \quad\reffleft{eq2:picard:contraction0}{\le} \Cpic\,\enorm{u_\ell^{\k,\star}-u_\ell^{\k-1,\j}}  +  \enorm{u_\ell^{\k,\star}-u_\ell^{\k,\j}}                              \\
		 & \quad\le (1 + \Cpic) \enorm{u_\ell^{\k,\star}-u_\ell^{\k,\j}}  +  \Cpic\, \enorm{u_\ell^{\k,\j}-u_\ell^{\k-1,\j}}                                                        \\
		 & \quad\reffleft{eq2:pcg:contraction}{\le} (1 + \Cpic)\Cpcg \enorm{u_\ell^{\k,\j}-u_\ell^{\k,\j-1}}  +  \Cpic\, \enorm{u_\ell^{\k,\j}-u_\ell^{\k-1,\j}}                    \\
		 & \quad\reffleft{eq:pcg:stopping}{\le} (1 + \Cpic)\Cpcg\lpcg\,\eta_\ell(u_\ell^{\k,\j})  +  \big[\Cpic + (1+\Cpic)\Cpcg\lpcg\big]  \enorm{u_\ell^{\k,\j}-u_\ell^{\k-1,\j}} \\
		 & \quad\reffleft{eq:picard:stopping}{\le} \Big((1 + \Cpic)\Cpcg\lpcg + \big[\Cpic + (1+\Cpic)\Cpcg\lpcg\big]\,\lpic\Big)\eta_\ell(u_\ell^{\k,\j})                          \\
		 & \quad \reffleft{eq:opt:theta'}{=} \Cstab^{-1}\Big(\theta'\big(1-\lpic/\lpic^\star\big)  -  \theta\Big)\eta_\ell(u_\ell^{\k,\j}).
	\end{align*}

	{\bf Step~3.}\quad
	Let $\RR_\ell\subseteq\TT_\ell$ be the subset from Lemma~\ref{lemma:comparison} with $\theta'$ from~\eqref{eq:opt:theta'}. From Step~3, we obtain that
	\begin{align}\label{eq:opt_conv:aux1}
		\begin{split}
			\eta_\ell(\RR_\ell,u_\ell^\star) \,&\reffleft{axiom:stability}{\le}
			\eta_\ell(\RR_\ell,u_\ell^{\k,\j})  +  \Cstab\enorm{u_\ell^\star-u_\ell^{\k,\j}}\\
			&\le \eta_\ell(\RR_\ell,u_\ell^{\k,\j})  +  \Big(\theta'\big(1-\lpic/\lpic^\star\big)  -  \theta\Big)\eta_\ell(u_\ell^{\k,\j}).
		\end{split}
	\end{align}
	With the equivalence~\eqref{eq:eta-star}, Lemma~\ref{lemma:comparison}, and estimate~\eqref{eq:opt_conv:aux1}, we see that
	\begin{align*}
		 & \theta'\big(1-\lpic/\lpic^\star\big)\eta_\ell(u_\ell^{\k,\j})  \reff{eq:eta-star}{\le}  \theta'\eta_\ell(u_\ell^\star)  \reff{eq:comparison_lemma}{\le}  \eta_\ell(\RR_\ell,u_\ell^\star) \\
		 & \qquad\reff{eq:opt_conv:aux1}\le  \eta_\ell(\RR_\ell,u_\ell^{\k,\j})  +  \Big(\theta'\big(1-\lpic/\lpic^\star\big)  -  \theta\Big)\eta_\ell(u_\ell^{\k,\j}).
	\end{align*}
	Thus, we are led to
	\begin{align*}
		\theta\,\eta_\ell(u_\ell^{\k,\j})  \le  \eta_\ell(\RR_\ell,u_\ell^{\k,\j}).
	\end{align*}
	Hence, $\RR_\ell$ satisfies the D\"orfler marking criterion~\eqref{eq:doerfler} used in Algorithm~\ref{algorithm}. By the (quasi-)minimality of $\MM_\ell$ in~\eqref{eq:doerfler}, we infer that
	\begin{align*}
		\#\MM_\ell  \lesssim  \#\RR_\ell  \reff{eq:comparison_lemma0}{\lesssim}  \norm{u^\star}{\mathbb{A}_s}^{1/s}\,\eta_\ell(u_\ell^\star)^{-1/s}  \reff{eq:eta-star}{\simeq}  \norm{u^\star}{\mathbb{A}_s}^{1/s}\eta_\ell(u_\ell^{\k,\j})^{-1/s}.
	\end{align*}
	Recall from \eqref{eq:init:*} that $u_{\ell+1}^{0,\j}=u_\ell^{\k,\j}$. Thus, \eqref{eq3:ell-k-j} and the equivalence~\eqref{eq:Alpha-Delta} lead to
	\begin{align*}
		\eta_\ell(u_\ell^{\k,\j})^{-1/s}  \reff{eq3:ell-k-j}{\lesssim}  (\Alpha_{\ell+1}^{0,\j})^{-1/s} \reff{eq:Alpha-Delta}{\simeq}  (\Delta_{\ell+1}^{0,\j})^{-1/s}.
	\end{align*}
	Overall, we end up with
	\begin{align}\label{eq:opt_conv:aux2}
		\#\MM_\ell  \lesssim  \norm{u^\star}{\mathbb{A}_s}^{1/s}(\Delta_{\ell+1}^{0,\j})^{-1/s} \quad\text{for all }(\ell+1,0,0)\in\QQ.
	\end{align}
	The hidden constant depends only on $\Cstab$, $\Crel$, $\Cmark$, $1-\lpic/\lpic^\star$, $\Ccea=L/\alpha$, $\Crel'$ and $s$.

	{\bf Step~4.}\quad
	For $(\ell,k,j)\in\QQ$ such that $(\ell+1,0,0)\in\QQ$ and such that $\TT_{\ell}\neq\TT_0$, Step~4 and the closure estimate~\eqref{axiom:closure} lead to
	\begin{align*}
		\#\TT_\ell - \#\TT_0 + 1 & \simeq  \#\TT_\ell - \#\TT_0  \reff{axiom:closure}{\lesssim}  \sum_{\widetilde\ell=0}^{\ell-1}\#\MM_{\widetilde\ell}  \reff{eq:opt_conv:aux2}{\lesssim}  \norm{u^\star}{\mathbb{A}_s}^{1/s} \sum_{\widetilde\ell=0}^{\ell}(\Delta_{\widetilde\ell}^{0,\j})^{-1/s} \\
		                         & \le \norm{u^\star}{\mathbb{A}_s}^{1/s} \sum_{\substack{(\widetilde\ell,\widetilde{k},\widetilde{j})\in\QQ                                                                                                                                                         \\ (\widetilde\ell,\widetilde{k},\widetilde{j}) \le (\ell,k,j)}}(\Delta_{\widetilde{\ell}}^{\widetilde{k},\widetilde{j}})^{-1/s},
	\end{align*}
	where the hidden constant depends only on the constant of~\eqref{eq:opt_conv:aux2} and additionally on $\Cmesh$. Replacing $\norm{u^\star}{\mathbb{A}_s}$ with $\max\{\norm{u^\star}{\mathbb{A}_s},\Delta_0^{0,0}\}$, the overall estimate trivially holds for $\TT_\ell=\TT_0$. We thus get with linear convergence~\eqref{eq:linconv} and the geometric series (i.e., $\sum_{n=0}^\infty \qlin^n = 1/(1-\qlin)\lesssim 1$) that
	\begin{align}\label{eq:opt_conv:aux3}
		\begin{split}
			\#\TT_\ell - \#\TT_0 + 1
			&\lesssim  \max\{\norm{u^\star}{\mathbb{A}_s},\Delta_0^{0,0}\}^{1/s} \sum_{\substack{(\widetilde\ell,\widetilde{k},\widetilde{j})\in\QQ \\ (\widetilde\ell,\widetilde{k},\widetilde{j}) \le (\ell,k,j)}}(\Delta_{\widetilde{\ell}}^{\widetilde{k},\widetilde{j}})^{-1/s}\\
			&\reffleft{eq:linconv}{\lesssim}  \max\{\norm{u^\star}{\mathbb{A}_s},\Delta_0^{0,0}\}^{1/s} 	(\Delta_{\ell}^{k,j})^{-1/s} 	\sum_{\substack{(\widetilde\ell,\widetilde{k},\widetilde{j})\in\QQ \\ (\widetilde\ell,\widetilde{k},\widetilde{j}) \le (\ell,k,j)}}\qlin^{|(\ell,k,j)|-|(\widetilde{\ell},\widetilde{k},\widetilde{j})|}\\
			&\lesssim  \max\{\norm{u^\star}{\mathbb{A}_s},\Delta_0^{0,0}\}^{1/s} (\Delta_{\ell}^{k,j})^{-1/s},
		\end{split}
	\end{align}
	where the hidden constant depends only on $\Cstab$, $\Crel$, $\Cmark$, $1-\lpic/\lpic^\star$, $\Ccea=L/\alpha$, $\Crel'$, $\Cmesh$, $\Clin$, $\qlin$, and $s$. This proves that
	\begin{align}\label{eq_opt_conv_int}
		(\#\TT_{\ell} - \#\TT_0 + 1 )^s \Delta_{\ell}^{k,j}
		\lesssim \max\{\norm{u^\star}{\mathbb{A}_s},\Delta_0^{0,0}\}.
	\end{align}
	when $(\ell+1,0,0)\in\QQ$ and $\ell \geq 0$ as well as
	\begin{align}
		(\#\TT_{\ell} - \#\TT_0 + 1 )^s \Delta_{\ell}^{k,j}
		\lesssim \norm{u^\star}{\mathbb{A}_s}.	
	\end{align}
 when $(\ell+1,0,0)\in\QQ$ and $\ell\ge 1$.
 
	Let now $(\ell,k,j)\in\QQ$ with $\ell \geq 2$ but $(\ell+1,0,0)\not\in\QQ$, i.e., $\ell = \underline\ell < \infty$ and one of the cases discussed in detail in Step~1 of Lemma~\ref{lemma:opt_conv:lower_bound} arises. Since $\ell-1 \geq 1$ and $(\ell,0,0)\in\QQ$, \eqref{eq_opt_conv_int} shows that
	\begin{align*}
		(\#\TT_{\ell-1} - \#\TT_0 + 1 )^s \Delta_{\ell-1}^{\k,\j}
		\lesssim \norm{u^\star}{\mathbb{A}_s}.
	\end{align*}
	Moreover, Lemma~\ref{lemma:ell-k-j} leads to $\Delta_{\ell}^{k,j} \lesssim \Delta_{\ell-1}^{\k,\j}$. Therefore, we obtain from~\eqref{eq:upper bound for As2} that
	\begin{align}\label{eq_opt_conv_bis}
		\#\TT_{\ell} - \#\TT_0 + 1 \leq \Cson \#\TT_0 (\#\TT_{\ell-1} - \#\TT_0 + 1).
	\end{align}
	Altogether, \eqref{eq_opt_conv_int} holds for this case as well.

	As the next case, if $(\ell,k,j)\in\QQ$ with $\ell = \underline\ell = 1$, we can rely on the inequality
	\begin{align}\label{eq_opt_conv_quater}\begin{split}
			(\#\TT_{1} & - \#\TT_0 + 1 )^s \Delta_{1}^{k,j}
			\reff{eq_opt_conv_bis}{\leq} \Cson (\#\TT_0)\, \Delta_{1}^{k,j} \reff{eq_ell-k-j}{\lesssim} \Delta_{0}^{\k,\j}\\
			&\reff{eq:def:Delta}{=} \enorm{u^\star - u_0^{\k,\j}} + \enorm{u_0^{\k,\star} - u_0^{\k,\j}} + \eta_0(u_0^{\k,\j})\\
			& \reff{eq2:pcg:contraction}{\lesssim}
			\enorm{u^\star - u_0^\star} + \enorm{u_0^\star - u_0^{\k,\j}} + \enorm{u_0^{\k,\j} - u_0^{\k,\j-1}} + \eta_0(u_0^{\k,\j}) \\
			&\reff{eq:st_crit_pcg}{\lesssim} \enorm{u^\star - u_0^\star} + \enorm{u_0^\star - u_0^{\k,\j}} + \enorm{u_0^{\k,\j} - u_0^{\k-1,\j}} + \eta_0(u_0^{\k,\j})\\
			&\reff{eq2:picard:contraction}{\lesssim} \enorm{u^\star - u_0^\star} + \enorm{u_0^{\k,\j} - u_0^{\k-1,\j}} + \eta_0(u_0^{\k,\j})\\
			& \reff{eq:st_crit_pic}{\lesssim} \enorm{u^\star - u_0^\star} + \eta_0(u_0^{\k,\j})
			\reff{eq:eta-star}\lesssim \enorm{u^\star - u_0^\star} + \eta_0(u_0^\star)\le \norm{u^\star}{\A_s}.
		\end{split}\end{align}
	Thus, \eqref{eq_opt_conv_int} holds for this case as well.

As the final case, if $(\ell,k,j)\in\QQ$ with $\ell=\underline\ell=0$, we get with the linear convergence~\eqref{eq:linconv} that
\begin{align}
\Delta_0^{k,j}\reff{eq:linconv}{\lesssim}\Delta_0^{0,0}.
\end{align}
Hence, \eqref{eq_opt_conv_int} also holds for this case, and we conclude the proof of~\eqref{eq:opt_rate}

\end{proof}

\section{Proof of Theorem~\ref{theorem:cost} (optimal decay rate wrt.\ computational cost)}
\label{section:optimal_cost}

\begin{proof}[{\bfseries Proof of Theorem~\ref{theorem:cost}}]

Note that $\# \TT_{\ell'}-\#\TT_0 + 1 = 1 \leq \# \TT_0$ for $\ell'=0$ and $\# \TT_{\ell'}-\#\TT_0 + 1 \leq \# \TT_\ell'$ for $\ell'>0$, so that the left inequality in~\eqref{eq:opt_cost} immediately follows from the left inequality in~\eqref{eq:opt_rate}. In order to prove the right inequality in~\eqref{eq:opt_cost}, let $(\ell',k',j')\in\QQ$. Employing~\eqref{eq_opt_conv_int} from Step~5 of the proof of Theorem~\ref{theorem:rate}, the geometric series proves that
\begin{align*}
\sum_{\substack{(\ell,k,j)\in\QQ \\ (\ell,k,j) \le (\ell',k',j')}} \# \TT_\ell
		&\reff{eq:TT0}{\le} \#\TT_0  \sum_{\substack{(\ell,k,j)\in\QQ \\ (\ell,k,j) \le (\ell',k',j')}} (\# \TT_\ell-\#\TT_0 + 1)\\
		&\reff{eq_opt_conv_int}{\le} \#\TT_0	\,		\Copt^{1/s}		\,		 \max\{\norm{u^\star}{\mathbb{A}_s},\Delta_0^{0,0}\}^{1/s} \sum_{\substack{(\ell,k,j)\in\QQ \\ (\ell,k,j) \le (\ell',k',j')}} (\Delta_{\ell}^{k,j})^{-1/s}\\
		&\reff{eq:linconv}{\le} \#\TT_0	\,\Copt^{1/s}\,		\Clin^{1/s}\,		\frac{1}{1-\qlin^{1/s}}\,\max\{\norm{u^\star}{\mathbb{A}_s},\Delta_0^{0,0}\}^{1/s}(\Delta_{\ell'}^{k',j'})^{-1/s}.
\end{align*}
Rearranging this estimate, we end up with
\begin{align*}
\sup_{(\ell',k',j') \in \QQ} \bigg( \sum_{\substack{(\ell,k,j)\in\QQ, \, \ell \geq 1 \\ (\ell,k,j) \le (\ell',k',j')}} \# \TT_\ell \bigg)^s \Delta_{\ell'}^{k',j'}
				\lesssim \max\{\norm{u^\star}{\mathbb{A}_s},\Delta_0^{0,0}\},
\end{align*}
where the hidden constant depends only on $\Cstab$, $\Crel$, $\Cmark$, $1-\lpic/\lpic^\star$, $\Ccea=L/\alpha$, $\Crel'$, $\Cmesh$, $\Clin$, $\qlin$, $\#\TT_0$, and $s$. This proves the right inequality in~\eqref{eq:opt_cost}. 
\end{proof}

\section{Numerical experiments}
\label{section:num_exp}

In this section, we present numerical experiments in 2D to underpin our theoretical findings. We compare the performance of Algorithm~\ref{algorithm} for 
\begin{itemize}
	\item different values of $\lpcg \in \{10^{-1},10^{-2},10^{-3},10^{-4}\}$,
	\item different values of $\lpic \in \{1, 10^{-1},10^{-2},10^{-3},10^{-4}\}$,
	\item different values of $\theta \in \{0.1, 0.3,0.5,0.7,0.9,1\}$,
\end{itemize}
As model problems serve nonlinear boundary value problems which arise, e.g., from nonlinear material laws in magnetostatic computations, where the mesh-refinement is steered by newest vertex bisection.

As an algebraic solver for the linear problems arising from the Banach--Picard iteration, we use PCG with multilevel additive Schwarz preconditioner from~\cite[Section~7.4.1]{dissFuehrer} which is an optimal preconditioner, i.e., the condition number of the preconditioned system is uniformly bounded; cf. also~\cite[Section 2.9]{banach2}.

\subsection{Model problem}
\label{sec:numerics_model_problem}
With $\d \ge 2$, let $\Omega \subset \R^d$ be a bounded Lipschitz domain with polytopal boundary $\Gamma = \partial \Omega$. We suppose that the boundary $\Gamma$ is split into relatively open and disjoint Dirichlet and Neumann boundaries $\Gamma_D, \Gamma_N\subseteq\Gamma$ with $|\Gamma_D|>0$, i.e., $\Gamma=\overline\Gamma_D\cup\overline\Gamma_N$. While the numerical experiments in Section~\ref{section:example_known}--\ref{section:example_unknown} only consider $d=2$, we stress that the following model problem is covered by the abstract theory for any $d\ge 2$.
For a given right-hand side $f\in L^2(\Omega)$ and $g\in L^2(\Gamma)$, it reads as follows:
\begin{eqnarray}\label{eq:example}
\begin{split}
- \div( \mu(x,|\nabla u^\star(x)|^2) \nabla u^\star(x) ) &= f(x) \quad &\textrm{in }& \Omega, \\
u^\star(x) &= 0  &\textrm{on }& \Gamma_D,\\
\mu(x,|\nabla u^\star(x)|^2)\,\partial_{\bf n} u^\star(x) &= g(x)  &\textrm{on }& \Gamma_N,
\end{split}
\end{eqnarray}
where the scalar nonlinearity $\mu: \Omega \times \R_{\geq 0} \rightarrow \R$ satisfies the following properties~\eqref{item:mu_bounded}--\eqref{item:mu_lipschitz1}, similarly considered in~\cite{gmz2012,banach}:

\renewcommand{\theenumi}{M\arabic{enumi}}
\begin{enumerate}
 \bf
 \item \label{item:mu_bounded}
 \rm
 There exist constants $0<\gamma_1<\gamma_2<\infty$ such that
 \begin{align}\label{eq:mu_bounded}
 \gamma_1 \le \mu(x,t) \le \gamma_2 \quad \textrm{for all } x \in \Omega \text{ and all } t \geq 0.
 \end{align}
 
 \bf
 \item \label{item:mu_differentiable}
 \rm
There holds $\mu(x,\cdot) \in C^1(\R_{\geq 0} ,\R)$ for all $x \in \Omega$, and there exist constants $0 < \widetilde{\gamma}_1<\widetilde{\gamma}_2<\infty$ such that
 \begin{align}\label{eq:mu_differentiable}
 \widetilde{\gamma}_1 \le \mu(x,t) +2 t  \frac{\d{}}{\d{t}} \mu(x,t) \le \widetilde{\gamma}_2 \quad \textrm{for all } x \in \Omega \text{ and all } t \geq 0.
 \end{align}
  \bf
 \item \label{item:mu_lipschitz}
 \rm
 Lipschitz continuity of $\mu(x,t)$ in $x$, i.e., there exists a constant $L_{\mu}>0$ such that
 \begin{align}\label{eq:mu_lipschitz}
 | \mu(x,t) - \mu(y,t) | \le L_{\mu} | x - y | \quad \textrm{for all } x,y \in \Omega \text{ and all } t \geq 0.
 \end{align}
 
 \bf
 \item \label{item:mu_lipschitz1}
 \rm
Lipschitz continuity of $t \frac{\d{}}{\d{t}} \mu(x,t)$ in $x$, i.e., there exists a constant $\widetilde{L}_{\mu}>0$ such that
 \begin{align}\label{eq:mu_lipschitz1}
 | t \frac{\d{}}{\d{t}} \mu(x,t) - t \frac{\d{}}{\d{t}} \mu(y,t) | \le \widetilde{L}_{\mu} | x - y | \quad \textrm{for all } x,y \in \Omega
 \text{ and all } t \geq 0.
\end{align}
\end{enumerate}

\subsection{Weak formulation}
\label{sec:numerics_weak_formulation}

The weak formulation of~\eqref{eq:example} reads as follows: Find $u \in H^1_D(\Omega):= \{w \in H^1(\Omega): \, w=0 \text{ on } \Gamma_D \}$ such that
\begin{align}\label{eq:weak_example}
\int_{\Omega} \mu(x,|\nabla  u^\star(x)|^2) \,\nabla u^\star\cdot \nabla v \d{x} = \int_{\Omega} f v \d{x} + \int_{\Gamma_N} g v \d{s}
 \quad \textrm{for all } v \in 
H^1_D(\Omega).
\end{align}
With respect to the abstract framework of Section~\ref{section:abstract}, we take $\HH = H^{1}_D(\Omega)$, $\K=\R$, $\product{\cdot}{\cdot}=\product{\nabla\cdot}{\nabla\cdot}$ with $\enorm{v}= \norm{\nabla v}{L^2(\Omega)}$. We obtain~\eqref{eq:exact_solution} with operators
\begin{subequations}
\begin{align}
\dual{\AA w}{v}_{\HH'\times\HH} &= \int_{\Omega} \mu(x,|\nabla w(x)|^2) \, \nabla w(x)\cdot\nabla v(x) \d{x},  \label{eq:AG} \\
F(v) &= \int_{\Omega} f v \d{x} + \int_{\Gamma_N} g v \d{s} \label{eq:F}
\end{align}
\end{subequations}
for all $v,w\in\HH$.
We recall from~ \cite[Proposition~8.2]{banach} that \eqref{item:mu_bounded}--\eqref{item:mu_differentiable} implies that $\AA$ is strongly monotone (with $\alpha:=\widetilde\gamma_1$) and Lipschitz continuous (with $L:=\widetilde\gamma_2$), so that \eqref{eq:example} fits into the setting of Section~\ref{section:abstract}. Moreover,~\eqref{item:mu_lipschitz}--\eqref{item:mu_lipschitz1} are required to prove the well-posedness and the properties~\eqref{axiom:stability}--\eqref{axiom:discrete_reliability} of the residual {\sl a~posteriori} error estimator. 


\subsection{Discretization and \textsl{a~posteriori} error estimator} \label{example:estimator}
Let $\mathcal{T}_0$ be a conforming initial triangulation of $\Omega$ into simplices $\TT\in\TT_0$.
For each $\TT_H\in\T$, consider the lowest-order FEM space
\begin{align}\label{eq:dpr:XX}
\HH_{H} := \set{ v\in C(\Omega)}{ v|_\Gamma=0 \textrm{ and }v|_{T} \in \mathcal{P}^1(T) \textrm{ for all } T \in \mathcal{T}_{H}}.
\end{align}

As in \cite[Section 3.2]{gmz2012}, we define for all  $T \in \mathcal{T}_{H}$ and all $v_H \in \HH_H$, the corresponding weighted residual error indicators
\begin{align}\label{eq:dpr:eta} 
\begin{split}\eta_{H}(T,v_{H})^2 &:= |T|^{2/d} \norm{f+ \div ( \mu(\cdot,|\nabla v_{H}|^2) \nabla v_{H} )}{L^2(T)}^2 \\
&\qquad+|T|^{1/d} \norm{ [( \mu(\cdot,|\nabla v_{H}|^2)\nabla v_{H}) \cdot\textbf{\textrm{n}}] }{{L^2(\partial T\cap\Omega)}^2},
\end{split}
\end{align}
where $[\cdot]$ denotes the usual jump of discrete functions across element interfaces, and \textbf{n} is the outer normal vector of the considered element.

Due to~\eqref{item:mu_lipschitz}, the error estimator is well-posed, since the nonlinearity $\mu(x,t)$ is Lipschitz continuous in $x$. Then, reliability~\eqref{axiom:reliability} and discrete reliability~\eqref{axiom:discrete_reliability} are proved as in the linear case; see, e.g.,~\cite{ckns2008} for the linear case or \cite[Theorem~3.3]{gmz2012} and \cite[Theorem 3.4]{gmz2012}, respectively, for strongly monotone nonlinearities.

The verification of stability~\eqref{axiom:stability} and reduction~\eqref{axiom:reduction} requires the validity of a certain inverse estimate. For scalar nonlinearities and under the assumptions~\eqref{item:mu_bounded}--\eqref{item:mu_lipschitz1}, the latter is proved in~{\cite[Lemma 3.7]{gmz2012}}. Using this inverse estimate, the proof of~\eqref{axiom:stability} and~\eqref{axiom:reduction} 
follows as for the linear case; see, e.g.,~\cite{ckns2008} for the linear case or \cite[Section~3.3]{gmz2012} for scalar nonlinearities. We note that the necessary inverse estimate is, in particular, open for non-scalar nonlinearities. In any case, the arising constants in~\eqref{axiom:stability}--\eqref{axiom:discrete_reliability} depend also on the uniform shape regularity of the triangulations generated by newest vertex bisection.

\begin{figure}
  \centering
  \raisebox{-0.5\height}{\includegraphics[width=0.4\textwidth]{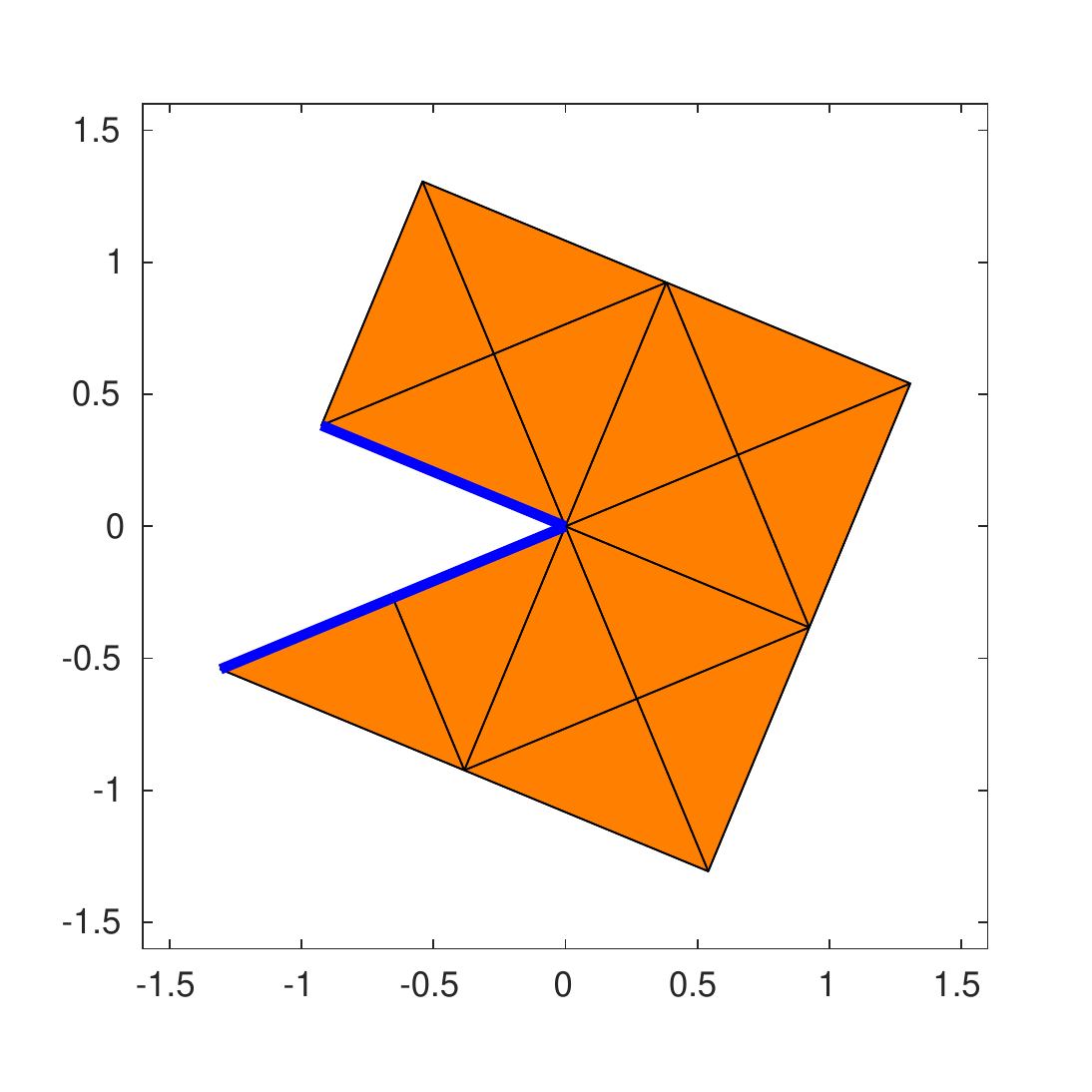}}
  \quad\quad
  \raisebox{-0.5\height}{\includegraphics[width=0.4\textwidth]{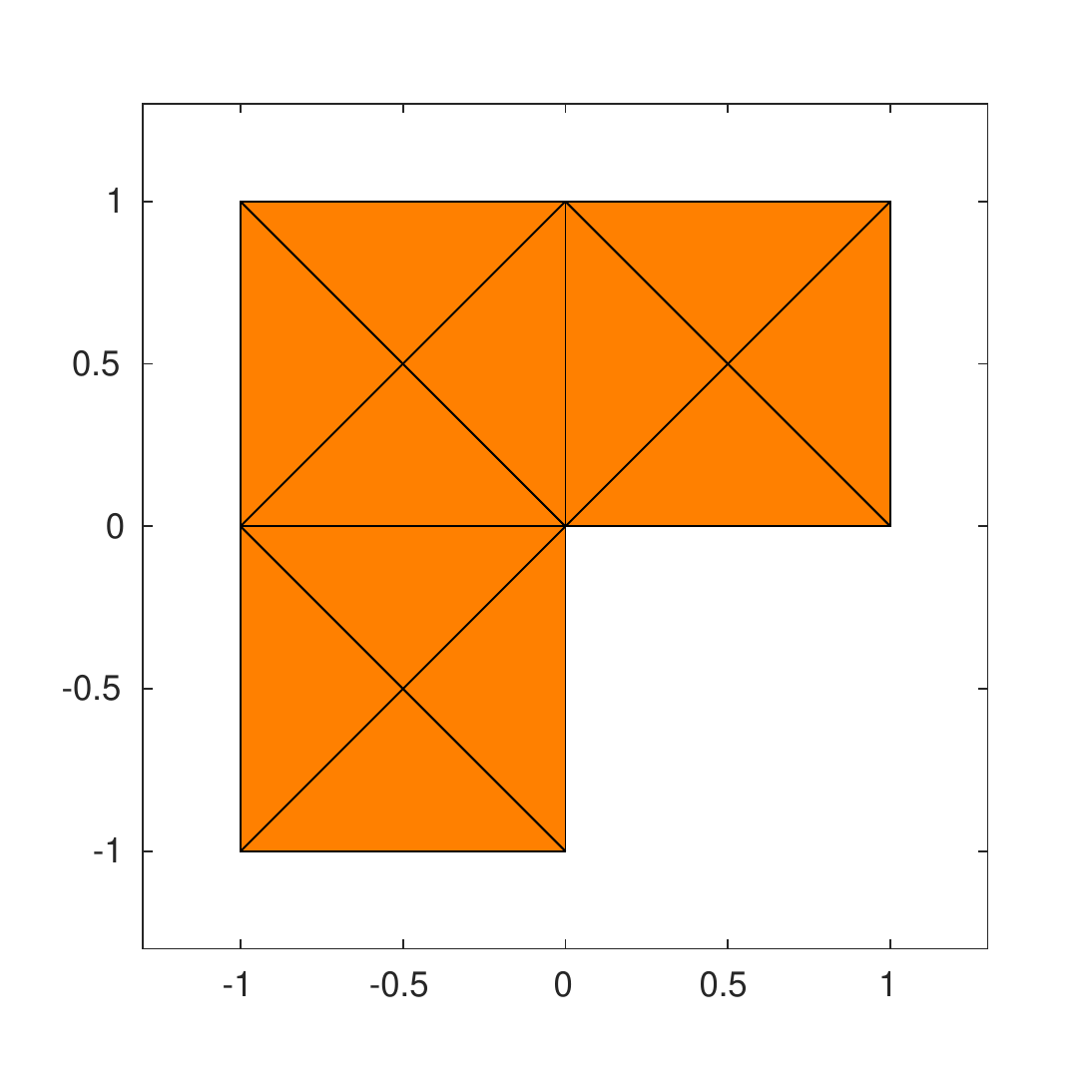}}
  \caption{$Z$-shaped domain $\Omega\subset\R^2$ with initial mesh $\TT_0$  and $\Gamma_D$ marked by a thick blue line (left) and $L$-shaped domain $\Omega\subset\R^2$ with initial mesh $\TT_0$ (right).}
  \label{fig:geometries}
\end{figure}

\subsection{Experiment with known solution}\label{section:example_known}
We consider the $Z$-shaped domain $\Omega\subset\R^2$ from Figure~\ref{fig:geometries} (left) with mixed boundary conditions and the nonlinear problem~\eqref{eq:example} with $\mu(x,|\nabla u^\star(x)|^2):=2+\frac{1}{\sqrt{1+|\nabla u^\star(x)|^2}}$. This leads to the bounds $\alpha=2$ and $L=3$ in~\eqref{def:assumptions_operator}. We prescribe the solution $u^\star$ in polar coordinates $(x,y)=r(\cos\phi,\sin\phi)$ with $\phi\in(-\pi,\pi)$ by
\begin{align}
u^\star(x,y)=r^\beta\cos(\beta\,\phi),
\end{align}
with $\beta=4/7$ and compute $f$ and $g$ in~\eqref{eq:example} accordingly. We note that $u^\star$ has a generic singularity at the re-entrant corner $(x,y)=(0,0)$.

In Figure~\ref{fig:conv_known}, we compare uniform mesh-refinement ($\theta=1$) to adaptive mesh-refinement ($0<\theta<1$) for different values of $\lpcg$ and $\lpic$. We plot the error estimator $\eta_\ell(u_\ell^{\k,\j})$ over the number of elements $N:=\#\TT_\ell$. First (top), we fix $\theta=0.5$, $\lpic=10^{-2}$, and choose $\lpcg\in\{10^{-1},10^{-2},10^{-3},10^{-4}\}$. We see that uniform mesh-refinement leads to the suboptimal rate of convergence $\OO(N^{-2/7})$, whereas Algorithm~\ref{algorithm} with adaptive mesh-refinement regains the optimal rate of convergence $\OO(N^{-1/2})$, independently of the actual choice of $\lpcg$. We observe the very same if we fix $\theta=0.5$, $\lpcg=10^{-2}$, and choose $\lpic\in\{1, 10^{-1},10^{-2},10^{-3},10^{-4}\}$ (middle), or, if we fix $\lpcg=\lpic=10^{-2}$ and vary $\theta\in\{0.1,0.3,0.5,0.7,0.9\}$ (bottom). Since we know from Proposition~\ref{proposition:reliability} and the estimate
\begin{align*}
 \enorm{u_{\ell}^{\k,\star} - u_{\ell}^{\k,\j}}
 		\reff{eq2:pcg:contraction}\lesssim  \enorm{u_{\ell}^{\k,\j} - u_{\ell}^{\k,\j-1}}
 		\reff{eq:pcg:stopping}\lesssim \eta_{\ell}(u_{\ell}^{\k,\j}) + \enorm{u_{\ell}^{\k,\j}-u_{\ell}^{\k-1,\j}}
 		\reff{eq:picard:stopping}\lesssim \eta_{\ell}(u_{\ell}^{\k,\j})
\end{align*}
that $\eta_\ell(u_\ell^{\k,\j})\simeq\Delta_\ell^{\k,\j}$, this empirically underpins Theorem~\ref{theorem:rate}.
\begin{figure}
	\centering
	\hspace*{-7mm}\includegraphics[width=\textwidth]{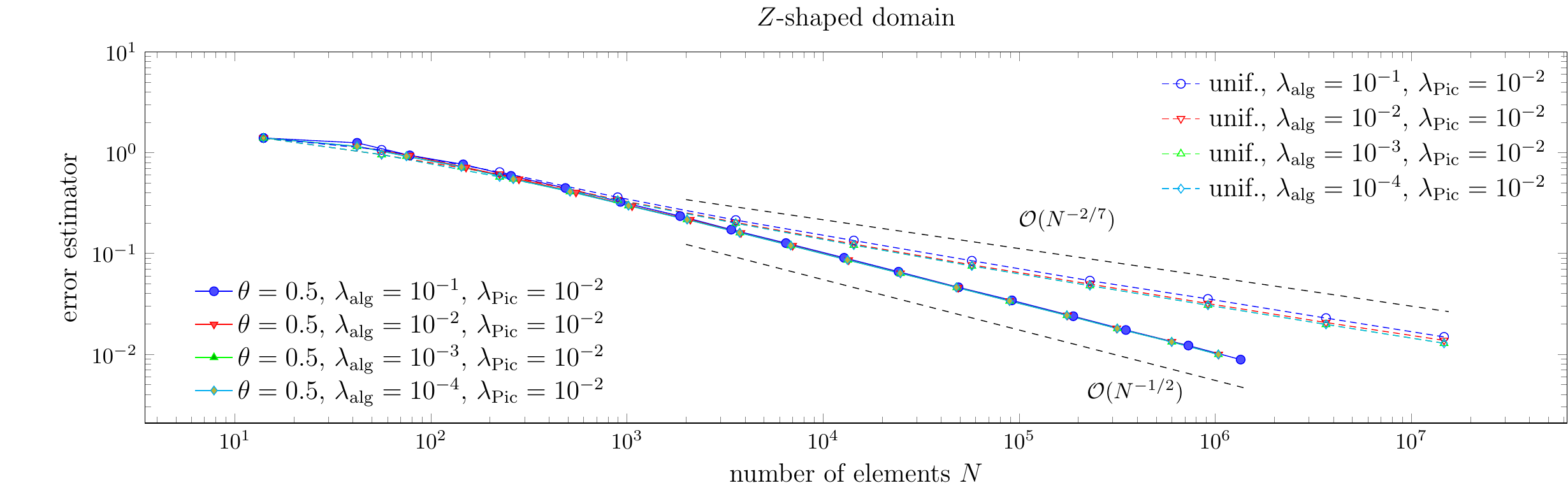}
	\vspace{0.2cm}\\%
	\hspace*{-7mm}\includegraphics[width=\textwidth]{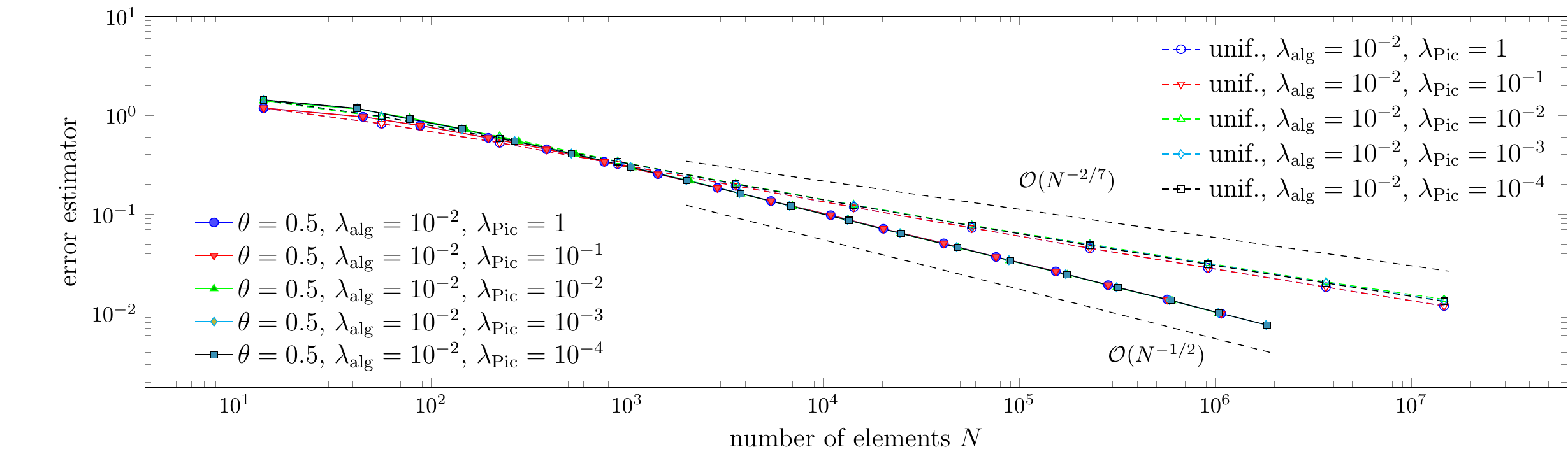}
	\vspace{0.2cm}\\%
	\hspace*{-7mm}\includegraphics[width=\textwidth]{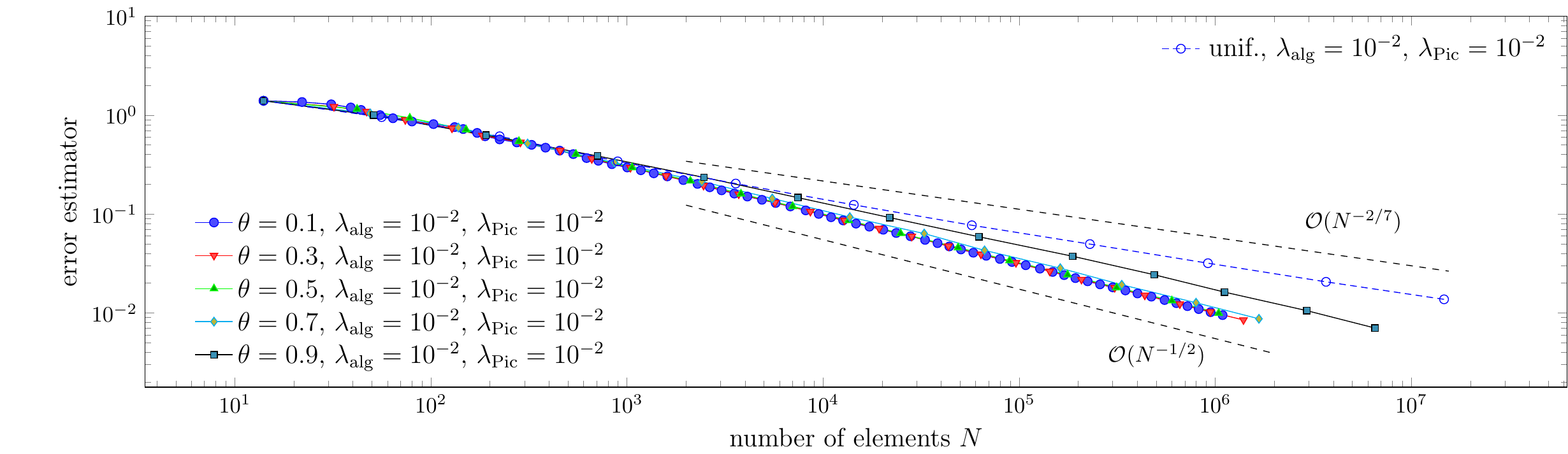}
	\caption{Example from Section~\ref{section:example_known}: Error estimator $\eta_\ell(u_\ell^{\k,\j})$ with respect to the number of elements $N:=\#\TT_\ell$ for $\theta=0.5$, $\lpic=10^{-2}$, and $\lpcg\in\{10^{-1},\ldots,10^{-4}\}$ (top), for $\theta=0.5$, $\lpcg=10^{-2}$, and $\lpic\in\{1,10^{-1},\ldots,10^{-4}\}$ (middle), as well as for $\lpcg=\lpic=10^{-2}$ and $\theta\in\{0.1, 0.3, \ldots, 0.9\}$ (bottom).}
\label{fig:conv_known}
\end{figure}

In Figure~\ref{fig:compl_known}, analogously to Figure~\ref{fig:conv_known}, we choose different combinations of $\theta$, $\lpcg$, and $\lpic$. We plot the error estimator $\eta_{\ell'}(u_{\ell'}^{\k',\j'})$ over the cumulative sum $\sum_{(\ell,k,j)\le(\ell',\k',\j')}\#\TT_{\ell}$. Again, independently of the choice of $\theta\in\{0.1,0.3,0.5,0.7,0.9\}$, $\lpcg \in \{10^{-1},10^{-2},10^{-3},$ $10^{-4}\}$ and $\lpic\in\{1, 10^{-1},$ $10^{-2},10^{-3},10^{-4}\}$, we observe the optimal order of convergence $\OO(N^{-1/2})$ with respect to the computational complexity in accordance with Theorem~\ref{theorem:cost}.

\begin{figure}
	\centering
	\hspace*{-7mm}\includegraphics[width=\textwidth]{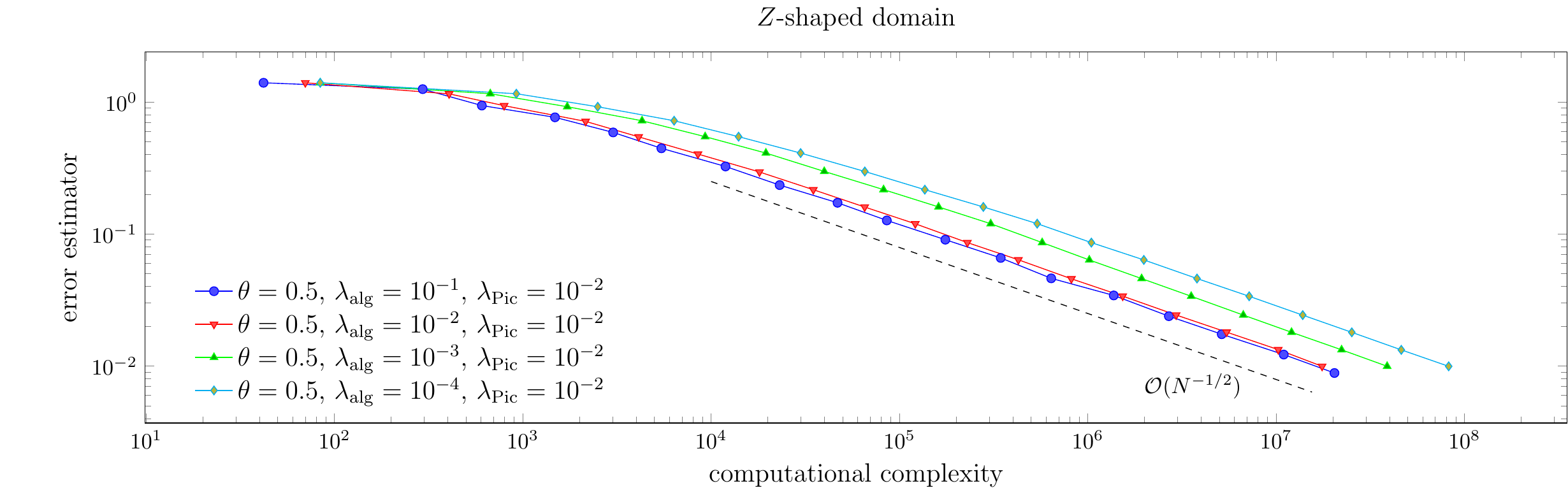}
	\vspace{0.2cm}\\%
	\hspace*{-7mm}\includegraphics[width=\textwidth]{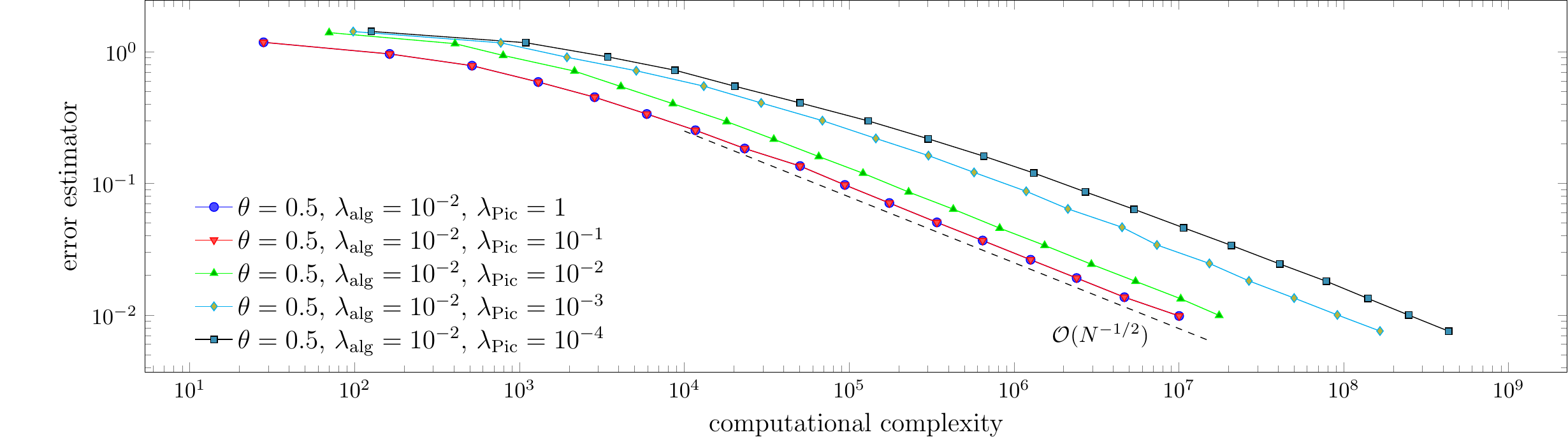}
	\vspace{0.2cm}\\%
	\hspace*{-7mm}\includegraphics[width=\textwidth]{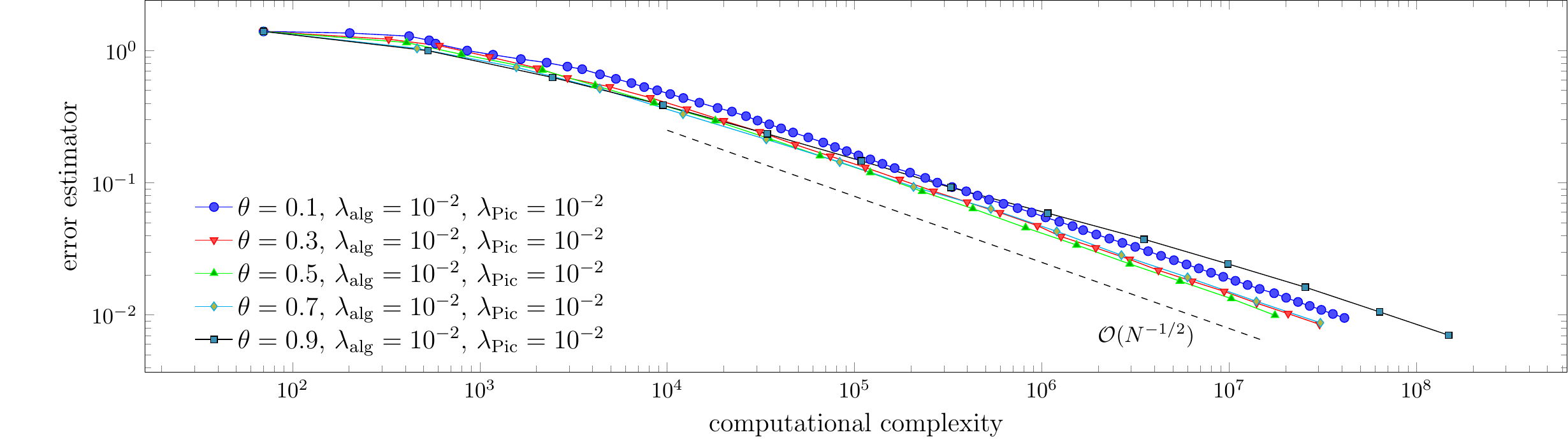}
	\caption{Example from Section~\ref{section:example_known}: Error estimator $\eta_{\ell'}(u_{\ell'}^{\k',\j'})$ with respect to the cumulative sum $\sum_{(\ell,k,j)\le(\ell',\k',\j')}\#\TT_{\ell}$ for $\theta=0.5$, $\lpic=10^{-2}$, and $\lpcg\in \{10^{-1},\ldots,10^{-4}\}$ (top), for $\theta=0.5$, $\lpcg=10^{-2}$, and $\lpic\in\{1,10^{-1},\ldots,10^{-4}\}$ (middle), as well as for $\lpcg=\lpic=10^{-2}$ and $\theta\in\{0.1, 0.3, \ldots, 0.9\}$ (bottom).}
\label{fig:compl_known}
\end{figure}

In Figure~\ref{fig:nIt_known}, we consider the total number of PCG iterations cumulated over all Picard steps on the given mesh for different combinations of $\theta$, $\lpcg$, and $\lpic$. We observe that independently of the choice of these parameters, the total number of PCG iterations stays uniformely bounded. Additionally, we see that for larger values of $\lpcg$ and $\lpic$, as well as for smaller values of $\theta$, the total number of PCG iterations is smaller.

\begin{figure}
	\centering
	\hspace*{-7mm}\includegraphics[width=\textwidth]{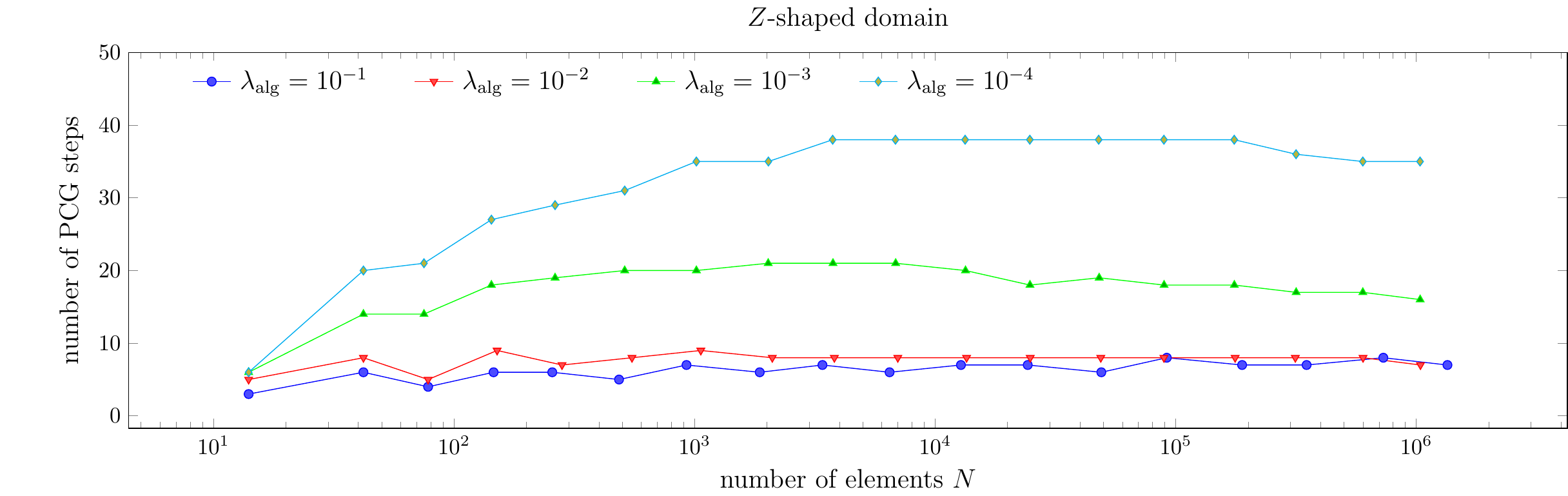}%
	\vspace{0.2cm}\\%
	\hspace*{-7mm}\includegraphics[width=\textwidth]{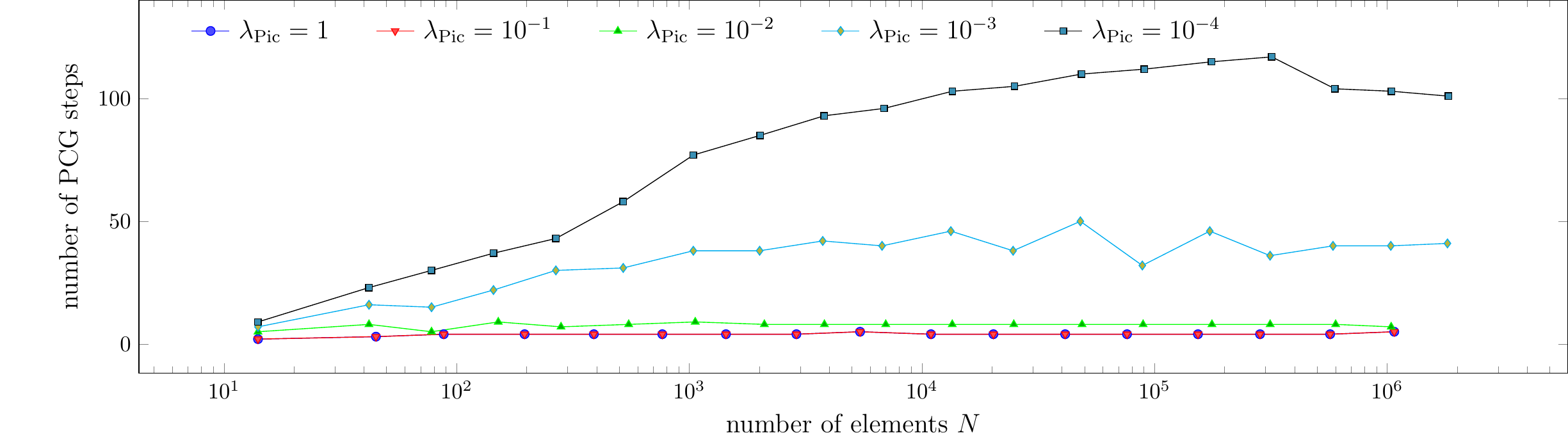}%
	\vspace{0.2cm}\\%
	\hspace*{-7mm}\includegraphics[width=\textwidth]{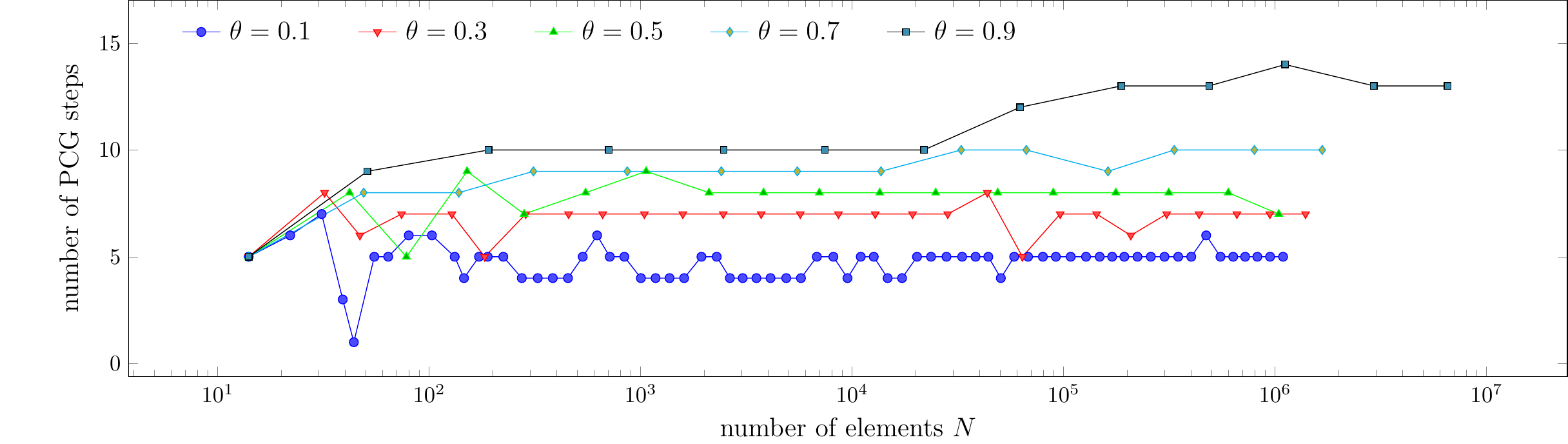}%
	\caption{Example from Section~\ref{section:example_known}: Number of algebraic solver iterations with respect to the number of elements $N:=\#\TT_\ell$ for $\theta=0.5$, $\lpic=10^{-2}$, and $\lpcg\in \{10^{-1},\ldots,10^{-4}\}$ (top), for $\theta=0.5$, $\lpcg=10^{-2}$, and $\lpic\in\{1,10^{-1},\ldots,10^{-4}\}$ (middle), as well as for $\lpcg=\lpic=10^{-2}$ and $\theta\in\{0.1, 0.3, \ldots, 0.9\}$ (bottom).}
\label{fig:nIt_known}
\end{figure}

\subsection{Experiment with unknown solution} \label{section:example_unknown}
We consider the $L$-shaped domain $\Omega\subset\R^2$ from Figure~\ref{fig:geometries} (right) and the nonlinear problem~\eqref{eq:example} with $f(x)=1$ and $\mu(x,|\nabla u^\star(x)|^2)$ $:=1+\frac{\ln(1+|\nabla u^\star|^2)}{1+|\nabla u^\star|^2}$. Then,~\eqref{item:mu_bounded}--\eqref{item:mu_lipschitz1} hold with $\alpha\approx0.9582898$ and $L\approx 1.5423438$.

In Figure~\ref{fig:conv_unknown}, we compare Algorithm~\ref{algorithm} for different values of $\theta$, $\lpcg$, and $\lpic$. As in Section~\ref{section:example_known}, we vary $\theta\in\{0.1, 0.3,0.5,0.7,0.9,1\}$, $\lpcg \in \{10^{-1},$ $10^{-2},10^{-3},10^{-4}\}$, and $\lpic\in\{1, 10^{-1},$ $10^{-2},10^{-3},10^{-4}\}$. We plot the error estimator $\eta_\ell(u_\ell^{\k,\j})$ over the the number of elements $N:=\#\TT_\ell$. Uniform mesh-refinement leads to the suboptimal rate of convergence $\OO(N^{-1/3})$, whereas Algorithm~\ref{algorithm} with adaptive mesh-refinement regains the optimal rate of convergence $\OO(N^{-1/2})$. Again, this empirically confirms Theorem~\ref{theorem:rate}. The latter rate of convergence appears to be even robust with respect to $\theta$, $\lpcg$, and $\lpic$.
\begin{figure}
	\centering
	\hspace*{-7mm}\includegraphics[width=\textwidth]{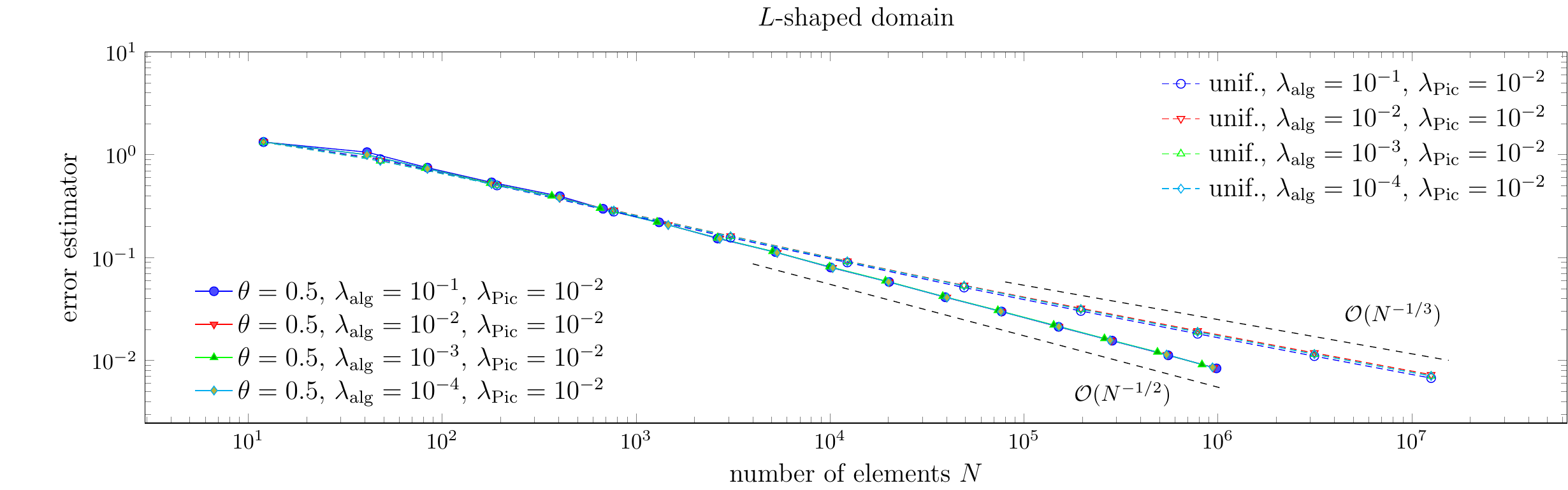}
	\vspace{0.2cm}\\%
	\hspace*{-7mm}\includegraphics[width=\textwidth]{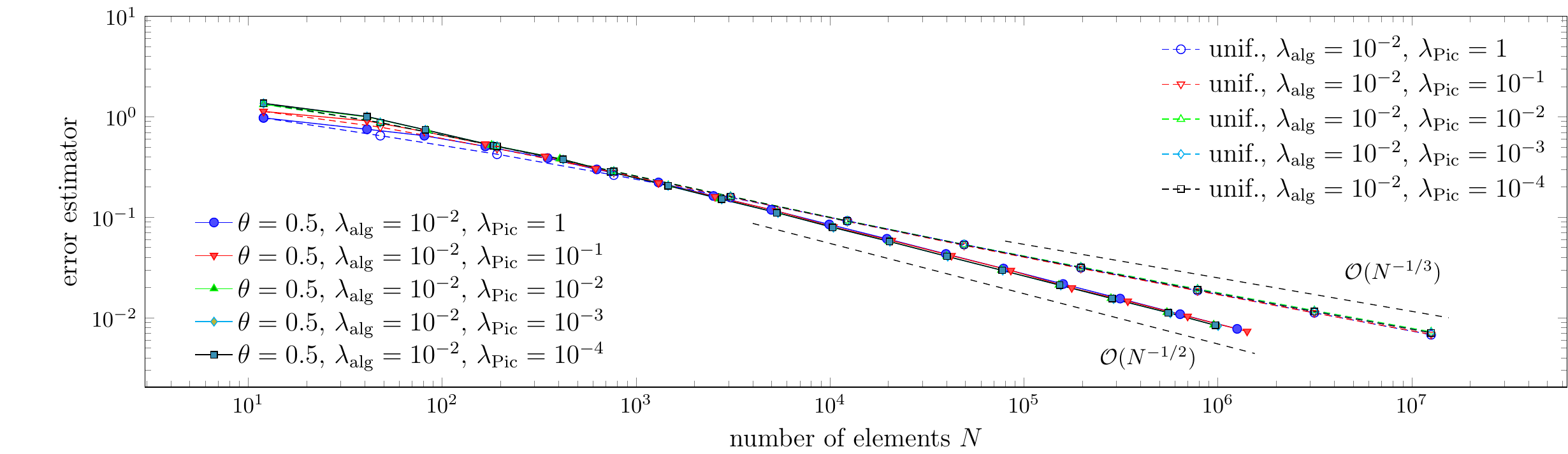}
	\vspace{0.2cm}\\%
	\hspace*{-7mm}\includegraphics[width=\textwidth]{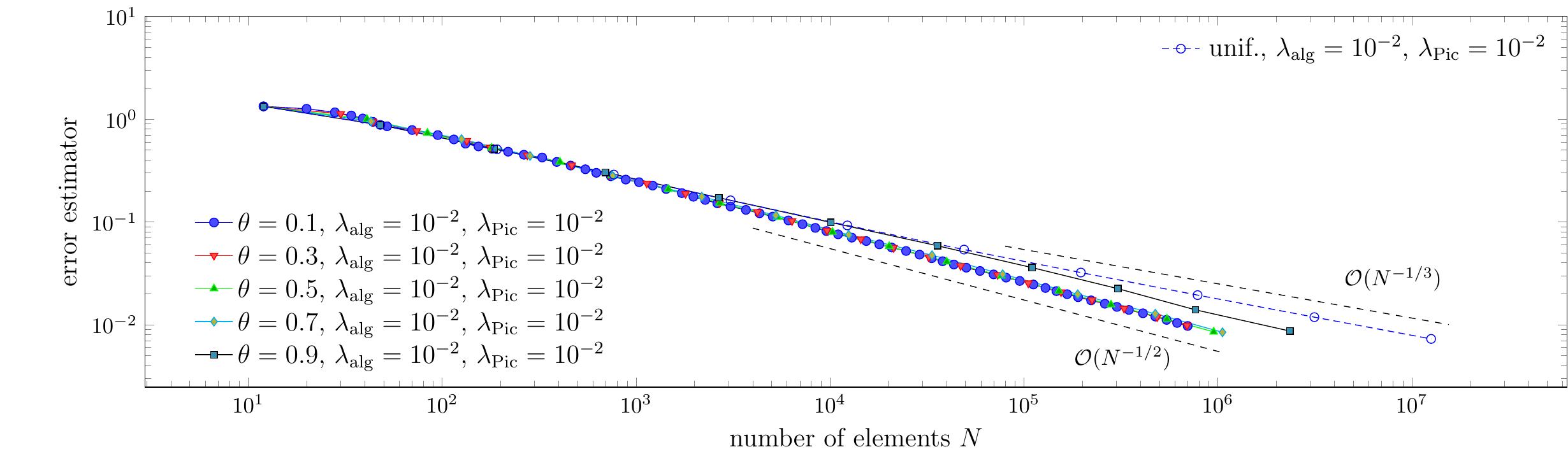}
	\caption{Example from Section~\ref{section:example_unknown}: Error estimator $\eta_\ell(u_\ell^{\k,\j})$ with respect to the number of elements $N:=\#\TT_\ell$ for $\theta=0.5$, $\lpic=10^{-2}$, and $\lpcg\in\{10^{-1},\ldots,10^{-4}\}$ (top), for $\theta=0.5$, $\lpcg=10^{-2}$, and $\lpic\in\{1,10^{-1},\ldots,10^{-4}\}$ (middle), as well as for $\lpcg=\lpic=10^{-2}$ and $\theta\in\{0.1, 0.3, \ldots, 0.9\}$ (bottom).}
\label{fig:conv_unknown}
\end{figure}		

In Figure~\ref{fig:compl_unknown}, we again choose different combinations of $\theta$, $\lpcg$, and $\lpic$. We plot the error estimator $\eta_{\ell'}(u_{\ell'}^{\k',\j'})$ over the cumulative sum $\sum_{(\ell,k,j)\le(\ell',\k',\j')}\#\TT_{\ell}$. Independently of the choice of the parameters, we observe the optimal order of convergence $\OO(N^{-1/2})$ with respect to the computational complexity which empirically underpins Theorem~\ref{theorem:cost}.

\begin{figure}
	\centering
	\hspace*{-7mm}\includegraphics[width=\textwidth]{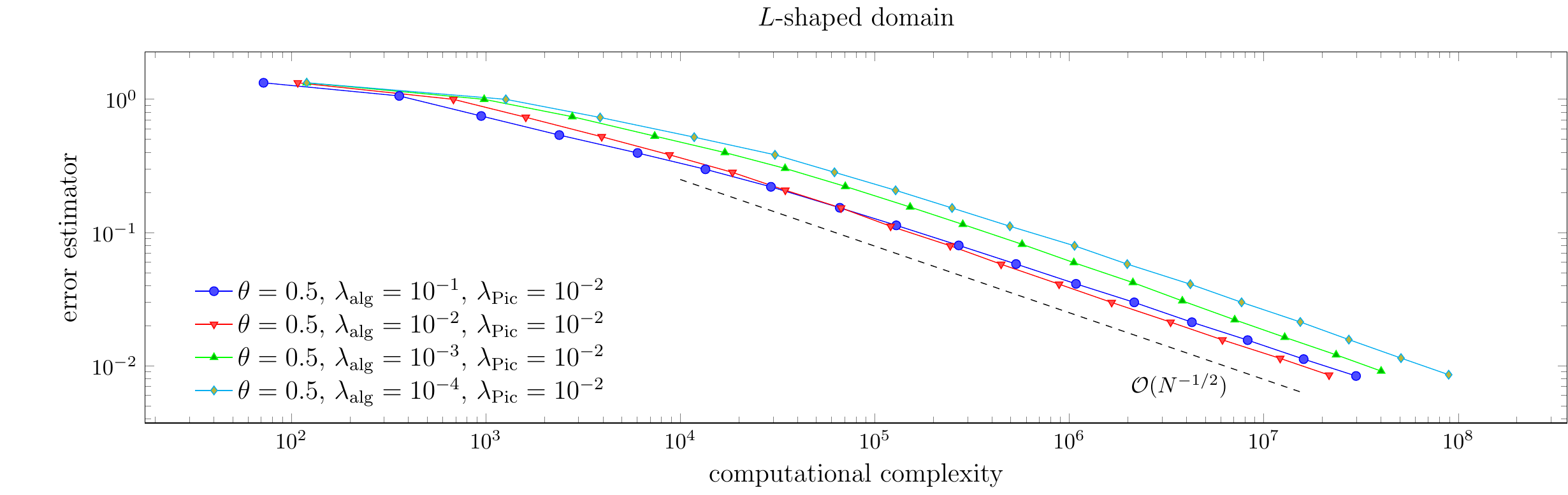}
	\vspace{0.2cm}\\%
	\hspace*{-7mm}\includegraphics[width=\textwidth]{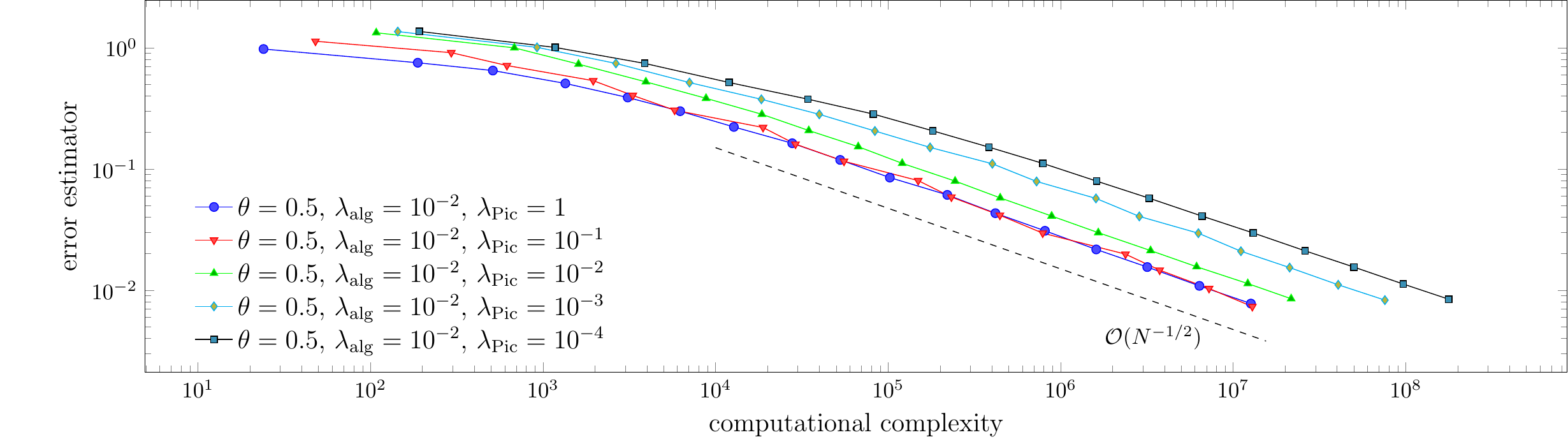}
	\vspace{0.2cm}\\%
	\hspace*{-7mm}\includegraphics[width=\textwidth]{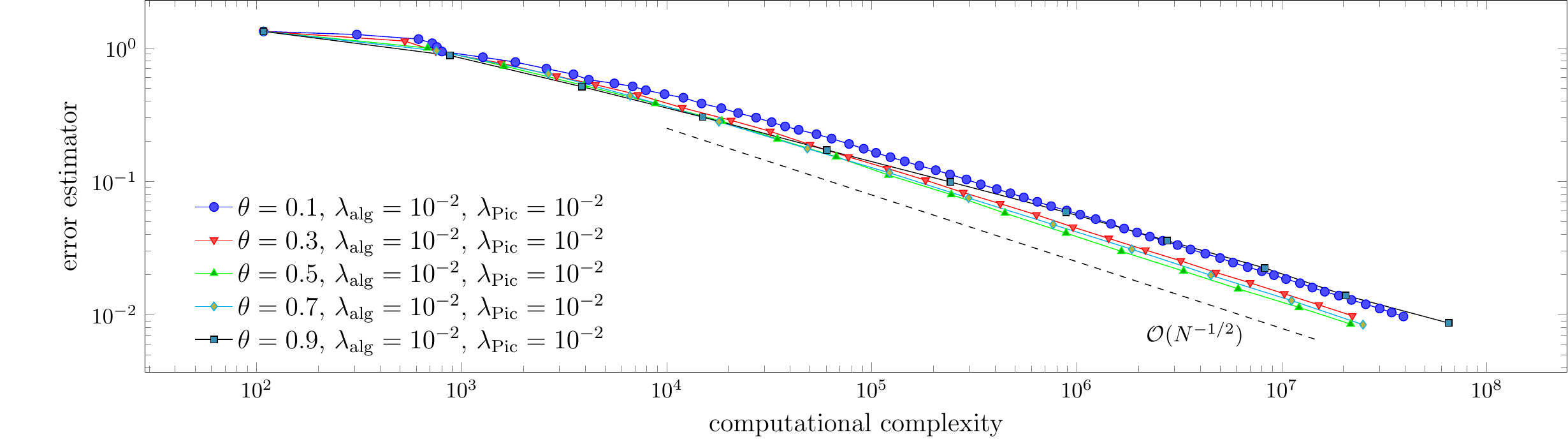}
	\caption{Example from Section~\ref{section:example_unknown}: Error estimator $\eta_{\ell'}(u_{\ell'}^{\k',\j'})$ with respect to the cumulative sum $\sum_{(\ell,k,j)\le(\ell',\k',\j')}\#\TT_{\ell}$ for $\theta=0.5$, $\lpic=10^{-2}$, and $\lpcg\in \{10^{-1},\ldots,10^{-4}\}$ (top), for $\theta=0.5$, $\lpcg=10^{-2}$, and $\lpic\in\{1,10^{-1},\ldots,10^{-4}\}$ (middle), as well as for $\lpcg=\lpic=10^{-2}$ and $\theta\in\{0.1, 0.3, \ldots, 0.9\}$ (bottom).}
\label{fig:compl_unknown}
\end{figure}	

In Figure~\ref{fig:nIt_unknown}, we consider the total number of PCG iterations cumulated over all Picard steps on the given mesh. We observe that independently of the choice $\theta$, $\lpcg$, and $\lpic$, the total number of PCG iterations stays uniformely bounded. Additionally, we see that for larger values of $\lpcg$ and $\lpic$, as well as for smaller values of $\theta$, the total number of PCG iterations is smaller.

\begin{figure}
	\centering
	\hspace*{-7mm}\includegraphics[width=\textwidth]{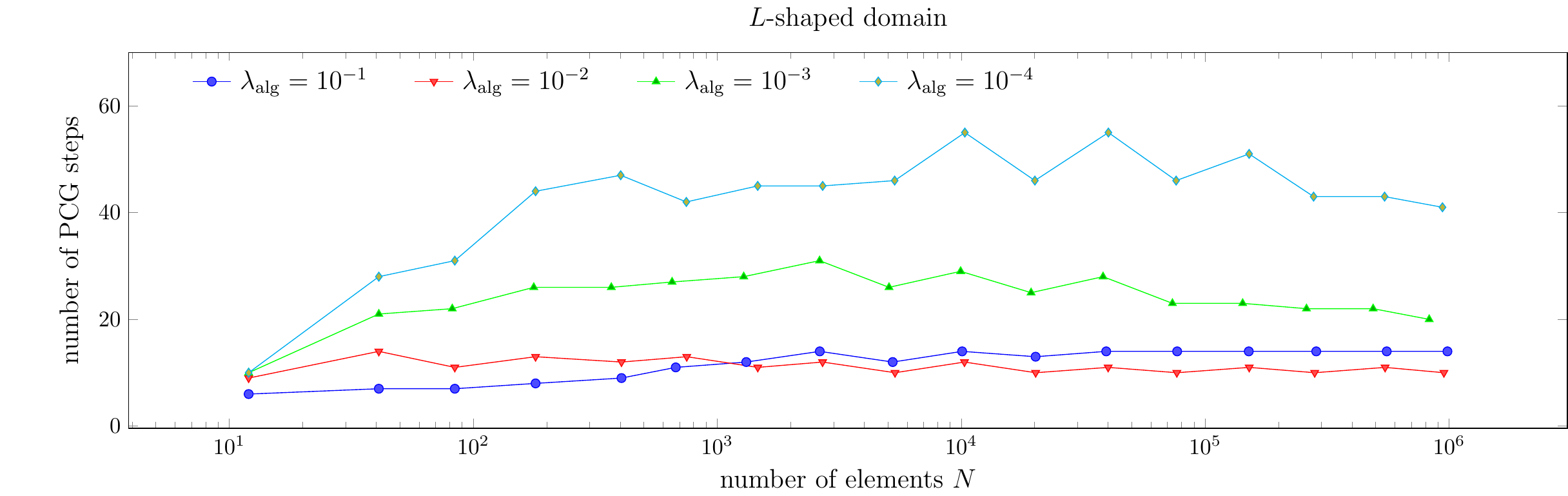}%
	\vspace{0.2cm}\\%
	\hspace*{-7mm}\includegraphics[width=\textwidth]{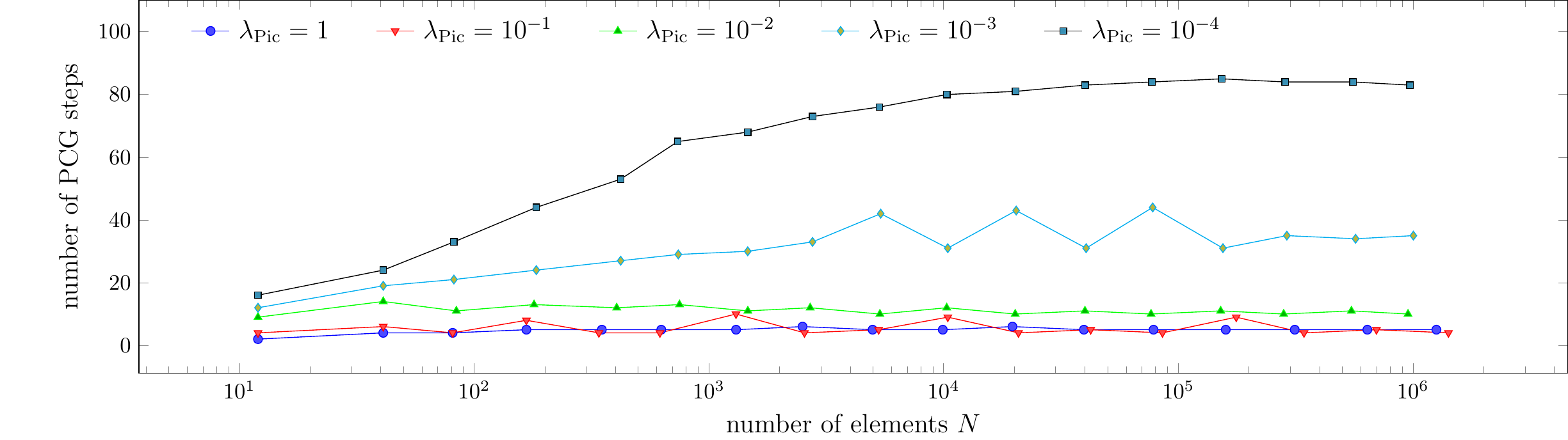}%
	\vspace{0.2cm}\\%
	\hspace*{-7mm}\includegraphics[width=\textwidth]{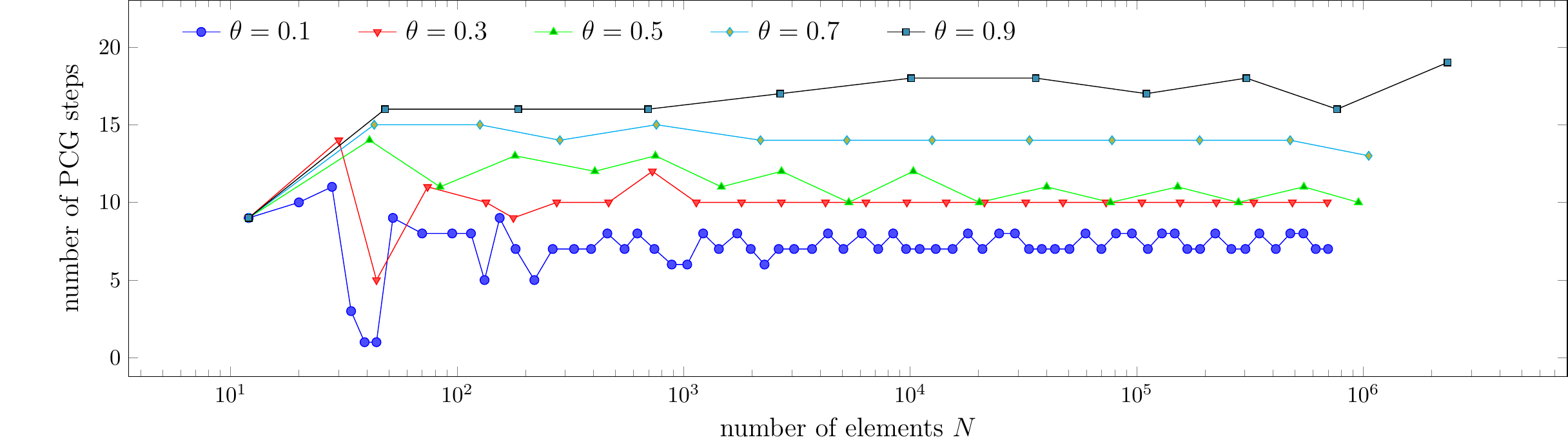}%
	\caption{Example from Section~\ref{section:example_unknown}: Number of algebraic solver iterations with respect to the number of elements $N:=\#\TT_\ell$ for $\theta=0.5$, $\lpic=10^{-2}$, and $\lpcg\in \{10^{-1},\ldots,10^{-4}\}$ (top), for $\theta=0.5$, $\lpcg=10^{-2}$, and $\lpic\in\{1,10^{-1},\ldots,10^{-4}\}$ (middle), as well as for $\lpcg=\lpic=10^{-2}$ and $\theta\in\{0.1, 0.3, \ldots, 0.9\}$ (bottom).}
\label{fig:nIt_unknown}
\end{figure}	


\bibliographystyle{alpha}
\bibliography{literature}

\end{document}